\documentclass[11pt]{amsart}
\usepackage[utf8]{inputenc}

\usepackage{hyperref}

\usepackage{fullpage}
\usepackage[utf8]{inputenc}

\usepackage[T1]{fontenc}
\usepackage{graphicx,subfigure} 
\usepackage{amsmath,amsfonts,amssymb}
\usepackage{mathrsfs,dsfont}
\usepackage{amsthm,,mathtools}
\usepackage{mathabx}
\usepackage{caption}
\usepackage{tabularx}
\usepackage{enumitem} 
\usepackage{float}
\usepackage[usenames]{color}
\allowdisplaybreaks
\usepackage[dvipsnames]{xcolor}
\usepackage{comment}

\newtheorem{theo}{Theorem}
\newtheorem{lemma}{Lemma}[section]
\newtheorem{ass}{Assumption}[section]
\newtheorem{rem}{Remark}[section]
\newtheorem{propo}{Proposition}[section]
\newtheorem{exam}{Example}[section]

\newtheorem{fact}{Fact}

\newcommand{\N}{\mathbb{N}}
\newcommand{\R}{\mathbb{R}}
\newcommand{\E}{\mathbb{E}}

\newcommand{\X}{\mathcal{X}}
\newcommand{\Y}{\mathcal{Y}}
\newcommand{\LipF}{L_F}
\newcommand{\Lips}{L_{\sigma}}

\title{Error analysis for stochastic gradient optimization schemes using modified equations}

\author[C.-E. Br\'ehier]{Charles-Edouard Br\'ehier}
\address{Universite de Pau et des Pays de l'Adour, E2S UPPA, CNRS, LMAP, Pau, France}
\email{charles-edouard.brehier@univ-pau.fr}

\author[M. Dambrine]{Marc Dambrine}
\address{Universite de Pau et des Pays de l'Adour, E2S UPPA, CNRS, LMAP, Pau, France}
\email{marc.dambrine@univ-pau.fr}

\author[N. En--Nebbazi]{Nassim En--Nebbazi}
\address{Universite de Pau et des Pays de l'Adour, E2S UPPA, CNRS, LMAP, Pau, France}
\email{nenebbazi@univ-pau.fr}

\subjclass{60H35;65C30;65K05;65K10}

\keywords{stochastic optimization schemes; modified equations; weak and strong error analysis}

\date{}

\begin{document}

\maketitle

\begin{abstract}
We consider a class of stochastic gradient optimization schemes. Assuming that the objective function is strongly convex, we prove weak error estimates which are uniform in time for the error between the solution of the numerical scheme, and the solutions of continuous-time modified (or high-resolution) differential equations at first and second orders, with respect to the time-step size. At first order, the modified equation is deterministic, whereas at second order the modified equation is stochastic and depends on a modified objective function. We go beyond existing results where the error estimates have been considered only on finite time intervals and were not uniform in time. This allows us to then provide a rigorous complexity analysis of the method in the large time and small time-step size regimes. We provide numerical experiments to illustrate the convergence results.
\end{abstract}

\section{Introduction}\label{sec:intro}

The analysis of numerical optimization methods, whether deterministic or stochastic, has been the subject of intense research activity. Those methods are widely applied in many situations, in particular recently in machine learning. In order to approximate minimizers of an objective function $F:\R^d\to\R$, the simplest method is the gradient descent algorithm
\[
x_{n+1}=x_n-h\nabla F(x_n),
\]
with a constant time-step size $h$. To study the asymptotic behavior of $x_N$ when $N\to\infty$, it is fruitful to make a connection with the continuous-time gradient dynamics
\[
x'(t)=-\nabla F(x(t)).
\]
In fact, the gradient descent algorithm can be interpreted as the standard explicit Euler scheme applied to the ordinary differential equation above with time-step size $h$.

Connections between (discrete-time) optimization schemes and continuous-time dynamics have been studied in the literature, in deterministic and stochastic settings, not only for gradient algorithms but also for other methods such as the Polyak's heavy-ball method \cite{polyak1964} and Nesterov's acceleration method \cite{nesterov1983}. The latter methods intend to accelerate the convergence to minimizers by introducing inertia in the scheme, and as a result the related continuous-time dynamics are second-order differential equations. In the last years, the celebrated work \cite{SuBoydCandes} on Nesterov's acceleration method has driven the attention of many researchers on that connection. Instead of considering leading order approximations, better understanding on the discrete-time dynamics can be obtained by the identification of more accurate continuous-time dynamics, which depend on the time-step size. In the numerical analysis community for ordinary and stochastic differential equations, such more accurate approximations are referred to as modified equations, see for instance~\cite{DHZ,GGK2,LiTaiE,Zygalakis}. In the optimization community, they are referred to as high-resolution equations, see for instance~\cite{DDPR,ShiDu,Shi2019}. In this article, we use the terminology modified equations.

The connection between the discrete and continuous time dynamics motivates the decomposition of the analysis of the asymptotic behavior of optimization schemes into two parts: the analysis of the asymptotic behavior of the continuous-time system, and the analysis of the approximation error between the discrete and continuous time systems. On the one hand, differential and stochastic calculus techniques can be applied to show convergence of the continuous-time dynamics when time $T$ goes to infinity, by the identification of so-called Lyapunov functionals. For instance, if the objective function $F$ is strongly convex, the square of the norm is a Lyapunov function, and applying Gr\"onwall's inequality it is straightforward to show that the solution of the continuous-time gradient dynamics converges exponentially fast to the unique minimizer of $F$. More complex Lyapunov functions dedicated to objective functions satisfying weaker conditions have been constructed in the literature. On the other hand, when standard numerical analysis techniques are applied to study the approximation error between the discrete and continuous time systems, usually the error bounds with respect to the time-step size $h$ depend on the final time $T=Nh$. The dependence obtained when discrete Gr\"onwall inequalities are employed may be exponential. Instead, it is desirable to obtain uniform in time error estimates, i.e. with upper bounds independent of the final time $T$. Combining the two types of error bounds then enables to analyze the complexity of the algorithm, i.e. to determine how to choose the parameters $h$ and $N$ of the algorithm to ensure that the error is below any arbitrary threshold.

In this article, we analyze the behavior of the following stochastic gradient optimization scheme
\begin{equation*}
\left\lbrace
\begin{aligned}
&X_{n+1}=X_n-h\nabla F(X_n)+h\sigma(t_n,X_n)\gamma_{n+1},\quad \forall~n\in\N,\\
&X_0=x_0,
\end{aligned}
\right.
\end{equation*}
where $F:\R^d\to\R$ is the objective function. Compared with the deterministic version above, the additional term $h\sigma(t_n,X_n)\gamma_{n+1}$ describes random perturbations, where $\bigl(\gamma_n\bigr)_{n\ge 1}$ are independent standard $\R^d$-valued Gaussian random variables. We assume that the objective function $F$ is strongly convex, it thus admits a unique minimizer denoted by $x^\star$. Our objective is to analyze the convergence of $X_N$ to $x^\star$ when $N\to\infty$ and $h\to 0$. The stochastic gradient optimization scheme above has been considered previously in the literature, see for instance~\cite{DDPR} for the proof of strong error estimates and in~\cite{LiTaiE} for the proof of weak error estimates. Let us recall that when dealing with approximation of stochastic systems and associated error bounds, one may consider either strong approximation and error bounds, concerned with $\E[\|X_N-x^\star\|]$ or with $\bigl(\E[\|X_N-x^\star\|^2]\bigr)^{1/2}$, or weak approximation and error bounds, concerned with $|\E[\varphi(X_N)]-\varphi(x^\star)|$. We refer for instance to the reference books \cite{kloeden1992,lord2014} on numerical methods for stochastic (partial) differential equations.

The objective of this article is to fill an important gap in the analysis of the stochastic gradient optimization scheme above: our main contribution is to establish uniform in time weak error estimates between the solution $X_N$ of the stochastic optimization scheme and the solution at time $T=Nh$ of associated modified equations which are either deterministic and stochastic. This is a major improvement compared with previous works such as~\cite{DDPR,LiTaiE} where the error bounds are not uniform with respect to time $T$ and are thus not relevant for the analysis of the algorithm in the large time regime. As explained above, combined with results on the large time behavior of the continuous-time associated deterministic and stochastic dynamics, we are able to establish complexity results for the stochastic gradient optimization scheme. We consider two modified equations, which provide first and second order approximation with respect to $h$.

For the first order approximation, we exploit the connection between the stochastic gradient optimization scheme and the deterministic gradient system, i.e. the ordinary differential equation
\begin{equation*}
\frac{d\X^0(t)}{dt}=-\nabla F(\X^0(t)),\quad \forall~t\ge 0.
\end{equation*}
As already recalled above, assuming that $F$ is strongly convex, $\X^0(T)$ converges to $x^\star$ exponentially fast when $T\to\infty$. Our first main result is the proof of weak error estimates
\[
|\E[\varphi(X_N)]-\varphi(x^\star)|={\rm O}\bigl(e^{-\mu Nh}+h\bigr),
\]
see Theorem~\ref{theo:1} for the precise statement. This follows from uniform in time weak error estimates
\[
\underset{N\in\N}\sup~|\E[\varphi(X_N)]-\varphi(\X^0(Nh))|={\rm O}(h).
\]
Moreover, applying the weak error estimates with $\varphi=F$ and noting that the residual $F(X_N)-F(x^\star)$ dominates $\|X_N-x^\star\|^2$, we also obtain strong error estimates
\[
\bigl(\E[\|X_N-x^\star\|^2]\bigr)^{\frac12}={\rm O}\bigl(e^{-\mu Nh}+h^{\frac12}\bigr).
\]
This is not a standard result in stochastic numerics, since usually weak error bounds may deduced from strong error bounds. To prove weak error estimates, which are uniform in time, we revisit the arguments usually employed to study approximations of invariant distributions for stochastic differential equations, in particular we need to study the regularity properties of an auxiliary mapping $u^0$ (defined by Equation~\eqref{eq:u}), which is the solution to the (Kolmogorov) partial differential equation~\eqref{eq:Kolmogorov_u} associated with the deterministic gradient system. Precisely we prove the exponential decay with respect to $t$ for the first and second order spatial partial derivatives $\nabla_xu^0(t,x)$ and $\nabla_x^2u^0(t,x)$, see Lemma~\ref{lem:u}.

The considered stochastic gradient optimization scheme can be interpreted as the Euler--Maruyama scheme applied to the stochastic differential equation
\[
d\X^h(t)=-\nabla F(\X^h(t))dt+\sqrt{h}\sigma(t,\X^h(t))dB(t),\quad \forall~t\ge 0,
\]
which can be interpreted as a modified equation since it depends on the time-step size $h$. However that simple connection is not sufficient to obtain higher order approximation with respect to $h$. Instead, one needs to consider the following modified stochastic differential equation
\begin{equation*}
d\Y^h(t)=-\nabla F^h(\Y^h(t))dt+\sqrt{h}\sigma(t,\Y^h(t))dB(t),\quad \forall~t\ge 0,
\end{equation*}
which has the same diffusion coefficient $\sqrt{h}\sigma$ but depends on the modified objective function \[
F^h=F+\frac{h}{4}\|\nabla F\|^2.
\]
The modified stochastic differential equation and objective function above have been identified in the recent articles~\cite{DDPR,LiTaiE}. We refer to~\cite{DHZ,MR4084147,GGK2,MR4267491} for other examples of applications of modified/high-resolution equations in the context of deterministic and stochastic optimization. In this article, we prove uniform in time weak error estimates
\[
\underset{N\in\N}\sup~|\E[\varphi(X_N)]-\E[\varphi(\Y^h(Nh))]|={\rm O}(h^2),
\]
under stronger conditions on the modified objective function $F$: we assume that the modified objective function $F^h$ is strongly convex uniformly with respect to $h$ (in an interval $(0,h_0)$). We refer to Assumption~\ref{ass:F2} for a rigorous statement. To the best of our knowledge, this condition has not been exhibited before in the literature. That assumption plays a crucial role in the analysis. First, it allows us to prove the exponential decay with respect to $t$ of spatial derivatives of the solutions $v_T^h$ (defined by Equation~\eqref{eq:vh}) to associated (backward) Kolmogorov equations~\eqref{eq:Kolmogorov_vh}, uniformly with respect to $h$. In addition, it allows us to prove error bounds for $\E[\|\Y^h(T)-x^\star\|^2]$ in the large time regime.

Our second main result is thus the proof of weak error estimates
\begin{equation*}
\big|\E[\varphi(X_N)]-\varphi(x^\star)\big|={\rm O}\left(e^{-\mu Nh}+\sqrt{h} \sqrt{e^{-2\mu Nh}\rho(Nh)}+h^2\right),
\end{equation*}
where the mapping $\rho$ is defined by~\eqref{eq:rhoT} and depends on the diffusion coefficient $\sigma$. We refer to Theorem~\ref{theo:2} for a precise statement. Like for the first order approximation discussed above, one obtains strong error estimates
\[
\bigl(\E[\|X_N-x^\star\|^2]\bigr)^{\frac12}={\rm O}\left(e^{-\mu Nh}+\sqrt{h} \sqrt{e^{-2\mu Nh}\rho(Nh)}+h\right).
\]
It is worth mentioning that the mapping $T\mapsto e^{-2\mu T}\rho(T)$ is bounded, and thus the strong error bounds obtained using the second order (stochastic) approximation are always at least as good as the ones obtained using the first order (stochastic) approximation. It remains to compare the two results in terms of complexity, depending on the large time behavior of the diffusion coefficient, to see whether considering the second order approximation leads to improved error bounds.

We provide numerical experiments to illustrate the strong error estimates obtained in Theorems~\ref{theo:1} and~\ref{theo:2}. We observe the strong error estimates obtained in the second main result are sharp since the three error terms appearing on the right-hand side are visible in some regimes.

The comparison between the first and second order approximation results in the large time regime is not as straightforward as it may seem. Indeed, to obtain the first order approximation result it is sufficient to assume that some upper bounds on $\sigma(t,\cdot)$ which are uniform with respect to $t$. It turns out that the second order approximation result reduces the complexity only if $\sigma(t,\cdot)$ converges to $0$ when $t$ goes to infinity fast enough. We refer to Section~\ref{sec:results} for a precise discussion on the computational complexity of the stochastic gradient optimization scheme and a comparison of the two results.

Let us mention several directions for future works. It would be interesting to study the behavior of the stochastic gradient optimization scheme under weaker conditions on the objective function $F$. One may also investigate the behavior of other stochastic optimization schemes, such as stochastic versions of the celebrated heavy ball and Nesterov algorithms. The analysis may also be extended in an infinite dimensional setting.

The article is organized as follows. Section~\ref{sec:setting} is devoted to the description of the setting and the statement of the fundamental assumptions on the objective function $F$ and the diffusion coefficient $\sigma$. The main results of this article, Theorem~\ref{theo:1} for the first order approximation, and Theorem~\ref{theo:2}  for the second order approximation, are stated in Section~\ref{sec:results}, which also contains numerical experiments and a discussion on the computational complexity. Section~\ref{sec:aux} presents some results on the stochastic gradient optimization scheme, including uniform moment bounds given by Lemma~\ref{lem:momentbounds-scheme}. The proof of the first order approximation result, Theorem~\ref{theo:1}, is provided in Section~\ref{sec:error1}, and the proof of the second order approximation result, Theorem~\ref{theo:2}, is provided in Section~\ref{sec:error2}. Appendix~\ref{sec:app} contains the proof of technical but fundamental results, on the regularity properties of the solutions of the Kolmogorov equations associated with the first and second order modified equations.

\section{Setting}\label{sec:setting}

\subsection{Notation}\label{sec:setting_notation}

We denote by $\mathbb{N}^\star=\{1,\ldots\}$ and $\mathbb{N}=\{0,1,\ldots\}$ the sets of positive and nonnegative integers respectively. Given $d\in\mathbb{N}^\star$ be an arbitrary positive integer, the Euclidean space $\R^d$ is equipped with the standard inner product and norm denoted by $\langle \cdot,\cdot \rangle$ and $\|\cdot\|$ respectively. Let $\mathcal{M}_d(\R)$ denote the space of real-valued square matrices with $d$ rows and columns. The identity matrix is denoted by $I_d$. For any matrix $M\in\mathcal{M}_d(\R)$, $M^\star$ is the transpose of $M$ and ${\rm Tr}(M)$ is the trace of $M$. The space $\mathcal{M}_d(\R)$ is equipped with the norm $\|\cdot\|$ defined by
\[
\|M\|=\bigl(\sum_{1\le i,j\le d}|M_{i,j}|^2\bigr)^{\frac12},\quad \forall~M\in\mathcal{M}_d(\R).
\]
If $M,N\in\mathcal{M}_d(\R)$ are two matrices, let
\[
M:N=\sum_{1\le i,j\le d}M_{i,j}N_{i,j}={\rm Tr}\bigl(MN^\star\bigr).
\]

If $\phi:\R^d\to\R$ is a real-valued function of class $\mathcal{C}^1$, then $\nabla \phi(x)=\bigl(\partial_{x_i}\phi(x)\bigr)_{1\le i\le d}\in\R^d$ denotes the gradient of $\phi$ at $x\in\R^d$. Similarly, if $\phi:\R^d\to\R$ is a function of class $\mathcal{C}^2$, then $\nabla^2 \phi(x)=\bigl(\partial_{x_i}\partial_{x_j}\phi(x)\bigr)_{1\le i,j\le d}\in\mathcal{M}_d(\R)$ denotes the Hessian matrix of $\phi$ at $x\in\R^d$.

Given $d_1,d_2,m\in\N^\star$, if $\phi:\R^{d_1}\to\R^{d_2}$ is a mapping of class $\mathcal{C}^m$, then $D^m\phi(x).(k_1,\ldots,k_m)\in\R^{d_2}$ denotes the differential of order $m$ of $\phi$ at $x\in\R^{d_1}$ in the directions $k_1,\ldots,k_m\in\R^{d_1}$. In the case $d_1=d$ and $d_2=1$, if $\varphi:\R^d\to\R$ is of class $\mathcal{C}^2$, with the notation introduced above one has
\[
D\varphi(x).(k)=\langle \nabla\varphi(x),k\rangle,\quad D^2\varphi(x).(k_1,k_2)=\langle \nabla^2\varphi(x)k_1,k_2\rangle. 
\]
For any $m\in\N^\star$, if $\varphi:\R^d\to\R$ is of class $\mathcal{C}^m$, then set
\[
\vvvert\varphi\vvvert_m=\underset{x\in\R^d}\sup~\underset{k_1,\ldots,k_m\in\R^d\setminus\{0\}}\sup~\frac{|D^m\varphi(x).(k_1,\ldots,k_m)|}{\|k_1\|\ldots\|k_m\|}.
\]
In the sequel we consider functions $\varphi$ of class $\mathcal{C}^2$ with bounded second order derivative. Note that its first order derivative of such functions satisfies the following upper bound: for all $x\in\R^d$ and $k\in\R^d$ one has
\begin{equation}\label{eq:ineqDvarphi}
|D\varphi(x).k|\le \Bigl(\|\nabla \varphi(x^\star)\|+\vvvert\varphi\vvvert_2\|x-x^\star\|\Bigr)\|k\|.
\end{equation}

Finally, let $\bigl(B(t)\bigr)_{t\ge 0}$ be a standard $\R^d$-valued Wiener process defined on a probability space denoted by $(\Omega,\mathcal{F},\mathbb{P})$, and which is adapted to a filtration $\bigl(\mathcal{F}_t\bigr)_{t\ge 0}$ which satisfies the usual conditions. If $X$ is an integrable real-valued random variable, its expectation is denoted by $\E[X]$. In addition, given a $\sigma$-field $\mathcal{G}\subset\mathcal{F}$, $\E[X|\mathcal{G}]$ denotes the conditional expectation of $X$ given $\mathcal{G}$.

In the analysis, the value of (deterministic) positive real numbers $C\in(0,\infty)$ may vary from line to line. Unless specified, the value of $C$ is uniform with respect to variables such as time $t\in\R^+=[0,\infty)$, the state $x\in\R^d$ or the time-step size $h$.

\subsection{The considered stochastic gradient optimization scheme}\label{sec:setting_scheme}

Let $x_0\in\R^d$ be an arbitrary initial value assumed to be deterministic. Considering $\mathcal{F}_0$-measurable random initial values satisfying appropriate integrability properties could be considered with straightforward minor modifications of the statements.

Let $h>0$ denote the time-step size of the considered stochastic gradient optimization scheme. Without loss of generality, it is assumed that $h\in(0,h_{\rm max})$ for some positive real number $h_{\rm max}$. For any nonnegative integer $n\in\N$, set $t_n=nh$. Let $\bigl(\gamma_n\bigr)_{n\in\N^\star}$ be a sequence of independent standard centered $\R^d$-valued Gaussian random variables.

The objective of this article is to study the behavior of the stochastic gradient optimization scheme 
\begin{equation}\label{eq:scheme}
\left\lbrace
\begin{aligned}
&X_{n+1}=X_n-h\nabla F(X_n)+h\sigma(t_n,X_n)\gamma_{n+1},\quad \forall~n\in\N,\\
&X_0=x_0,
\end{aligned}
\right.
\end{equation}
where the mapping $F:\R^d\to\R$ is the objective function and the mapping $\sigma:[0,\infty)\times\R^d\to \mathcal{M}_d(\R)$ is the diffusion coefficient. We refer to Section~\ref{sec:setting_assumptions} for precise assumptions.

Let $\Delta B_{n+1}=B(t_{n+1})-B(t_n)$ denote the increments of the Wiener process on the interval $[t_n,t_{n+1}]$. Recall that $\bigl(h^{-\frac12}\Delta B_{n+1}\bigr)_{n\ge 0}$ are independent standard centered Gaussian $\R^d$-valued random variables. Thus the sequences of random variables $\bigl(\sqrt{h}\gamma_n\bigr)_{n\in\N^\star}$ and $\bigl(\Delta B_n\bigr)_{n\in\N^\star}$ are equal in distribution, and as a result, the scheme~\eqref{eq:scheme} can be interpreted using the alternative formulation
\begin{equation}\label{eq:schemebis}
\left\lbrace
\begin{aligned}
&X_{n+1}=X_n-h\nabla F(X_n)+\sqrt{h}\sigma(t_n,X_n)\Delta B_{n+1},\quad \forall~n\in\N,\\
&X_0=x_0.
\end{aligned}
\right.
\end{equation}
The formulation~\eqref{eq:schemebis} suggests to interpret the scheme as a numerical method applied to a stochastic differential equation, as will be explained below.

\subsection{Assumptions}\label{sec:setting_assumptions}

This section provides basic assumptions on the objective function $F$ and on the diffusion coefficient $\sigma$, which are necessary to justify rigorously the main properties of the stochastic gradient optimization scheme in the large time and small time-step regime. More precisely, Assumptions~\ref{ass:F1} and~\ref{ass:sigma1} are required to prove the first main result of this paper, stated in Theorem~\ref{theo:1}. Stronger additional conditions are imposed in Assumptions~\ref{ass:F2} and~\ref{ass:sigma2} below in order to prove the second main result of this paper, stated in Theorem~\ref{theo:2}, we refer to Section~\ref{sec:setting_order2} for their statements.

Let us first state the basic assumptions on the objective function $F$.
\begin{ass}[Basic assumptions on the objective function $F$]\label{ass:F1}
The mapping $F:\R^d\to\R$ is of class $\mathcal{C}^3$, and it satisfies the following properties:
\begin{enumerate}[label=(\roman*)]
\item $\nabla F:\R^d\to\R^d$ is globally Lipschitz continuous, i.e. there exists a positive real number $\LipF \in(0,\infty)$ such that one has
\begin{equation}\label{eq:assF1-Lip}
\|\nabla F(x_2)-\nabla F(x_1)\|\le\LipF \|x_2-x_1\|,\quad \forall~x_1,x_2\in\R^d.
\end{equation}
\item $F$ is a $\mu$-convex function, i.e. there exists a positive real number $\mu\in(0,\infty)$ such that one has
\begin{equation}\label{eq:assF1-muconvex}
\langle \nabla F(x_2)-\nabla F(x_1),x_2-x_1\rangle \ge \mu\|x_2-x_1\|^2,\quad \forall~x_1,x_2\in\R^d.
\end{equation}
\item The third order derivative of $F$ is bounded.
\end{enumerate}
\end{ass}

\begin{exam}
Assumption~\ref{ass:F1} is satisfied if $F$ is defined as
\[
F(x)=\|x\|^2+\epsilon G(x),\quad \forall~x\in\R^d,
\]
where the mapping $G:\R^d\to\R$ is of class $\mathcal{C}^3$ with bounded first, second and third order derivatives, and for sufficiently small $\epsilon$.
\end{exam}

Let us state several elementary consequences of Assumption~\ref{ass:F1}.

When the objective function $F$ satisfies Assumption~\ref{ass:F1}, it is straightforward to check that $\min(F)$ exists and is attained for a unique minimizer denoted by $x^\star$. Moreover the minimizer $x^\star$ is characterized by the first-order optimality condition $\nabla F(x^\star)=0$.

Note that the second order derivative $\nabla^2 F$ satisfies the property
\begin{equation}\label{eq:assF1-muconvex-bis}
\mu\|k\|^2 \leq \langle \nabla^2 F(x)k,k\rangle \le \LipF \|k\|^2,\quad \forall~k\in\R^d,
\end{equation}
where the upper bound follows from the Lipschitz continuity~\eqref{eq:assF1-Lip} of $\nabla F$ and the lower bound is a consequence of the $\mu$-convexity property~\eqref{eq:assF1-muconvex}. In addition, the global Lipschitz continuity condition~\eqref{eq:assF1-Lip} implies that $\nabla F$ has at most polynomial growth:
\begin{equation}\label{eq:assF1-growth}
\|\nabla F(x)\|\le \LipF \|x-x^\star\|,\quad \forall~x\in\R^d.
\end{equation}

Finally, note that combining the $\mu$-convexity condition~\eqref{eq:assF1-muconvex} and the optimality condition $\nabla F(x^\star)=0$, a second-order Taylor expansion yields the following lower bound on the residual 
\begin{equation}\label{eq:residual_low}
F(x)-F(x^\star)\ge \frac{\mu}{2}\|x-x^\star\|^2,\quad \forall~x\in\R^d.
\end{equation}
Similarly, one has an upper bound on the residual: 
\begin{equation}\label{eq:residual_up}
F(x)-F(x^\star)\le \frac{\LipF}{2}\|x-x^\star\|^2,\quad \forall~x\in\R^d.
\end{equation}

Let us now describe the basic assumptions on the diffusion coefficient $\sigma$.

\begin{ass}[Assumptions on the diffusion coefficient $\sigma$]\label{ass:sigma1}
The mapping $\sigma:[0,\infty)\times\R^d\to\mathcal{M}_d(\R)$ satisfies the following properties:

\begin{enumerate}[label=(\roman*)]
\item The mapping $\sigma:[0,\infty)\times\R^d\to \mathcal{M}_d(\R)$ is continuous.  
\item For all $t\in[0,\infty)$, the mapping $\sigma(t,\cdot):\R^d\to\mathcal{M}_d(\R)$ is globally Lipschitz continuous and grows at most linearly : for all $t\ge 0$, there exists $\varsigma(t)\in[0,\infty)$ such that one has
\begin{align}
&\|\sigma(t,x_2)-\sigma(t,x_1)\|\le \varsigma(t)\|x_2-x_1\|,\quad \forall~x_1,x_2\in\R^d,\label{eq:asssigma1Lip}\\
&\underset{x\in\R^d}\sup~\frac{\|\sigma(t,x)\|}{1+\|x-x^\star\|}\le \varsigma(t).\label{eq:asssigma1growth}
\end{align}
\item The mapping $\varsigma:[0,\infty)\to[0,\infty)$ is assumed to be bounded. Let $\|\varsigma\|_\infty=\underset{t\ge 0}\sup~|\varsigma(t)|$.
\end{enumerate}
\end{ass}

Finally, let us introduce the auxiliary mapping $a:[0,\infty)\times\R^d\to\mathcal{M}_d(\R)$ defined by
\begin{equation}\label{eq:a}
a(t,x)=\sigma(t,x)\sigma(t,x)^\star,\quad \forall~t\ge 0, x\in\R^d.
\end{equation}

\subsection{First order approximate ordinary and stochastic differential equations}\label{sec:setting_order1}

In this subsection, we consider that the basic assumptions on $F$ and $\sigma$, i.e. Assumptions~\ref{ass:F1} and~\ref{ass:sigma1} respectively, are satisfied.

Let us first study the deterministic situation: when $\sigma=0$, the scheme~\eqref{eq:schemebis} is the explicit Euler scheme applied to the deterministic gradient system
\begin{equation}\label{eq:ODE1}
\left\lbrace
\begin{aligned}
&\frac{d\X^0(t)}{dt}=-\nabla F(\X^0(t)),\quad \forall~t\ge 0,\\
&\X^0(0)=x_0.
\end{aligned}
\right.
\end{equation} 
The analysis of the gradient system~\eqref{eq:ODE1} under Assumption~\ref{ass:F1} is straightforward and well-known. Owing to the $\mu$-convexity condition~\eqref{eq:assF1-muconvex} from Assumption~\ref{ass:F1} on the objective function $F$, and due to the first-order optimality condition $\nabla F(x^\star)=0$, one obtains for all $t\ge 0$
\begin{align*}
\frac12\frac{d\|\X^0(t)-x^\star\|^2}{dt}&=\langle \frac{d(\X^0(t)-x^\star)}{dt},\X^0(t)-x^\star\rangle=-\langle \nabla F(\X^0(t))-\nabla F(x^\star),\X^0(t)-x^\star\rangle \\&\le -\mu\|\X^0(t)-x^\star\|^2.
\end{align*}
Therefore, applying the Gr\"onwall lemma, for all $t\ge 0$ one obtains
\begin{equation}\label{eq:cvX0}
\|\X^0(t)-x^\star\|\le e^{-\mu t}\|x_0-x^\star\|.
\end{equation}
This means that, for any initial value $x_0$, the solution $\X^0(t)$ at time $t$ of the deterministic gradient system~\eqref{eq:ODE1} converges exponentially fast to the unique minimizer $x^\star$ of the objective function $F$ in the large time regime $t\to\infty$, and that the rate of convergence is given by the parameter $\mu$.

Let us now take into account the stochastic perturbation. The stochastic gradient optimization scheme~\eqref{eq:scheme}, using the equivalent formulation~\eqref{eq:schemebis}, can be interpreted as the Euler--Maruyama scheme with time-step size $h$ applied to the stochastic differential equation
\begin{equation}\label{eq:SDE1}
\left\lbrace
\begin{aligned}
&d\X^h(t)=-\nabla F(\X^h(t))dt+\sqrt{h}\sigma(t,\X^h(t))dB(t),\quad \forall~t\ge 0,\\
&\X^h(0)=x_0.
\end{aligned}
\right.
\end{equation}
The noise in the stochastic differential equation~\eqref{eq:SDE1} is interpreted in the It\^o sense. Observe that the equation~\eqref{eq:SDE1} depends on the time-step size $h$, the diffusion coefficient is $\sqrt{h}\sigma$. Let us mention that the stochastic differential equation~\eqref{eq:SDE1} is globally well-posed, since the mappings $\nabla F$ and $\sigma(t,\cdot)$ satisfy appropriate globally Lipschitz continuity properties, given in Assumptions~\ref{ass:F1} and~\ref{ass:sigma1}: for any initial value $x_0\in\R^d$, there exists a unique global solution to the stochastic differential equation~\eqref{eq:SDE1}.

The deterministic system~\eqref{eq:ODE1} and the stochastic system~\eqref{eq:SDE1} are first-order approximations, in the sense of weak convergence, of the stochastic gradient optimization scheme~\eqref{eq:scheme} in the small time-step size limit $h\rightarrow 0$, for arbitrary fixed value of $T\in(0,\infty)$. More precisely, for any time $T\in(0,\infty)$ and any sufficiently smooth function $\varphi:\R^d\to\R$, there exists $C(T,x_0,\varphi)\in(0,\infty)$ such that for all $h\in(0,h_{\rm max})$ one has
\[
\underset{nh\le T}\sup~|\E[\varphi(X_n)]-\E[\varphi(\X^0(nh))]|+\underset{nh\le T}\sup~|\E[\varphi(X_n)]-\E[\varphi(\X^h(nh))]|\le C(T,x_0,\varphi)h.
\]

One of the objectives of this work is to show that under appropriate conditions the weak error estimates are uniform with respect to $T\in(0,\infty)$, which is of tremendous importance in the context of optimization where one needs to consider the large time regime $T\to \infty$. In other words, it will be proved that the weak error estimates above hold, with $\underset{T\in(0,\infty)}\sup~C(T,x_0,\varphi)<\infty$, we refer to Theorem~\ref{theo:1} stated in Section~\ref{sec:results} for a precise statement.

Note that in order to understand the behavior of the numerical scheme~\eqref{eq:scheme} considering the stochastic differential equation~\eqref{eq:SDE1} instead of the deterministic gradient system~\eqref{eq:ODE1} is not relevant since the weak order of convergence is equal to $1$ even when the noise is introduced. The objective of the next section is to introduce a stochastic differential equation depending on a modified objective function $F^h$, depending on the time-step size $h$, in order to obtain a second order weak approximation. As already mentioned, stronger conditions will be imposed on $F$ and $\sigma$ to prove this result.

\subsection{Second order approximate stochastic differential equation}\label{sec:setting_order2}

Following the works~\cite{LiTaiE,DDPR}, instead of considering the stochastic differential equation~\eqref{eq:SDE1}, we introduce the family of stochastic differential equations depending on the time-step size $h$ defined by
\begin{equation}\label{eq:SDE2}
\left\lbrace
\begin{aligned}
&d\Y^h(t)=-\nabla\bigl(F+\frac{h}{4}\|\nabla F\|^2\bigr)(\Y^h(t))dt+\sqrt{h}\sigma(t,\Y^h(t))dB(t),\quad \forall~t\ge 0,\\
&\Y^h(0)=x_0.
\end{aligned}
\right.
\end{equation}
This family provides second order weak approximation for the stochastic gradient optimization scheme~\eqref{eq:scheme}: for any time $T\in(0,\infty)$ and any sufficiently smooth function $\varphi:\R^d\to\R$, there exists $C(T,x_0,\varphi)\in(0,\infty)$ such that one has
\[
\underset{nh\le T}\sup~|\E[\varphi(X_n)]-\E[\varphi(\Y^h(nh))]|\le C(T,x_0,\varphi)h^2.
\]
Like for the first order approximations described in Section~\ref{sec:setting_order1}, the main objective of this work is to show that under appropriate conditions the weak error estimates are uniform with respect to $T\in(0,\infty)$. We refer to Theorem~\ref{theo:2} stated in Section~\ref{sec:results} for a precise statement.

Note that the stochastic differential equations~\eqref{eq:SDE1} and~\eqref{eq:SDE2} have the same diffusion coefficient $\sqrt{h}\sigma$, but they do not have the same drift: the drift in~\eqref{eq:SDE1} is $-\nabla F$ where $F$ is the objective function, whereas the drift in~\eqref{eq:SDE2} is given by $-\nabla F^h$ where $F^h:\R^d\to\R$ is the so-called modified objective function defined by
\begin{equation}\label{eq:Fh}
F^h(x)= F(x)+\frac{h}{4}\|\nabla F(x)\|^2,\quad \forall~x\in\R^d.
\end{equation}
It is worth mentioning that $x^\star$ is the unique minimizer of $F^h$ for any time-step size $h$. Indeed, one has $F^h(x^\star)=F(x^\star)$ and $F^h(x)>F(x)$ for all $x\in\R^d\setminus\{x^\star\}$.

Following from straightforward computations, note that $\nabla \|\nabla F\|^2(x^\star)=0$, and that there exists a positive real number $C \in (0, \infty)$ such that for all $x \in \mathbb{R}^d$ one has
\begin{equation}
   \|\nabla \|\nabla F\|^2(x)\|\le C\|x-x^\star\|, \hspace{0.6cm}
   \|\nabla^2 \|\nabla F\|^2(x)\|\le C(1+\|x-x^\star\|).\label{eq:nablaFh}
\end{equation}
As a consequence, there exists a positive real number $L_{F,2}\in(0,\infty)$ such that, for all $x_1,x_2\in\R^d$, one has
\begin{equation}\label{eq:assF2-LocLip}
\|\nabla \|\nabla F\|^2(x_2)-\nabla \|\nabla F\|^2(x_1)\|\le L_{F,2}(1+\|x_1-x^\star\|+\|x_2-x^\star\|)\|x_2-x_1\|.
\end{equation}

The conditions imposed on $F$ in Assumption~\ref{ass:F1} are not sufficient for the analysis, additional conditions are given in Assumption~\ref{ass:F2} below. Note that as usual higher order regularity conditions on $F$ and $\sigma$ are needed to justify the second order weak approximation result stated above, compared with the first order weak approximation result. A more restrictive condition also needs to be imposed: the mapping $F^h$ is assumed to be $\mu$-convex for any (sufficiently small) time-step size $h$, and the positive real number $\mu$ is independent of $h$. This assumption is non trivial and is discussed below. To the best of our knowledge, this type of condition has not been presented so far in the literature.

\begin{ass}[Strengthened assumptions on the objective function $F$]\label{ass:F2}

Assume that the objective function $F$ satisfies the basic Assumption~\ref{ass:F1}.

Furthermore, assume that the mapping $F:\R^d\to\R$ is of class $\mathcal{C}^5$, and that it satisfies the following properties:
\begin{enumerate}[label=(\roman*)]
\item There exist positive real numbers $h_0\in(0,h_{\rm max})$ and $\mu\in(0,\infty)$, such that for all $h\in(0,h_0)$ the mapping $F^h$ is $\mu$-convex:
\begin{equation}\label{eq:assF2-muconvex}
\langle \nabla F^h(x_2)-\nabla F^h(x_1),x_2-x_1\rangle \ge \mu\|x_2-x_1\|^2,\quad \forall~x_1,x_2\in\R^d,~\forall~h\in(0,h_0).
\end{equation}
\item The fourth and fifth order derivatives of $F$ are bounded.
\end{enumerate}
\end{ass}
Let us give an example of a function such that Assumption~\ref{ass:F2} is satisfied.
\begin{exam}
Let $F$ be given by $F(x)=\|x\|^2+\epsilon G(x)$ for all $x\in\R^d$, where $G:\R^d\to\R$ is a mapping of class $\mathcal{C}^5$ with bounded derivatives of any order and $\epsilon$ is chosen to ensure that the $\mu$-convexity condition from Assumption~\ref{ass:F1} is satisfied. Then Assumption~\ref{ass:F2} is satisfied if and only if there exists $c_F\in(0,\infty)$ such that
\[
\underset{x\in\R^d}\inf~D^3G(x).(x,k,k)\ge -c_F\|k\|^2, \forall~k\in\R^d,
\]
is satisfied.
\end{exam}

\begin{rem}
The mapping $-\nabla F^h:\R^d\to\R^d$ is only locally Lipschitz continuous, owing to the inequality~\eqref{eq:assF2-LocLip}, following from the basic Assumption~\ref{ass:F1}: there exists $C\in(0,\infty)$ such that for all $x_1,x_2\in\R^d$ one has
\begin{equation}\label{eq:localLipFh}
\|\nabla F^h(x_2)-\nabla F^h(x_1)\|\le (\LipF+hL_{F,2})(1+\|x_1-x^\star\|+\|x_2-x^\star\|)\|x_2-x_1\|.
\end{equation}
This local Lipschitz continuity property ensures local well-posedness of the stochastic differential equation~\eqref{eq:SDE2}. Owing to the strengthened Assumption~\ref{ass:F2}, $F^h$ satisfies the $\mu$-convexity condition~\eqref{eq:assF2-muconvex}, which implies that $-\nabla F^h$ satisfies a one-sided Lipschitz continuity property. This provides the global well-posedness of~\eqref{eq:SDE2} by standard arguments. Moreover, the $\mu$-convexity condition~\eqref{eq:assF2-muconvex} is employed to prove moment bounds on the solution of the stochastic differential equation~\eqref{eq:SDE2}, see Lemma~\ref{lem:momentboundsYh} in Section~\ref{sec:proof2}, which are uniform with respect to $t\ge 0$.
\end{rem}

Finally, let us describe additional regularity conditions imposed on the diffusion coefficient $\sigma$. These assumptions are less restrictive than those imposed on $F$ in Assumption~\ref{ass:F2} above.

\begin{ass}[Strengthened assumptions on the diffusion coefficient $\sigma$]\label{ass:sigma2}
    Assume that the diffusion coefficient $\sigma$ satisfies the basic Assumption~\ref{ass:sigma1}.
    
    Furthermore, assume that for all $t\in[0,\infty)$ the mapping $\sigma(t,\cdot):\R^d\to\mathcal{M}_d(\R)$ is of class $\mathcal{C}^3$. Moreover, the second and the third derivatives of $\sigma(t,\cdot)$ are assumed to be bounded, uniformly with respect to $t\in[0,\infty)$: one has
    \[
    \underset{t\in[0,\infty)}\sup~\underset{x\in\R^d}\sup~\underset{k_1,k_2\in\R^d\setminus\{0\}}\sup~\frac{\|D^2\sigma(t,x).(k_1,k_2)\|}{\|k_1\|\|k_2\|}+\underset{t\in[0,\infty)}\sup~\underset{x\in\R^d}\sup~\underset{k_1,k_2,k_3\in\R^d\setminus\{0\}}\sup~\frac{\|D^3\sigma(t,x).(k_1,k_2,k_3)\|}{\|k_1\|\|k_2\|\|k_3\|}<\infty.
    \]
    Finally, assume that for all $x\in\R^d$, the mapping $\sigma(\cdot,x):[0,\infty)\to\mathcal{M}_d(\R)$ is globally Lipschitz continuous, uniformly with respect to $x$: there exists a positive real number $\Lips \in(0,\infty)$ such that for all $x\in\R^d$ one has
\begin{equation}\label{eq:asssigma2Lip}
    \|\sigma(t_2,x)-\sigma(t_1,x)\|\le \Lips|t_2-t_1|,\quad \forall~t_1,t_2\in [0,\infty).
\end{equation}
\end{ass}

We conclude this section by remarks on the modified objective function $F^h$.
\begin{rem}
The gradient $\nabla F^h$ of the modified objective function can be written as
\[
\nabla F^h(x)=(I+\frac{h}{2}\nabla^2F(x))\nabla F(x), \forall~x\in\R^d.
\]
This observation suggests to propose alternative stochastic differential equations which provide a second order weak approximation instead of~\eqref{eq:SDE2}. For instance, one can consider the stochastic differential equations
\[
d\overline{\Y}^h(t)=\overline{b}^h(\overline{\Y}^h(t))dt+\sqrt{h}\sigma(t,\overline{\Y}^h(t))dB(t)
\]
with $\overline{b}^h(x)=-(I-\frac{h}{2}\nabla^2F(x))^{-1}\nabla F(x)$ (under the condition $h<2L_F$ to ensure that the matrix is invertible) or with $\overline{b}^h(x)=-\exp(\frac{h}{2}\nabla^2F(x))\nabla F(x)$. Note that in general it does not seem possible to write $\overline{b}^h(x)$ given in the two examples above as a gradient. It can be proved that second order weak error estimates are still valid with $\Y^h(nh)$ replaced by $\overline{\Y}^h(nh)$, for $nh\le T$.

It would be interesting to exhibit a choice of alternative drift coefficient $\overline{b}^h$ such that one could obtain uniform in time weak error estimates at second order in $h$ under weaker conditions on the objective function $F$. For the two examples mentioned above this does not seem to be the case. We leave investigations in this direction for future works.
\end{rem}

\section{Main results}\label{sec:results}

This section is devoted to the full and rigorous statements of the main results of this article in Section~\ref{sec:results1}. We then discuss the relevance of the results through numerical experiments in Section~\ref{sec:results2} and complexity analysis in Section~\ref{sec:results3}. The proofs of Theorems~\ref{theo:1} and~\ref{theo:2} are technical and require some auxiliary results, they are postponed to Sections~\ref{sec:proof1} and~\ref{sec:proof2} respectively.

In each of the main results three types of errors are considered. Weak error estimates deal with $\E[\varphi(X_N)]-\varphi(x^\star)$ for an appropriate class of functions $\varphi$, errors on the expected residual are concerned with $\E[F(X_N)]-F(x^\star)$, and strong error estimates give upper bounds on $\E[\|X_N-x^\star\|]$ and on $\bigl(\E[\|X_N-x^\star\|^2]\bigr)^{\frac12}$. By the Cauchy--Schwarz inequality, one has  $\E[\|X_N-x^\star\|]\le \bigl(\E[\|X_N-x^\star\|^2]\bigr)^{\frac12}$. The main task is to prove the weak error estimates. The error bounds on the residual will be obtained from weak error bounds when $\varphi=F$. Furthermore, the strong error estimates follow from the error bounds on the residual by a straightforward application of the inequality~\eqref{eq:residual_low}.

Error bounds are made of two contributions. The first contribution is related to the large time behavior of solutions to the associated continuous-time dynamics (given by the ordinary differential equation~\eqref{eq:ODE1} for Theorem~\ref{theo:1} and by the stochastic differential equation~\eqref{eq:SDE2} for Theorem~\ref{theo:2}). The second contribution is related to a time-stepping error as $h$ goes to $0$, when comparing the solutions to the associated continuous-time dynamics and the stochastic optimization scheme~\eqref{eq:scheme}.

\subsection{Error estimates}\label{sec:results1}

The first main result is Theorem~\ref{theo:1} below.

\begin{theo}\label{theo:1}
Assume that the objective function $F$ and that the diffusion coefficient $\sigma$ satisfy the basic Assumption~\ref{ass:F1} and the basic Assumption~\ref{ass:sigma1} respectively.

There exist positive real numbers $H\in(0,h_{\max})$ and $C\in(0,\infty)$ such that, for any initial value $x_0\in\R^d$, for all $N\in\N^\star$ and $h\in(0,H)$, the following error estimates are satisfied:
\begin{itemize}
\item weak error estimates: for any mapping $\varphi:\R^d\to\R$ of class $\mathcal{C}^2$ with bounded second order derivative, one has
\begin{equation}\label{eq:theo1-weak}
\big|\E[\varphi(X_N)]-\varphi(x^\star)\big|\le C_1(\varphi,x_0)\bigl(e^{-\mu Nh}+h\bigr),
\end{equation}
where the positive real number $C_1(\varphi,x_0)$ is given by
\begin{equation}\label{eq:constantC1}
C_1(\varphi,x_0)=C\bigl(\|\nabla\varphi(x^\star)\|+\vvvert\varphi\vvvert_2\bigr)\bigl(1+\|x_0-x^\star\|^3\bigr),
\end{equation}
\item error estimates on the expected residual: one has
\begin{equation}\label{eq:theo1-residual}
\E[F(X_N)]-F(x^\star)\le C\bigl(1+\|x_0-x^\star\|^3\bigr)\bigl(e^{-2\mu Nh}+h\bigr),
\end{equation}
\item strong error estimates: one has
\begin{equation}\label{eq:theo1-strong}
\E[\|X_N-x^\star\|]\le \bigl(\E[\|X_N-x^\star\|^2]\bigr)^{\frac12}\le C\bigl(1+\|x_0-x^\star\|^{\frac32}\bigr)\bigl(e^{-\mu Nh}+h^{\frac12}\bigr).
\end{equation}
\end{itemize}
\end{theo}
We refer to Section~\ref{sec:proof1} for a presentation of the strategy of the proof and to Section~\ref{sec:conclusion-proof1} for the details of the proofs of each of the error estimates.

Observe that Theorem~\ref{theo:1} only requires to impose the basic Assumptions~\ref{ass:F1} and~\ref{ass:sigma1} on $F$ and $\sigma$. The time-stepping error is of size $h$ in the weak error estimates~\eqref{eq:theo1-weak} and the estimates on the expected residual~\eqref{eq:theo1-residual}, and of size $h^{\frac12}$ in the strong error estimates~\eqref{eq:theo1-strong}. Improved error bounds, where the time-stepping error is of size $h^2$ in weak error estimates and estimates on the expected residual, and of size $h$ in the strong error estimates, are stated in Theorem~\ref{theo:2} below. Increasing the order of convergence for the time-stepping error requires to impose the strengthened Assumptions~\ref{ass:F2} and~\ref{ass:sigma2} on $F$ and $\sigma$. As will be clear in the statement of Theorem~\ref{theo:2}, there is a non-trivial impact on the other contribution of the error related to the large time behavior of solutions to the stochastic differential equation~\eqref{eq:SDE2}.

Let us introduce the mapping $\rho$ defined by
\begin{equation}\label{eq:rhoT}
\rho(T) = \int_0^{T}e^{2\mu s}\varsigma(s)^2 \,ds,\quad \forall~T\in[0,\infty),
\end{equation}
where the mapping $\varsigma:[0,\infty)\to[0,\infty)$ is introduced in Assumption~\ref{ass:sigma1}. Note that the mapping $T\in[0,\infty)\mapsto e^{-2\mu T}\rho(T)\in[0,\infty)$ is bounded: owing to Assumption~\ref{ass:sigma1}, $\varsigma$ is bounded, therefore one has
\begin{equation}\label{eq:boundrho}
\underset{T\ge 0}\sup~e^{-2\mu T}\rho(T)\le \frac{\|\varsigma\|_\infty^2}{2\mu}.
\end{equation}
As will be clearly explained below, without further conditions on the behavior of $\varsigma(T)$ when $T\to\infty$, obtaining Theorem~\ref{theo:2} may not be of any practical interest. However assuming that $\varsigma(T)$ tends to $0$ as $T\to\infty$, one obtains
\begin{equation}\label{eq:cvrho}
e^{-2\mu T}\rho(T)\underset{T\to\infty}\to 0,
\end{equation}
and the speed of convergence has an impact when analyzing the computational complexity of the algorithm when comparing Theorems~\ref{theo:1} and~\ref{theo:2}.

\begin{theo}\label{theo:2}
Assume that the objective function $F$ and that the diffusion coefficient $\sigma$ satisfy the strengthened Assumption~\ref{ass:F2} and the strengthened Assumption~\ref{ass:sigma2} respectively.

There exist positive real numbers $H\in(0,h_{\max})$ and $C\in(0,\infty)$ such that, for any initial value $x_0\in\R^d$, for all $N\in\N^\star$ and $h\in(0,H)$, the following error estimates are satisfied:
\begin{itemize}
\item weak error estimates: for any mapping $\varphi:\R^d\to\R$ of class $\mathcal{C}^3$ with bounded second and third order derivatives, one has
\begin{equation}\label{eq:theo2-weak}
\big|\E[\varphi(X_N)]-\varphi(x^\star)\big|\le C_2(\varphi,x_0)\left(e^{-\mu Nh}\left(1+h\rho(Nh)\right)^{\frac12}+h^2\right),
\end{equation}
where the positive real number $C_2(\varphi,x_0)$ is given by
\begin{equation}\label{eq:constantC2}
C_2(\varphi,x_0)=C\bigl(\|\nabla\varphi(x^\star)\|+\vvvert\varphi\vvvert_2+\vvvert\varphi\vvvert_3\bigr)\bigl(1+\|x_0-x^\star\|^5\bigr).
\end{equation}
\item error estimates on the expected residual: one has
\begin{equation}\label{eq:theo2-residual}
\E[F(X_N)]-F(x^\star)\le C\bigl(1+\|x_0-x^\star\|^5\bigr)\left(e^{-2\mu Nh}\left(1+h\rho(Nh)\right)+h^2\right),
\end{equation}
\item strong error estimates: one has
\begin{equation}\label{eq:theo2-strong}
\E[\|X_N-x^\star\|]\le \bigl(\E[\|X_N-x^\star\|^2]\bigr)^{\frac12}\le C\bigl(1+\|x_0-x^\star\|^{\frac52}\bigr)\left(e^{-\mu Nh}\left(1+h\rho(Nh)\right)^{\frac{1}{2}}+h\right).
\end{equation}
\end{itemize}
\end{theo}

Owing to the boundedness property~\eqref{eq:boundrho}, it is apparent that comparing the error estimate for the expected residual~\eqref{eq:theo1-residual} and~\eqref{eq:theo2-residual}, respectively the strong error estimates~\eqref{eq:theo1-strong} and~\eqref{eq:theo2-strong}, the upper bound from Theorem~\ref{theo:2} is dominated by the upper bound from Theorem~\ref{theo:1}. If one compares the weak error estimates~\eqref{eq:theo1-weak} and~\eqref{eq:theo2-weak}, the comparison is not clear. We refer to Section~\ref{sec:results3} for an analysis of the computational complexity.

\begin{rem}
The reason why $\rho(T)$ with $T=Nh$ appears in Theorem~\ref{theo:2}, whereas it does not appear in Theorem~\ref{theo:1} is due to the use of different decompositions of the weak error
\begin{align*}
\E[\varphi(X_N)]-\varphi(x^\star)&=\varphi(\X^0(Nh))-\varphi(x^\star)+\E[\varphi(X_N)]-\varphi(\X^0(Nh))\\
&=\E[\varphi(\Y^h(Nh))]-\varphi(x^\star)+\E[\varphi(X_N)]-\E[\varphi(\Y^h(Nh))].
\end{align*}
On the one hand, the first decomposition~\eqref{eq:decomperror1} (see Section~\ref{sec:proof1}) is the primary step in the proof of Theorem~\ref{theo:1}, and exploits the large time behavior of the solution $\X^0(T)$ of the deterministic gradient system~\eqref{eq:ODE1}. The diffusion coefficient $\sigma$ does not appear in this description. On the other hand, the second decomposition~\eqref{eq:decomperror2} (see Section~\ref{sec:proof2}) is the primary step in the proof of Theorem~\ref{theo:2}. To deal with the first term of the decomposition, it is necessary to study the large time behavior of the solution $\Y^h(T)$ of the stochastic differential equation~\eqref{eq:SDE2}, see Proposition~\ref{propo:YT}, which depends on $\rho$.
\end{rem}

\subsection{Numerical experiments}\label{sec:results2}

The objective of this section is investigate the behavior of the stochastic optimization scheme~\eqref{eq:scheme} and to verify that the terms appearing in the upper bounds from Theorem~\ref{theo:2} are relevant. We focus only on strong error estimates~\eqref{eq:theo2-strong}, which may be rewritten as
\begin{equation}\label{eq:num}
\bigl(\E[\|X_N-x^\star\|^2]\bigr)^{\frac12}\le C(x_0)\left(e^{-\mu T}+\sqrt{h} \sqrt{e^{-2\mu T}\rho(T)}+h\right),
\end{equation}
using the notation $T=Nh$.

The numerical experiments are performed in dimension $d=1$. In addition, the noise is additive, i.e. $\sigma$ only depends only on time $t$: one has
\[
\sigma(t,x)=\sigma(t),\quad\forall~t\ge 0,~x\in\R.
\]
Choosing different objective functions $F$ and mappings $\sigma$ allows us to illustrate below that the three terms appearing in the right-hand side of~\eqref{eq:num} can be observed when evaluating the error as a function of the time-step size $h$, for different values of $T$.

\subsubsection{Quadratic objective function}

Let us first assume that the objective function is given by
\[
F(x)=\frac{x^2}{2}.
\]
The unique minimizer of $F$ is $x^\star=0$. In the numerical experiments, the initial value is $x_0=1$. To estimate the strong error $\bigl(\E[\|X_N-x^\star\|^2]\bigr)^{\frac12}$, Monte Carlo averages over $100$ independent realizations is computed. The time-step size $h$ takes values in
\[
\{1.0~10^{-2},5.0~10^{-3},2.5~10^{-3},1.25~10^{-3},6.25~10^{-4},3.125~10^{-4},1.5625~10^{-4}\}.
\]
The error is represented in logarithmic scale and a reference line with slope $1/2$ is displayed.

First, let us assume that $\sigma(t)=e^{-at}$, with $a\in\{1.5,1.0,0.5,0.1\}$. Figure~\ref{fig:quad-expo} illustrates the dependence of the strong error with respect to $h$ for different values of the final time $T\in\{2,4,8,16,32\}$. For small values of $T$ ($T=2$ and $T=4$), the error does not depend on $h$, and is also seen to be independent of the parameter $a$. This behavior corresponds to the observation of the error term $e^{-\mu T}$ in the right-hand side of~\eqref{eq:num}. For large values of $T$ ($T=16$ and $T=32$), the error does not depend on $h$ if $a$ is large ($a=1.5$ and $a=1$) however it depends on $h$ and converges with rate $1/2$ for smaller values of $a$ ($a=0.5$ and $a=0.1$). In addition, the error depends on $T$ (for a fixed  value of $a$). This behavior corresponds to the observation of the error term $\sqrt{h} \sqrt{e^{-2\mu T}\rho(T)}$ in the right-hand side of~\eqref{eq:num}.

\begin{figure}[ht]
\centering
\begin{subfigure}
  \centering
  \includegraphics[width=0.4\textwidth]{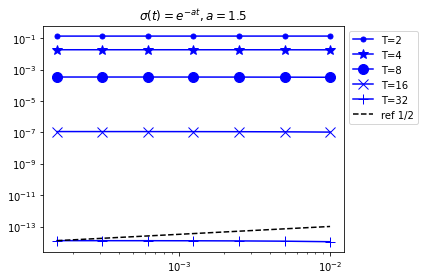}
\end{subfigure}%
\begin{subfigure}
  \centering
  \includegraphics[width=0.4\textwidth]{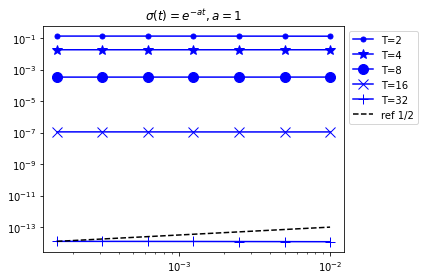}
\end{subfigure}%
\begin{subfigure}
  \centering
  \includegraphics[width=0.4\textwidth]{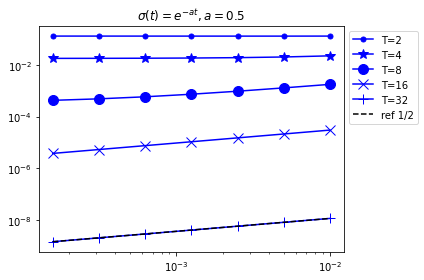}
\end{subfigure}%
\begin{subfigure}
  \centering
  \includegraphics[width=0.4\textwidth]{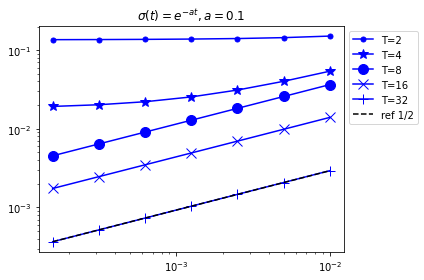}
\end{subfigure}%
\caption{Illustration of strong error estimates for the quadratic objective function $F$, and with $\sigma(t)=e^{-at}$ for $a=1.5$ (top left), $a=1.0$ (top right), $a=0.5$ (bottom left) and $a=0.1$ (bottom right). The values of the final time are $T\in\{2,4,8,16,32\}$.}
\label{fig:quad-expo}
\end{figure}

Second, let us assume that $\sigma(t)=\frac{1}{1+t^a}$, with $a\in\{2.0,1.5,1.0,0.5\}$. Figure~\ref{fig:quad-poly} illustrates the dependence of the strong error with respect to $h$ for different values of the final time $T\in\{2,4,8,16,32\}$. Like in the numerical experiments reported in Figure~\ref{fig:quad-expo}, for small values of $T$ ($T=2$ and $T=4$), the error does not depend on $h$, and is also seen to be independent of the parameter $a$. This behavior corresponds to the observation of the error term $e^{-\mu T}$ in the right-hand side of~\eqref{eq:num}. For large values of $T$ ($T=16$ and $T=32$), the error converges with rate $1/2$ for all the values of $a$. In addition, the error depends on $T$. This behavior corresponds to the observation of the error term $\sqrt{h} \sqrt{e^{-2\mu T}\rho(T)}$ in the right-hand side of~\eqref{eq:num}.

\begin{figure}[ht]
\centering
\begin{subfigure}
  \centering
  \includegraphics[width=0.4\textwidth]{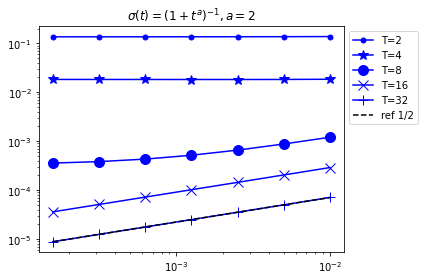}
\end{subfigure}%
\begin{subfigure}
  \centering
  \includegraphics[width=0.4\textwidth]{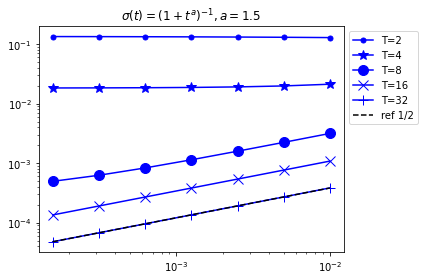}
\end{subfigure}%
\begin{subfigure}
  \centering
  \includegraphics[width=0.4\textwidth]{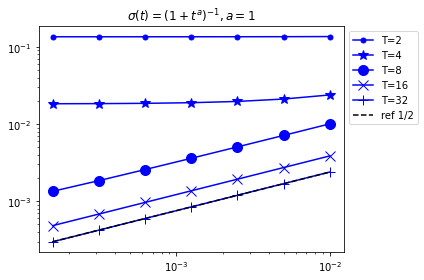}
\end{subfigure}%
\begin{subfigure}
  \centering
  \includegraphics[width=0.4\textwidth]{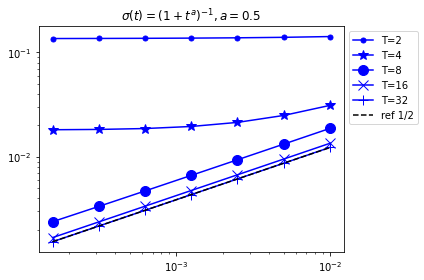}
\end{subfigure}%
\caption{Illustration of strong error estimates for the quadratic objective function $F$, and with $\sigma(t)=\frac{1}{1+t^a}$ for $a=2.0$ (top left), $a=1.5$ (top right), $a=1.0$ (bottom left) and $a=0.5$ (bottom right). The values of the final time are $T\in\{2,4,8,16,32\}$.}
\label{fig:quad-poly}
\end{figure}

Note that in the two numerical experiments reported in Figures~\ref{fig:quad-expo} and~\ref{fig:quad-poly}, one does not observe the third error term $h$ appearing in the right-hand side of~\eqref{eq:num}.

\subsubsection{Non-quadratic objective function}

In order to exhibit the third error term $h$ appearing in the right-hand side of~\eqref{eq:num}, we need to consider a non-quadratic objective function $F=F_\epsilon$, given by
\[
F_\epsilon'(x)=x+\arctan(x/\epsilon)
\]
where $\epsilon$ is a parameter, which takes values in $\{10^{-1},10^{-2},10^{-3},10^{-4},10^{-5},10^{-6}\}$.

The unique minimizer of $F$ is $x^\star=0$. In the numerical experiments, the initial value is $x_0=1$. To estimate the strong error $\bigl(\E[\|X_N-x^\star\|^2]\bigr)^{\frac12}$, Monte Carlo averages over $10$ independent realizations is computed. The time-step size $h$ takes values in
\[
\{1.0~10^{-2},5.0~10^{-3},2.5~10^{-3},1.25~10^{-3},6.25~10^{-4},3.125~10^{-4},1.5625~10^{-4},7.8125~10^{-5},3.90625~10^{-5}\}.
\]
The error is represented in logarithmic scale and reference lines with slopes $1/2$ and $1$ are displayed.

In Figures~\ref{fig:arctan-expo} and~\ref{fig:arctan-poly}, one observes how the error depends on the parameter $\epsilon$ for different values of $T\in\{8,16,32\}$, respectively with $\sigma(t)=e^{-t}$ and $\sigma(t)=\frac{1}{1+t}$. The value of $\epsilon$ has a strong impact. For large values of $\epsilon$ ($\epsilon=10^{-1}$ and $\epsilon=10^{-2}$), the same behavior as in the case of a quadratic objective function is observed: the rate of convergence with respect to $h$ is $1/2$ and the error decreases when $T$ increases. This behavior corresponds to the observation of the error term $\sqrt{h} \sqrt{e^{-2\mu T}\rho(T)}$ in the right-hand side of~\eqref{eq:num}. For small values of $\epsilon$ ($\epsilon=10^{-5}$ and $\epsilon=10^{-6}$), one observes a different behavior: the rate of convergence with respect to $h$ is $1$ and the error is independent of $T$. This behavior corresponds to the observation of the error term $h$ in the right-hand side of~\eqref{eq:num}. For intermediate values of $\epsilon$ ($\epsilon=10^{-5}$ and $\epsilon=10^{-6}$), a transition between the two regimes is observed, depending on the size of $h$. Figures~\ref{fig:arctan-expo} and~\ref{fig:arctan-poly} thus show that the error bounds obtained in Theorem~\ref{theo:2} are sharp.

In Figure~\ref{fig:arctan-T}, a small sample of the same numerical experiments are presented differently, in order to illustrate how the error depends on $T$ when the value of $\epsilon$ is fixed. First, one observes that for $\epsilon=10^{-1}$ the rate of convergence with respect to $h$ is $1/2$ and that the error decreases as $T$ increases. Second, one observes that for $\epsilon=10^{-6}$ the rate of convergence with respect to $h$ is $1$ and that the error does not depend significantly on $T$. This confirms the interpretation of the results above.

Let us explain the role of the parameter $\epsilon$: the second and the third error terms in the right-hand side of~\eqref{eq:num} depend on $\epsilon$, more precisely one should write 
\[
C_{\epsilon,1}(x_0)\sqrt{h}\sqrt{e^{-2\mu T}\rho(T)}+C_{\epsilon,2}(x_0)h,
\]
where $C_{\epsilon,1}(x_0),C_{\epsilon,2}(x_0)$ depend on derivatives of the objective function $F_\epsilon$ and go to infinity as $\epsilon$ goes to $0$ at different speeds: when $\epsilon\to 0$, one has $C_{\epsilon,1}(x_0)={\rm o}\bigl(C_{\epsilon,2}(x_0)\bigr)$. For a fixed value of $\epsilon$, when $h$ goes to $0$ only the first error term above should be visible in the numerical experiments. However, choosing small values of $\epsilon$ allows us to exhibit the second error term for the considered values of $h$.
\begin{figure}[ht]
\centering
\begin{subfigure}
  \centering
  \includegraphics[width=0.4\textwidth]{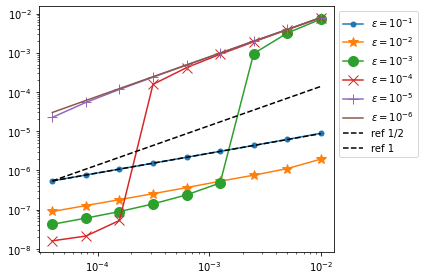}
\end{subfigure}%
\begin{subfigure}
  \centering
  \includegraphics[width=0.4\textwidth]{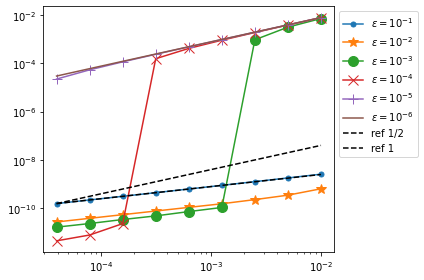}
\end{subfigure}%
\begin{subfigure}
  \centering
  \includegraphics[width=0.4\textwidth]{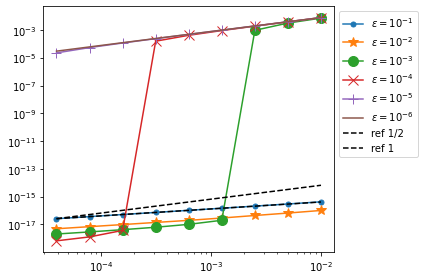}
\end{subfigure}%
\caption{Illustration of strong error estimates for the non-quadratic objective function $F_\epsilon$, for different values of $\epsilon$, and with $\sigma(t)=e^{-t}$. The values of the final time are $T=8$ (left), $T=16$ (right) and $T=32$ (bottom).}
\label{fig:arctan-expo}
\end{figure}

\newpage

\begin{figure}[ht]
\centering
\begin{subfigure}
  \centering
  \includegraphics[width=0.4\textwidth]{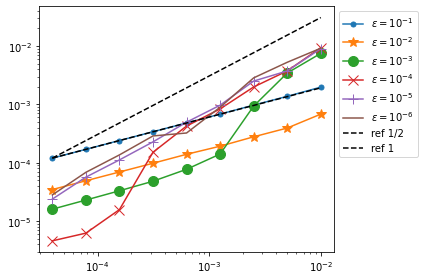}
\end{subfigure}%
\begin{subfigure}
  \centering
  \includegraphics[width=0.4\textwidth]{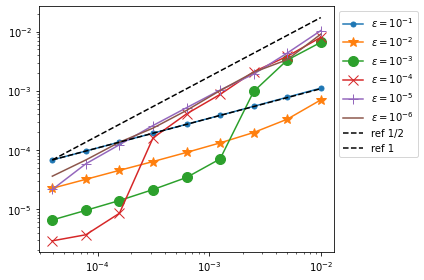}
\end{subfigure}%
\begin{subfigure}
  \centering
  \includegraphics[width=0.4\textwidth]{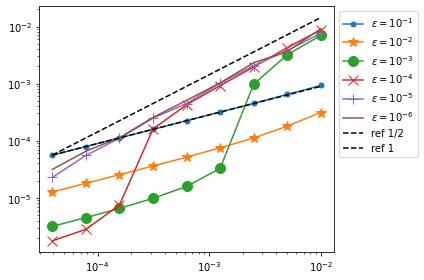}
\end{subfigure}%
\caption{Illustration of strong error estimates for the non-quadratic objective function $F_\epsilon$, for different values of $\epsilon$, and with $\sigma(t)=\frac{1}{1+t}$. The values of the final time are $T=8$ (left), $T=16$ (right) and $T=32$ (bottom).}
\label{fig:arctan-poly}
\end{figure}

\begin{figure}[H]
\centering
\begin{subfigure}
  \centering
  \includegraphics[width=0.39\textwidth]{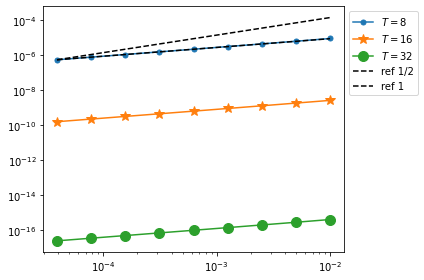}
\end{subfigure}%
\begin{subfigure}
  \centering
  \includegraphics[width=0.39\textwidth]{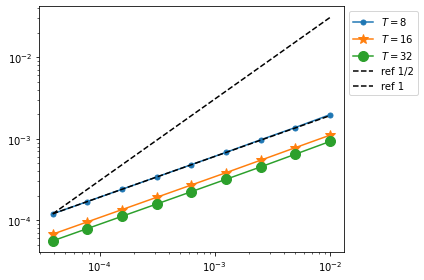}
\end{subfigure}%
\begin{subfigure}
  \centering
  \includegraphics[width=0.39\textwidth]{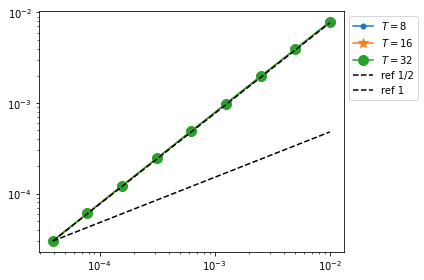}
\end{subfigure}%
\begin{subfigure}
  \centering
  \includegraphics[width=0.39\textwidth]{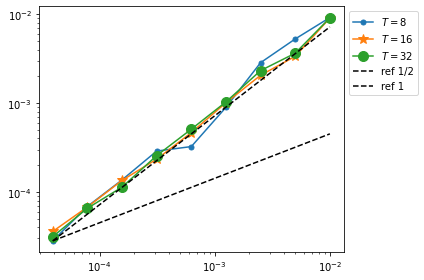}
\end{subfigure}%
\caption{Illustration of strong error estimates for the non-quadratic objective function $F_\epsilon$, for different values of $T$. Left: $\sigma(t)=e^{-t}$. Right: $\sigma(t)=\frac{1}{1+t}$. Top: $\epsilon=10^{-1}$. Bottom: $\epsilon=10^{-6}$.}
\label{fig:arctan-T}
\end{figure}

\subsection{Complexity analysis}\label{sec:results3}

Let us perform some complexity analysis associated with the convergence results stated in Theorems~\ref{theo:1} and~\ref{theo:2}. We only focus on the strong error estimates~\eqref{eq:theo1-strong} and~\eqref{eq:theo2-strong}. The conclusions would be the same for the error estimates on the expected residual~\eqref{eq:theo1-residual} and~\eqref{eq:theo2-residual}, owing to the inequalities~\eqref{eq:residual_low} and~\eqref{eq:residual_up}. However, the conclusions are different if one considers weak error estimates~\eqref{eq:theo1-weak} and~\eqref{eq:theo2-weak}, but the discussion is omitted in this work.

In order to study the computational complexity, let $\delta\in(0,1)$ be an arbitrarily small error size, then one needs to find $h\in(0,H)$ and $N\in\N^\star$ such that the strong error satisfies
\begin{equation}\label{eq:coststrong}
\bigl(\E[\|X_N-x^\star\|^2]\bigr)^{\frac12}\le \delta.
\end{equation}
The computational cost is proportional to the number of time steps $N$. Below, the symbol $\approx$ is used to denote the size of the error, constants which are independent of the considered parameters $\delta$, $h$ and $N$ are omitted.

First, if Theorem~\ref{theo:1} is considered, owing to the strong error estimate~\eqref{eq:theo1-strong} the upper bound~\eqref{eq:coststrong} is satisfied if $h$ is chosen such that $h\approx \delta^2$ and if then $N$ is chosen such that $Nh\approx \log(\delta^{-1})$. Therefore, the computational cost is of size
\[
N=N_1(\delta)\approx \delta^{-2}\log(\delta^{-1}).
\]
Second, if Theorem~\ref{theo:2} is considered, the analysis depends on the mapping $\varsigma$. As already explained in Section~\ref{sec:results1}, since the mapping $T\mapsto e^{-2\mu T}\rho(T)$ is bounded, the right-hand side of~\eqref{eq:theo2-strong} is dominated by the right-hand side of~\eqref{eq:theo1-strong}, therefore one always observe a lower computational cost when considering Theorem~\ref{theo:2} instead of Theorem~\ref{theo:1}. It remains to study whether some significant reduction of the computational cost is observed. We consider two cases: either we assume that $\varsigma(t)=e^{-at}$ or that $\varsigma(t)=(1+t^{a})^{-1}$, for some positive real number $a\in(0,\infty)$. These are the two cases considered in the numerical experiments reported in Section~\ref{sec:results2}.

If $\varsigma(t)=e^{-at}$, then it is straightforward to check one has
\[
e^{-2\mu T}\rho(T)\le C_ae^{-2\mu_aT},\quad \forall~T\ge 0,
\]
for some $C_a\in(0,\infty)$ and $\mu_a\in(0,\mu]$. As a result, owing to the strong error estimate~\eqref{eq:theo2-strong} the upper bound~\eqref{eq:coststrong} is satisfied if $h$ is chosen such that $h\approx\delta$ and if then $N$ is chosen such that $Nh\approx \log(\delta^{-1})$. Therefore, the computational cost is of size
\[
N=N_2(\delta)\approx \delta^{-1}\log(\delta^{-1}).
\] 
In that case, the comparison of the computational costs is favorable to Theorem~\ref{theo:2} compared with Theorem~\ref{theo:1}:
\[
\frac{N_2(\delta)}{N_1(\delta)}\underset{\delta\to 0}\to 0.
\]
It is thus worth using the higher order approximation with respect to $h$ to analyze the algorithm in that case, but it should be reminded that Theorem~\ref{theo:2} requires more stringent conditions on the objective function, imposed from Assumption~\ref{ass:F2}.

If $\varsigma(t)=(1+t^{a})^{-1}$, then it is straightforward to check one has
\[
e^{-2\mu T}\rho(T)\le \frac{C_a}{1+T^{2a}},\quad \forall~T\ge 0,
\]
for some $C_a\in(0,\infty)$. As a result, owing to the strong error estimate~\eqref{eq:theo2-strong} the upper bound~\eqref{eq:coststrong} is satisfied if $h$ is chosen such that $h\approx\delta$ and if then $N$ is chosen such that $Nh\approx \delta^{-\frac{1}{2a}}$. Therefore, the computational cost is of size
\[
N=N_2(\delta)\approx \delta^{-1-\frac{1}{2a}}.
\]
Contrary to the situation described above, the comparison of the computational costs $N_1(\delta)$ and $N_2(\delta)$ associated with the applications of Theorems~\ref{theo:1} and~\ref{theo:2} respectively depends on the value of the parameter $a$, i.e. on the speed of convergence to $0$ of the mapping $\sigma$. Indeed one has
\[
\frac{N_2(\delta)}{N_1(\delta)}\underset{\delta\to 0}\to
\begin{cases}
0,\quad \text{if}~a\ge 1/2,\\
\infty,\quad \text{if}~a<1/2.
\end{cases}
\]
As a conclusion, Theorem~\ref{theo:2} is more favorable than Theorem~\ref{theo:1} in terms of computational complexity only when $a\ge 1/2$.

\section{Auxiliary results on the stochastic gradient optimization scheme}\label{sec:aux}

This section is devoted to studying some properties of the discrete-time stochastic process $\bigl(X_n\bigr)_{n\in\N}$ defined by the stochastic gradient optimization scheme~\eqref{eq:scheme}. First, Lemma~\ref{lem:momentbounds-scheme} states moment bounds, which are uniform with respect to the time-step size $h$, and with respect to time $n\in\N$. Second, for any time-step size $h$ an auxiliary continuous time process $\bigl(\tilde{X}^h(t)\bigr)_{t\ge 0}$ is introduced and some properties are given in Lemma~\ref{lem:tildeX}.

\subsection{Moment bounds on the scheme}

\begin{lemma}\label{lem:momentbounds-scheme}
Assume that the objective $F$ satisfies the basic Assumption~\ref{ass:F1} and that the diffusion coefficient $\sigma$ satisfies the basic Assumption \ref{ass:sigma1}. There exist a non-increasing sequence $\bigl(H_p\bigr)_{p\in\N^\star}$, with $H_1\in(0,h_{\rm max})$, and a non-decreasing sequence $\bigl(C_p\bigr)_{p\in\N^\star}$, such that for any initial value $x_0\in\R^d$ and for any $p\in\N^\star$ one has
\begin{equation}\label{eq:momentbounds-scheme}
\underset{h\in(0,H_p)}\sup~\underset{n\in\N}\sup~\E[\|X_n-x^\star\|^{2p}]\le C_p(1+\|x_0-x^\star\|^{2p}).
\end{equation}
\end{lemma}

In the proof of Lemma~\ref{lem:momentbounds-scheme} given below, it will be obtained that $H_p$ tends to $0$ when $p$ tends to infinity. In the error analysis below, it will be sufficient to apply the result of Lemma~\ref{lem:momentbounds-scheme} with $p=3$, therefore the condition $h\in(0,H_3)$ on the time-step size will be sufficient.

\begin{rem}
Lemma~\ref{lem:momentbounds-scheme} is valid under the basic Assumptions~\ref{ass:F1} and~\ref{ass:sigma1} on the objective function $F$ and the diffusion coefficient $\sigma$, in other words it is not necessary to impose the strengthened Assumptions~\ref{ass:F2} and~\ref{ass:sigma2}.
\end{rem}

\begin{proof}
Let us first introduce some auxiliary notation.

For any positive integer $n\in\N^\star$, define the $\sigma$-field $\mathcal{F}_n=\sigma\bigl(\gamma_1,\ldots,\gamma_{n}\bigr)$. Set also $\mathcal{F}_0=\{\emptyset,\Omega\}$. Recall the notation $\E[\cdot|\mathcal{F}_n]$ for the conditional expectation.

For all $n\in\N$, set
\[
\Delta X_{n+1}=X_{n+1}-X_n=-h\nabla F(X_n)+h\sigma(t_n,X_n)\gamma_{n+1},
\]
Observe that the random variable $X_n$ is $\mathcal{F}_n$-measurable for all $n\in\N$. In addition, one has 
\begin{equation}\label{eq:auxscheme}
X_{n+1}-x^\star=X_n-x^\star+\Delta X_{n+1},\quad \forall~n\in\N.
\end{equation}

The proof then proceeds in two steps, in order to treat the cases $p=1$ and $p\ge 2$ separately.

\emph{Step 1: the case $p=1$.} Using the identity~\eqref{eq:auxscheme} above, one obtains for all $n\in\N$
\[
\|X_{n+1}-x^\star\|^2=\|X_n-x^\star\|^2+2\langle X_n-x^\star,\Delta X_{n+1}\rangle+\|\Delta X_{n+1}\|^2.
\]
Let $n\in\N$. Since the random variable $X_n$ is $\mathcal{F}_n$-measurable, one has
\[
\E[\|X_n-x^\star\|^2|\mathcal{F}_n]=\|X_n-x^\star\|^2.
\]
In addition, the random variable $\gamma_{n+1}$ is independent of $\mathcal{F}_n$, thus one has $$\E[\gamma_{n+1}|\mathcal{F}_n]=\E[\gamma_{n+1}]=0.$$ 
As a result, using the definition of $\Delta X_{n+1}$ and the properties of conditional expectation, one has
\begin{align*}
\E[\langle X_n-x^\star,\Delta X_{n+1}\rangle |\mathcal{F}_n]
&=-h\langle X_n-x^\star,\nabla F(X_n)\rangle+h\langle X_n-x^\star,\sigma(t_n,X_n)\E[\gamma_{n+1}|\mathcal{F}_n]\rangle\\
&=-h\langle X_n-x^\star,\nabla F(X_n)\rangle.
\end{align*}
Using the $\mu$-convexity of $F$ \eqref{eq:assF1-muconvex} in Assumption~\ref{ass:F1}, one obtains 
\[
\langle X_n-x^\star,\nabla F(X_n)\rangle=\langle X_n-x^\star,\nabla F(X_n)-\nabla F(x^\star)\rangle \geq \mu \| X_n-x^\star\|^2,
\]
since $\nabla F(x^\star)=0$ by the the first-order optimality condition. One obtains the 
the upper bound
\[
\E[\langle X_n-x^\star,\Delta X_{n+1}\rangle |\mathcal{F}_n]=-h\langle X_n-x^\star,\nabla F(X_n)\rangle\le -\mu h\|X_n-x^\star\|^2.
\]
Finally, using the definition of $\Delta X_{n+1}$ and similar arguments one has
\begin{align*}
\E[\|\Delta X_{n+1}\|^2|\mathcal{F}_n]&=h^2\E[\|\nabla F(X_n)\|^2|\mathcal{F}_n]+h^2\E[\|\sigma(t_n,X_n)\gamma_{n+1}\|^2|\mathcal{F}_n]\\
&\quad-2h^2\E[\langle \nabla F(X_n),\sigma(t_n,X_n)\gamma_{n+1}\rangle|\mathcal{F}_n]\\
&=h^2\|\nabla F(X_n)\|^2+h^2\|\sigma(t_n,X_n)\|^2.
\end{align*}
Recall that the mappings $\nabla F$ and $\sigma$ have linear growth (see~\eqref{eq:assF1-growth} and~\eqref{eq:asssigma1growth}), and that the mapping $\varsigma$ is bounded (by Assumption~\ref{ass:sigma1}). Then one obtains for all $n\in\N$ and all $h\in(0,h_{\rm max})$
\[
\E[\|\Delta X_{n+1}\|^2|\mathcal{F}_n]\le \LipF^2 h^2\|X_n-x^\star\|^2+h^2\varsigma(t_n)^2(1+\|X_n-x^\star\|)^2
\]
and gathering the upper bounds for $\E[\langle X_n-x^\star,\Delta X_{n+1}\rangle |\mathcal{F}_n]$ and $\E[\|\Delta X_{n+1}\|^2|\mathcal{F}_n]$ one has
\[
\E[\|X_{n+1}-x^\star\|^2|\mathcal{F}_n]\le \Bigl(1-2\mu h+\bigl(\LipF^2+2\|\varsigma\|_\infty^2\bigr)h^2\Bigr)\|X_{n}-x^\star\|^2+2h^2\|\varsigma\|_\infty^2.
\]
For all $h\in(0,h_{\rm max})$, set
\[
\rho_1(h)=1-2\mu h+\bigl(\LipF^2+2\|\varsigma\|_\infty^2\bigr)h^2.
\]
The tower property of conditional expectation yields the identity $$\E[\E[\|X_{n+1}-x^\star\|^2|\mathcal{F}_n]]=\E[\|X_{n+1}-x^\star\|^2],$$ 
thus taking expectation, for all $n\in\N$ one has
\[
\E[\|X_{n+1}-x^\star\|^2]\le \rho_1(h)\E[\|X_{n}-x^\star\|^2]+2h^2\|\varsigma\|_\infty^2.
\]
Set
\[
H_1=\frac{1}{2}\min\Bigl(\frac{1}{2\mu},\frac{2\mu}{\LipF^2+2\|\varsigma\|_\infty^2},h_{\rm max}\Bigr)
\]
and observe that one has $\rho_1(h)\in(0,1)$ for all $h\in(0,H_1)$. Applying the discrete Gr\"onwall inequality, one obtains for all $n\ge 0$ and for all $h\in(0,H_1)$
\begin{align*}
\E[\|X_{n}-x^\star\|^2]&\le \rho_1(h)^n\|x_0-x^\star\|^2+\frac{2h^2\|\varsigma\|_\infty^2}{1-\rho_1(h)}\\
&\le \|x_0-x^\star\|^2+\frac{2h\|\varsigma\|_{\infty}^2}{2\mu -\bigl(\LipF^2+2\|\varsigma\|_\infty^2\bigr)h}\\
&\le \|x_0-x^\star\|^2+\frac{2H_1\|\varsigma\|_{\infty}^2}{\mu}.
\end{align*}
This concludes the proof of the inequality~\eqref{eq:momentbounds-scheme} in the case $p=1$.

\emph{Step 2: the case $p\ge 2$.} 

We start by a useful elementary remark we shall use also in other proofs. 
\begin{fact}
Define the auxiliary mapping $\psi_p$ by $\psi_p(x)=\|x-x^\star\|^{2p}$. For all $x\in\R^d$ one has
\begin{eqnarray}
D\psi_p(x).k&=&2p\|x-x^\star\|^{2(p-1)}\langle x-x^\star,k\rangle, \forall~k\in\R^d, \label{gradient:norme^p}\\
|D^2\psi_p(x).(k_1,k_2)|&\le &2p(2p-1)\|x-x^\star\|^{2(p-1)}\|k_1\|\|k_2\|,\quad \forall~k_1,k_2\in\R^d.  \label{borne:hessienne:norme^p}
\end{eqnarray}
\end{fact}

Applying the Taylor formula, one obtains
\begin{align*}
\|X_{n+1}-x^\star\|^{2p}&=\psi_p(X_{n+1})=\psi_p(X_n+\Delta X_{n+1})\\
&=\psi_p(X_n)+D\psi_p(X_n).\Delta X_{n+1}\\
&\ \ \ \  +\int_{0}^{1}(1-\theta)D^2\psi_p(X_n+\theta \Delta X_{n+1}).(\Delta X_{n+1},\Delta X_{n+1}) \,d\theta\\
&\le\|X_n-x^\star\|^{2p}+2p\|X_n-x^\star\|^{2(p-1)}\langle X_n-x^\star,\Delta X_{n+1}\rangle\\
&+2p^2\bigl(\|X_n-x^\star\|+\|\Delta X_{n+1}\|\bigr)^{2(p-1)}\|\Delta X_{n+1}\|^2.
\end{align*}
We first consider the term with the inner product in the first line. Like in the treatment of the case $p=1$, considering the conditional expectation with respect to $\mathcal{F}_n$ one has
\begin{align*}
\E[2p\|X_n-x^\star\|^{2(p-1)}\langle X_n-x^\star,\Delta X_{n+1}\rangle|\mathcal{F}_n]&=2p\|X_n-x^\star\|^{2(p-1)}\E[\langle X_n-x^\star,\Delta X_{n+1}\rangle|\mathcal{F}_n]\\
&\le -2p\mu \|X_n-x^\star\|^{2p}.
\end{align*}
For the reminder term on the second line, the convexity inequality $(a+b)^q\leq 2^{q-1} (a^q+b^q)$ infers the upper bound
\[
\bigl(\|X_n-x^\star\|+\|\Delta X_{n+1}\|\bigr)^{2(p-1)}\|\Delta X_{n+1}\|^2\le 2^{2p-3}\bigl(\|X_n-x^\star\|^{2(p-1)}\|\Delta X_{n+1}\|^2+\|\Delta X_{n+1}\|^{2p}\bigr).
\]
Employing the linear growth conditions~\eqref{eq:assF1-growth} and~\eqref{eq:asssigma1growth} on $\nabla F$ and on $\sigma$ respectively, then yields the upper bound 
\[
\|\Delta X_{n+1}\|\le Ch(1+\|X_n-x^\star\|)(1+\|\gamma_{n+1}\|).
\]
As a result, one obtains the following upper bounds on conditional expectations: there exists a positive real number $C_p\in(0,\infty)$ such that
\begin{align*}
\E[\|X_n-x^\star\|^{2(p-1)}\|\Delta X_{n+1}\|^2|\mathcal{F}_n]&=\|X_n-x^\star\|^{2(p-1)}\E[\|\Delta X_{n+1}\|^2|\mathcal{F}_n]\\
&\le C_ph^2(1+\|X_n-x^\star\|^{2p})\\
\E[\|\Delta X_{n+1}\|^{2p}|\mathcal{F}_n]&\le C_ph^2(1+\|X_n-x^\star\|^{2p}).
\end{align*}
Combining the results above, for all $n\in\N$ one obtains
\[
\E[\|X_{n+1}-x^\star\|^{2p}|\mathcal{F}_n]\le \Bigl(1-2p\mu h+C_ph^2\Bigr)\|X_n-x^\star\|^{2p}+C_ph^2,
\]
For all $h\in(0,h_{\rm max})$ set
\[
\rho_p(h)=1-2p\mu h+C_ph^2.
\]
The tower property of conditional expectation yields the identity $$\E[\E[\|X_{n+1}-x^\star\|^{2p}|\mathcal{F}_n]]=\E[\|X_{n+1}-x^\star\|^{2p}],$$
thus taking expectation, for all $n\in\N$ one has
\[
\E[\|X_{n+1}-x^\star\|^{2p}]\le \rho_p(h)\E[\|X_n-x^\star\|^{2p}]+C_ph^2,
\]
Set
\[
H_p=\frac{1}{2}\min\Bigl(\frac{1}{2p\mu},\frac{2p\mu}{C_p},h_{\rm max}\Bigr),
\]
then for all $h\in(0,H_p)$ one has $\rho_p(h)\in(0,1)$. Applying the discrete Gr\"onwall inequality, one obtains for all $n\ge 0$ and for all $h\in(0,H_p)$
\begin{align*}
\E[\|X_{n}-x^\star\|^{2p}]&\le \rho_p(h)^n\|x_0-x^\star\|^{2p}+\frac{C_ph^2}{1-\rho_p(h)}\\
&\le \|x_0-x^\star\|^{2p}+\frac{C_ph}{2p\mu-C_ph}\\
&\le \|x_0-x^\star\|^{2p}+\frac{C_pH_p}{p\mu}.
\end{align*}
This concludes the proof of the inequality~\eqref{eq:momentbounds-scheme} in the case $p\ge 2$.
The proof of Lemma~\ref{lem:momentbounds-scheme} is thus completed.
\end{proof}

\subsection{Auxiliary process}

Given any time-step size $h\in(0,h_{\rm max})$, let us introduce an auxiliary stochastic process $\bigl(\tilde{X}^h(t)\bigr)_{t\ge 0}$ defined as follows: for any integer $n\ge 0$, and for all $t\in[t_n,t_{n+1}]$ one has
\begin{equation}\label{eq:tildeX}
\tilde{X}^h(t)=X_n-\nabla F(X_n)(t-t_n)+\sqrt{h}\sigma(t_n,X_n)\bigl(B(t)-B(t_n)\bigr).
\end{equation}
Note that $\tilde{X}^h(t_n)=X_n$ for all $n\in\N$, i.e. the auxiliary process interpolates the stochastic gradient optimization scheme~\eqref{eq:scheme} at the grid times $t_n=nh$. In addition, the stochastic process $\bigl(\tilde{X}^h(t)\bigr)_{t\ge 0}$ has continuous trajectories. Finally, for any integer $n\in\N$, on the interval $[t_n,t_{n+1}]$ the process $t\mapsto \tilde{X}^h(t)$ is the solution of the stochastic differential equation
\begin{equation}\label{eq:tildeX-SDE}
\left\lbrace
\begin{aligned}
&d\tilde{X}^h(t)=-\nabla F(X_n)dt+\sqrt{h}\sigma(t_n,X_n)dB(t),\quad t\in[t_n,t_{n+1}],\\
&\tilde{X}^h(t_n)=X_n.
\end{aligned}
\right.
\end{equation}

The auxiliary process $\bigl(\tilde{X}^h(t)\bigr)_{t\ge 0}$ plays a technical role in the proofs of the weak error estimates. The following properties, which follow from Lemma~\ref{lem:momentbounds-scheme}, are employed in the error analysis.
\begin{lemma}\label{lem:tildeX}
Assume that the objective $F$ satisfies the basic properties of Assumptions~\ref{ass:F1} and that the diffusion coefficient $\sigma$ satisfies Assumptions \ref{ass:sigma1}. For all $p\in\N^\star$, there exists $H_p\in(0,h_{\rm max})$ and $C_p\in(0,\infty)$ such that one has the moment bounds
\begin{equation}\label{eq:tildeX-momentbounds}
\underset{h\in(0,H_p)}\sup~\underset{t\ge 0}\sup~\E[\|\tilde{X}^h(t)-x^\star\|^{2p}]\le C_p(1+\|x_0-x^\star\|^{2p})
\end{equation}
and such that one has the increment bounds: for all $h\in(0,H_p)$,
\begin{equation}\label{eq:tildeX-increments}
\underset{n\in\N}\sup~\underset{t\in[t_n,t_{n+1}]}\sup~\E[\|\tilde{X}^h(t)-X_n\|^{2p}]\le C_p(1+\|x_0-x^\star\|^{2p})h^{2p}.
\end{equation}
\end{lemma}

\begin{proof}[Proof of Lemma~\ref{lem:tildeX}]
Let $p\in\N^\star$ and let $H_p$ be given by Lemma~\ref{lem:momentbounds-scheme}. 

\emph{Proof of the moments estimates \eqref{eq:tildeX-momentbounds}.} For all $h\in(0,H_p)$, for any integer $n\ge 0$, owing to the definition \eqref{eq:tildeX} of the interpolating process $\tilde{X}^h(t)$, one has for all $t\in[t_n,t_{n+1}]$
\begin{align*}
\bigl(\E[\|\tilde{X}^h(t)-x^\star\|^{2p}]\bigr)^{\frac{1}{2p}}&\le \bigl(\E[\|X_n-x^\star\|^{2p}]\bigr)^{\frac{1}{2p}}+(t-t_n)\bigl(\E[\|\nabla F(X_n)\|^{2p}]\bigr)^{\frac{1}{2p}}\\
& \ \ \ +\sqrt{h}\bigl(\E[\|\sigma(t_n,X_n)\bigl(B(t)-B(t_n)\bigr)\|^{2p}]\bigr)^{\frac{1}{2p}}.
\end{align*}
Then by the linear growth conditions~\eqref{eq:assF1-growth} and~\eqref{eq:asssigma1growth},
$$
\bigl(\E[\|\tilde{X}^h(t)-x^\star\|^{2p}]\bigr)^{\frac{1}{2p}}\le (1+\LipF h) \bigl(\E[\|X_n-x^\star\|^{2p}]\bigr)^{\frac{1}{2p}}\ +C_ph\|\varsigma\|_\infty\Bigl(1+\bigl(\E[\|X_n-x^\star\|^{2p}]\bigr)^{\frac{1}{2p}}\Bigr),
$$
using the property
\[
\bigl(\E[\|B(t)-B(t_n)\|^{2p}]\bigr)^{\frac{1}{2p}}=C_p(t-t_n)^{\frac12}\le C_ph^{\frac12}.
\]
The moment bounds~\eqref{eq:momentbounds-scheme} from Lemma~\ref{lem:momentbounds-scheme} then imply that for all $h\in(0,H_p)$ one has \begin{align*}
\underset{t\ge 0}\sup~\bigl(\E[\|\tilde{X}^h(t)-x^\star\|^{2p}]\bigr)^{\frac{1}{2p}}&=\underset{n\ge 0}\sup~\underset{t\in[t_n,t_{n+1}]}\sup~\bigl(\E[\|\tilde{X}^h(t)-x^\star\|^{2p}]\bigr)^{\frac{1}{2p}}\\
&\le C_p \Bigl(1+\underset{n\ge 0}\sup~\bigl(\E[\|X_n-x^\star\|^{2p}]\bigr)^{\frac{1}{2p}}\Bigr)\\
&\le C_p\Bigl(1+\|x_0-x^\star\|\Bigr).
\end{align*}
This concludes the proof of the moment bounds~\eqref{eq:tildeX-momentbounds}.

\emph{Proof of the increments bounds \eqref{eq:tildeX-increments}.} One proceeds by using similar arguments. Let $h\in(0,H_p)$ and $n\ge 0$. Then for all $t\in[t_n,t_{n+1}]$ owing to~\eqref{eq:tildeX} one has
\begin{align*}
\bigl(\E[\|\tilde{X}^h(t)-X_n\|^{2p}]\bigr)^{\frac{1}{2p}}&\le (t-t_n)\bigl(\E[\|\nabla F(X_n)\|^{2p}]\bigr)^{\frac{1}{2p}}+\sqrt{h}\bigl(\E[\|\sigma(t_n,X_n)\bigl(B(t)-B(t_n)\bigr)\|^{2p}]\bigr)^{\frac{1}{2p}}\\
&\le \LipF h\bigl(\E[\|X_n-x^\star\|^{2p}]\bigr)^{\frac{1}{2p}}+C_ph\|\varsigma\|_\infty\Bigl(1+\bigl(\E[\|X_n-x^\star\|^{2p}]\bigr)^{\frac{1}{2p}}\Bigr)\\
&\le C_ph \Bigl(1+\underset{n\ge 0}\sup~\bigl(\E[\|X_n-x^\star\|^{2p}]\bigr)^{\frac{1}{2p}}\Bigr)\\
&\le C_p h\Bigl(1+\|x_0-x^\star\|\Bigr).
\end{align*}
This concludes the proof of the increment bounds~\eqref{eq:tildeX-increments}. The proof of Lemma~\ref{lem:tildeX} is thus completed.
\end{proof}

\section{Proof of Theorem~\ref{theo:1}}\label{sec:proof1}

This section is devoted to the detailed proof of Theorem~\ref{theo:1}. In the sequel, it is assumed that the objective function $F$ satisfies the basic Assumption~\ref{ass:F1}, and that the diffusion coefficient $\sigma$ satisfies the basic Assumption~\ref{ass:sigma1}.

Let us explain the strategy of the proof. Let $\varphi:\R^d\to\R$ be a mapping of class $\mathcal{C}^2$. The error is decomposed as
\begin{equation}\label{eq:decomperror1}
\E[\varphi(X_N)]-\varphi(x^\star)=\varphi(\X^0(Nh))-\varphi(x^\star)+\E[\varphi(X_N)]-\varphi(\X^0(Nh)),
\end{equation}
where $\bigl(\X^0(t)\bigr)_{t\ge 0}$ is the solution of the deterministic gradient system~\eqref{eq:ODE1}. On the one hand, the treatment of the first error term $\varphi(\X^0(Nh))-\varphi(x^\star)$ follows from the exponential convergence property~\eqref{eq:cvX0} of $\X^0(T)-x^\star$ when $T\to\infty$, with $T=t_N=Nh$. On the other hand, the treatment of the second error term $\E[\varphi(X_N)]-\varphi(\X^0(Nh))$ requires more effort. Introduce the auxiliary function $u^0:[0,\infty)\times\R^d\to\R$ given by
\begin{equation}\label{eq:u}
u^0(t,x)=\varphi(\X_x^0(t)),
\end{equation}
where $\bigl(\X_x^0(t)\bigr)_{t\ge 0}$ denotes the solution of~\eqref{eq:ODE1} with arbitrary initial value $\X_x^0(0)=x\in\R^d$. As a result, the second error term is written as
\begin{equation}\label{eq:decomperror11}
\E[\varphi(X_N)]-\varphi(\X^0(Nh))=\E[u^0(0,X_N)]-\E[u^0(t_N,X_0)].
\end{equation}
Then a standard telescoping sum argument provides the decomposition of the error as
\[
\E[u^0(0,X_N)]-\E[u^0(t_N,X_0)]=\sum_{n=0}^{N-1}\bigl(\E[u^0(t_N-t_{n+1},X_{n+1})]-\E[u^0(t_N-t_n,X_n)]\bigr).
\]
Recalling that $X_n=\tilde{X}^h(t_n)$ for all $n\in\{0,\ldots,N\}$, where the auxiliary continuous time process $\bigl(\tilde{X}^h(t)\bigr)_{t\ge 0}$ is defined by~\eqref{eq:tildeX}, one obtains
\begin{equation}\label{eq:decomperror11epsilon}
\E[u^0(0,X_N)]-\E[u^0(t_N,X_0)]=\sum_{n=0}^{N-1}\epsilon_n,
\end{equation}
where for all $n\in\{0,\ldots,N-1\}$ the error term $\epsilon_n$ is defined by
\begin{equation}\label{eq:epsilon}
\epsilon_n=\E[u^0(t_N-t_{n+1},\tilde{X}^h(t_{n+1}))]-\E[u^0(t_N-t_n,\tilde{X}^h(t_n))].
\end{equation}
Then one proceeds by applying the It\^o formula, using~\eqref{eq:tildeX-SDE}. Error bounds for $\epsilon_n$ are obtained in the proof of Proposition~\ref{propo:schemeODE}. To analyze the error terms, it is fundamental to observe that the function $u^0$ is solution of the linear transport partial differential equation
\begin{equation}\label{eq:Kolmogorov_u}
\left\lbrace
\begin{aligned}
&\partial_tu^0(t,x)=-\langle \nabla u^0(t,x),\nabla F(x)\rangle,\quad \forall~t\ge 0, x\in\R^d,\\
&u^0(0,x)=\varphi(x),\quad \forall~x\in\R^d.
\end{aligned}
\right.
\end{equation}
The reminder of this section is organized as follows. Regularity properties of the function $u^0$ are stated in Section~\ref{sec:u0}, with proofs postponed to Appendix~\ref{app:lemu}. The application of those results to obtain upper bounds on the second error term appearing on the right-hand side of~\eqref{eq:decomperror1}, following the strategy outlined above, is then presented in Section~\ref{sec:error1}. Finally, Section~\ref{sec:conclusion-proof1} combines the upper bounds on the error terms and provides the proof of Theorem~\ref{theo:1}.

\subsection{Regularity properties of the auxiliary function $u^0$}\label{sec:u0}

\begin{lemma}\label{lem:u}
Assume that the objective function $F$ satisfies the basic properties from Assumptions~\ref{ass:F1}. Let $\varphi:\R^d\to\R$ be of class $\mathcal{C}^2$ with bounded second order derivative.

For all $t\ge 0$ the mapping $u^0(t,\cdot)$ defined by~\eqref{eq:u} is of class $\mathcal{C}^2$. Moreover,  there exists a positive real number $C^0\in(0,\infty)$ such that the first order derivative satisfies for all $t\ge 0$
\begin{equation}\label{eq:lemu1}
\underset{x\in\R^d}\sup~\underset{k\in\R^d\setminus\{0\}}\sup~\frac{|Du^0(t,x).k|}{(1+\|x-x^\star\|)\|k\|}\le C^0\bigl(\|\nabla \varphi(x^\star)\|+\vvvert\varphi\vvvert_2\bigr)e^{-\mu t},
\end{equation}
and such that the second order derivative satisfies for all $t\ge 0$
\begin{equation}\label{eq:lemu2}
\underset{x\in\R^d}\sup~\underset{k_1,k_2\in\R^d\setminus\{0\}}\sup~\frac{|D^2u^0(t,x).(k_1,k_2)|}{(1+\|x-x^\star\|)\|k_1\|\|k_2\|} \le C^0\bigl(\|\nabla \varphi(x^\star)\|+\vvvert\varphi\vvvert_2\bigr)e^{-\mu t}.
\end{equation}
\end{lemma}

The proof of Lemma~\ref{lem:u} is given in Appendix~\ref{app:lemu}.

\subsection{Proof of weak error estimates}\label{sec:error1}

We now focus on the main task of the proof of Theorem~\ref{theo:1}, which consists in proving error estimates for $\E[\varphi(X_N)]-\varphi(\X^0(Nh))$, uniformly with respect to $N$.

\begin{propo}\label{propo:schemeODE}
Assume that the objective function $F$ satisfies the basic Assumption~\ref{ass:F1}, and that the diffusion coefficient $\sigma$ satisfies the basic Assumption~\ref{ass:sigma1}.

There exist positive real numbers $H\in(0,h_{\max})$ and $C\in(0,\infty)$ such that for any mapping $\varphi:\R^d\to\R$ of class $\mathcal{C}^2$ with bounded second order derivative, for any initial value $x_0\in\R^d$, and for any $h\in(0,H)$, one has the weak error estimate
\begin{equation}\label{eq:proposchemeODE}
\underset{N\in\N^\star}\sup~\big|\E[\varphi(X_N)]-\varphi(\X^0(Nh))\big|\le  C_1(\varphi,x_0)h,
\end{equation}
where the positive real number $C_1(\varphi,x_0)$ is given by~\eqref{eq:constantC1}.
\end{propo}

In order to prove Proposition~\ref{propo:schemeODE}, we follow a standard approach employed to prove weak error estimates for numerical methods applied to stochastic differential equations, sketched at the beginning of Section~\ref{sec:proof1} above. Having exponential decrease with respect to $t$ in the bounds~\eqref{eq:lemu1} and~\eqref{eq:lemu2} of the first and second derivatives with respect to $x$ of the solution of the partial differential equation~\eqref{eq:u} is fundamental in order to obtain a right-hand side which is independent of $N$. The value of $H$ is imposed to apply some moment bounds on the solution of the stochastic optimization gradient scheme $\bigl(X_n\bigr)_{n\ge 0}$ given by Lemma~\ref{lem:momentbounds-scheme} and on the solution $\bigl(\tilde{X}^h(t)\bigr)_{t\ge 0}$ given by Lemma~\ref{lem:tildeX}.

\begin{proof}[Proof of Proposition~\ref{propo:schemeODE}]

Let the function $\varphi:\R^d\to\R$ and the initial value $x_0\in\R^d$ be given. Combining the expressions~\eqref{eq:decomperror11} and~\eqref{eq:decomperror11epsilon} for the weak error $\E[\varphi(X_N)]-\varphi(\X^0(Nh))$, it is sufficient to prove the following claim for the error terms $\epsilon_n$ defined by~\eqref{eq:epsilon}: for all $N\in\N^\star$ and $h\in(0,H)$, for all $n\in\{0,\ldots,N-1\}$, one has
\begin{equation}\label{eq:epsilon-claim}
|\epsilon_{n}|\le C_1(\varphi,x_0) h \int_{t_n}^{t_{n+1}}e^{-\mu(t_N-t)}\,dt,
\end{equation}
with $C_1(\varphi,x_0)$ defined by~\eqref{eq:constantC1} (up to a modification of the multiplicative constant $C\in(0,\infty)$).

Let us prove the claim~\eqref{eq:epsilon-claim} above. For all $N\in\N^\star$ and $n\in\{0,\ldots,N-1\}$, recall that $t\in[t_n,t_{n+1}]\mapsto \tilde{X}^{h}(t)$ is solution of the stochastic differential equation~\eqref{eq:tildeX-SDE}. Therefore, applying It\^o's formula on the interval $[t_n,t_{n+1}]$ yields the following decomposition of the error term $\epsilon_n$: one has
\begin{align*}
\epsilon_n&=\int_{t_n}^{t_{n+1}}\E[-\partial_t u^0(t_N-t,\tilde{X}^{h}(t))]dt-\int_{t_n}^{t_{n+1}}\E[\langle \nabla u^0(t_N-t,\tilde{X}^h(t)),\nabla F(X_n)\rangle]\,dt\\
&\ \ \ +\frac{h}{2}\int_{t_n}^{t_{n+1}}\E[\nabla^2 u^0(t_N-t,\tilde{X}^h(t)):a(t_n,\tilde{X}^h(t_n))]\,dt.
\end{align*}
Recall that the function $u^0$ is the solution of the partial differential equation~\eqref{eq:Kolmogorov_u}. This allows us to decompose $\epsilon_n$ as
\[
\epsilon_n=\epsilon_{n,1}+\epsilon_{n,2},
\]
where for all $n\in\{0,\ldots,N-1\}$ one defines
\begin{align*}
\epsilon_{n,1}&=\int_{t_n}^{t_{n+1}}\E[\langle \nabla u^0(t_N-t,\tilde{X}^h(t)),\nabla F(\tilde{X}^h(t))-\nabla F(X_n)\rangle]\,dt\\
\epsilon_{n,2}&=\frac{h}{2}\int_{t_n}^{t_{n+1}}\E[\nabla^2 u^0(t_N-t,\tilde{X}^h(t)):a(t_n,\tilde{X}^h(t_n))]\,dt.
\end{align*}
The claim~\eqref{eq:epsilon-claim} follows from the bounds for $\epsilon_{n,1}$ and $\epsilon_{n,2}$ which are obtained below.

\emph{Upper bounds for $\epsilon_{n,1}$.}
Recall that $\nabla F:\R^d\to\R^d$ is a Lipschitz continuous mapping (see Equation~\eqref{eq:assF1-Lip} from Assumption~\ref{ass:F1}). Therefore, using the inequality~\eqref{eq:lemu1} from Lemma~\ref{lem:u} on the first order derivative of $u(t_N-t,\cdot)$, one has
\[
|\epsilon_{n,1}|\le C(\|\nabla\varphi(x^\star)\|+\vvvert\varphi\vvvert_2)\int_{t_n}^{t_{n+1}}e^{-\mu(t_N-t)}\E[(1+\|\tilde{X}^h(t)-x^\star\|)\|\tilde{X}^h(t)-X_n\|]\,dt.
\]
Assume that $H\le H_1$ where $H_1\in(0,h_{\rm max})$ is given by Lemma~\ref{lem:tildeX}. Applying the Cauchy--Schwarz inequality and using the moment bounds~\eqref{eq:tildeX-momentbounds} and the increment bounds~\eqref{eq:tildeX-increments} from Lemma~\ref{lem:tildeX} for the auxiliary process $\bigl(\tilde{X}^h(t)\bigr)_{t\ge 0}$, one obtains the upper bounds
\begin{align*}
|\epsilon_{n,1}|&\le C(\|\nabla\varphi(x^\star)\|+\vvvert\varphi\vvvert_2) h (1+\|x_0-x^\star\|^2)\int_{t_n}^{t_{n+1}}e^{-\mu(t_N-t)}\,dt\\
&\le C_1(\varphi,x_0) h \int_{t_n}^{t_{n+1}}e^{-\mu(t_N-t)}\,dt,
\end{align*}
for all $N\in\N^\star$, all $n\in\{0,\ldots,N-1\}$ and all $h\in(0,H)$.

\emph{Upper bounds for $\epsilon_{n,2}$.}

Recall that the mapping $(t,x)\mapsto \sigma(t,x)$ has at most linear growth with respect to $x\in\R^d$, uniformly with respect to $t\ge 0$ (see Equation~\eqref{eq:asssigma1growth} from Assumption~\ref{ass:sigma1}). Recall also that $a$ is defined by~\eqref{eq:a}. Therefore, using the inequality~\eqref{eq:lemu2} from Lemma~\ref{lem:u} on the second order derivative of $u(t_N-t,\cdot)$, one has
\[
|\epsilon_{n,2}|\le C(\|\nabla\varphi(x^\star)\|+\vvvert\varphi\vvvert_2) h \varsigma(t_n)^2\int_{t_n}^{t_{n+1}} e^{-\mu(t_N-t)}\E[\bigl(1+\|\tilde{X}^h(t)-x^\star\|\bigr)\bigl(1+\|\tilde{X}^h(t_n)-x^\star\|\bigr)^2]\,dt.
\]
Recall that $\varsigma(t_n)\le \|\varsigma\|_\infty$ for all $n\in\N$. Assume that $H\le H_2$ where $H_2\in(0,h_{\rm max})$ is given by Lemma~\ref{lem:tildeX}. Applying the moment bounds~\eqref{eq:tildeX-momentbounds} from Lemma~\ref{lem:tildeX}, one obtains the upper bounds
\begin{align*}
|\epsilon_{n,2}|&\le C(\|\nabla\varphi(x^\star)\|+\vvvert\varphi\vvvert_2) h (1+\|x_0-x^\star\|^3)\int_{t_n}^{t_{n+1}}e^{-\mu(t_N-t)}\,dt\\
&\le C_1(\varphi,x_0) h \int_{t_n}^{t_{n+1}}e^{-\mu(t_N-t)}\,dt,
\end{align*}
for all $N\in\N^\star$, all $n\in\{0,\ldots,N-1\}$ and all $h\in(0,H)$.

\emph{Conclusion.}
Gathering the upper bounds on $|\epsilon_{n,1}|$ and $|\epsilon_{n,2}|$ obtained above then yields the claim~\eqref{eq:epsilon-claim}, under the condition $H\le H_2$.

The proof of the weak error estimate~\eqref{eq:proposchemeODE} is then straightforward: for all $N\in\N^\star$ one has
\begin{align*}
\big|\E[\varphi(X_N)]-\varphi(\X^0(Nh))\big|&\le \sum_{n=0}^{N-1}|\epsilon_n|\\
&\le C_1(\varphi,x_0) h\sum_{n=0}^{N-1} \int_{t_n}^{t_{n+1}}e^{-\mu(t_N-t)}\,dt\\
&\le C_1(\varphi,x_0)h\int_{0}^{t_{N}}e^{-\mu(t_N-t)}\,dt\le \frac{C_1(\varphi,x_0)h}{\mu}.
\end{align*}
This concludes the proof of Proposition~\ref{propo:schemeODE}.
\end{proof}

\subsection{Proof of Theorem~\ref{theo:1}}\label{sec:conclusion-proof1}

We are now in a position to provide the proof of Theorem~\ref{theo:1}.

\begin{proof}[Proof of Theorem~\ref{theo:1}]
First, Let us establish the weak error estimates~\eqref{eq:theo1-weak}. Recall that owing to~\eqref{eq:decomperror1} the error term appearing in the left-hand side of~\eqref{eq:theo1-weak} can be decomposed as
\[
\E[\varphi(X_N)]-\varphi(x^\star)=\varphi(\X^0(Nh))-\varphi(x^\star)+\E[\varphi(X_N)]-\varphi(\X^0(Nh)).
\]
On the one hand, using the inequality~\eqref{eq:ineqDvarphi} and applying the upper bound~\eqref{eq:cvX0}, one has
\begin{align*}
\big|\varphi(\X^0(Nh))-\varphi(x^\star)\big|&\le C\bigl(\|\nabla\varphi(x^\star)\|+\vvvert\varphi\vvvert_2\|\X^0(Nh)-x^\star\|\bigr) \|\X^0(Nh)-x^\star\|\\
&\le C_1(\varphi,x_0)e^{-\mu Nh}.
\end{align*}
On the other hand, applying the uniform in time weak error estimates~\eqref{eq:proposchemeODE} from Proposition~\ref{propo:schemeODE}, one has for all $N\in\N$ and all $h\in(0,H)$
\[
\big|\E[\varphi(X_N)]-\varphi(\X^0(Nh))\big|\le C_1(\varphi,x_0)h,
\]
where $C_1(\varphi,x_0)$ is given by~\eqref{eq:constantC1}. Gathering the two upper bounds yields the weak error estimates~\eqref{eq:theo1-weak}.

Second, let us establish the error estimates on the expected residual~\eqref{eq:theo1-residual}. Using the decomposition~\eqref{eq:decomperror1} with $\varphi=F$, one has
\[
\E[F(X_N)]-F(x^\star)=F(\X^0(Nh))-F(x^\star)+\E[F(X_N)]-F(\X^0(Nh)).
\]
On the one hand, using the inequality~\eqref{eq:residual_up} and applying the upper bound~\eqref{eq:cvX0}, one has
\[
F(\X^0(Nh))-F(x^\star)\le C\|\X^0(Nh)-x^\star\|^2\le Ce^{-2\mu Nh}\|x_0-x^\star\|^2
\]
On the other hand, applying the uniform in time weak error estimates~\eqref{eq:proposchemeODE} from Proposition~\ref{propo:schemeODE} with $\varphi=F$, one has for all $N\in\N$ and all $h\in(0,H)$
\[
\big|\E[F(X_N)]-F(\X^0(Nh))\big|\le C\bigl(1+\|x_0-x^\star\|^3\bigr)h.
\]
Gathering the two upper bounds yields the error estimates on the expected residual~\eqref{eq:theo1-residual}.

Finally, let us establish the strong error estimates~\eqref{eq:theo1-strong}. Owing to the inequality~\eqref{eq:residual_low}, one obtains
\[
\E[\|X_N-x^\star\|^2]\le\frac{2}{\mu}\bigl(\E[F(X_N)]-F(x^\star)\bigr)
\]
and it then suffices to apply the error estimates on the expected residual~\eqref{eq:theo1-residual} to get~\eqref{eq:theo1-strong}.

The proof of Theorem~\ref{theo:1} is thus completed.
\end{proof}

\section{Proof of Theorem~\ref{theo:2}}\label{sec:proof2}

This section is devoted to the detailed proof of Theorem~\ref{theo:2}. In the sequel, it is assumed that the objective function $F$ satisfies the strengthened Assumption~\ref{ass:F2}, and that the diffusion coefficient $\sigma$ satisfies the strengthened Assumption~\ref{ass:sigma2}. Recall that this implies that the basic Assumptions~\ref{ass:F1} and~\ref{ass:sigma1} are also satisfied.

Let us explain the strategy of the proof, and compare it with the proof of Theorem~\ref{theo:1} given in Section~\ref{sec:proof1}. Let $\varphi:\R^d\to\R$ be a mapping of class $\mathcal{C}^3$. The error is decomposed as
\begin{equation}\label{eq:decomperror2}
\E[\varphi(X_N)]-\varphi(x^\star)=\E[\varphi(\Y^h(Nh))]-\varphi(x^\star)+\E[\varphi(X_N)]-\E[\varphi(\Y^h(Nh))],
\end{equation}
where $\bigl(\Y^h(t)\bigr)_{t\ge 0}$ is the solution of the modified stochastic differential equation~\eqref{eq:SDE2}. On the one hand, the treatment of the first error term $\E[\varphi(\Y^h(Nh))]-\varphi(x^\star)$ follows from an upper bound on $\E[\|\Y^h(t)-x^\star\|^2]$ which is stated in Proposition~\ref{propo:YT} below. On the other hand, the approach used to treat the second error term $\E[\varphi(X_N)]-\E[\varphi(\Y^h(Nh))]$ follows arguments similar to those used in Section~\ref{sec:proof1}, where an auxiliary function $u^0$ has been introduced. The situation is more complex here: since the diffusion coefficient $\sigma$ depends on time in general, considering the stochastic differential equation~\eqref{eq:SDE2} on the time interval $[0,\infty)$, with initial value imposed at time $t=0$, is not sufficient. Instead one needs to consider the stochastic differential equation
\begin{equation}\label{eq:SDE2v2}
\left\lbrace
\begin{aligned}
&d\Y_y^h(s|t)=-\nabla F^h(\Y_y^h(s|t))ds+\sqrt{h}\sigma(s,\Y_y^h(s|t))dB(s),\quad \forall~s\ge t,\\
&\Y_y^h(t|t)=y,
\end{aligned}
\right.
\end{equation}
on time intervals $[t,\infty)$, with initial value imposed at time $s=t$, for arbitrary $t\in[0,\infty)$. The time variable in~\eqref{eq:SDE2v2} and below is denoted by $s$. The stochastic differential equation~\eqref{eq:SDE2} is retrieved by choosing $t=0$. Recall that the modified objective function $F^h$ is assumed to be $\mu$-convex for all $h\in(0,h_{0})$ (see Assumption~\ref{ass:F2}), and that the diffusion coefficient $\sigma(s,\cdot)$ is globally Lipschitz continuous, uniformly with respect to $s\in[0,\infty)$ (see Assumption~\ref{ass:sigma1}). As a result, the stochastic differential equation~\eqref{eq:SDE2v2} is globally well-posed: for any time $t\in[0,\infty)$ and any initial condition $y\in\R^d$, there exists a unique solution of~\eqref{eq:SDE2v2} denoted by $\bigl(\Y_y^h(s|t)\bigr)_{s\ge t}$. Note that choosing $t=0$ and $y=x_0$ in~\eqref{eq:SDE2v2}, one obtains $\Y^h(T)=\Y_{x_0}^h(T|0)$ for all $T\in[0,\infty)$. Moment bounds for the solutions $\bigl(\Y_y^h(s|t)\bigr)_{s\ge t}$ of~\eqref{eq:SDE2v2} are given in Lemma~\ref{lem:momentboundsYh} below.

Introducing the auxiliary stochastic differential equations~\eqref{eq:SDE2v2} indexed by the initial time $t\in[0,\infty)$ allows us to define, for any time $T\in[0,\infty)$, an auxiliary mapping $v_T^h:[0,T]\times\R^d\to\R$, as follows: for all $t\in[0,T]$ and all $y\in \R^d$, set
\begin{equation}\label{eq:vh}
v_T^h(t,y)=\E[\varphi(\Y_y^h(T|t))].
\end{equation}
The mapping $v_T^h$ is the solution of the backward Kolmorogov equation
\begin{equation}\label{eq:Kolmogorov_vh}
\left\lbrace
\begin{aligned}
&\partial_tv_T^h(t,y)-\langle \nabla v_T^h(t,y),\nabla F^h(y)\rangle+\frac{h}{2}\nabla^2 v_T^h(t,y):a(t,y)=0,\quad \forall~t\in [0,T], y\in\R^d,\\
&v_T^h(T,y)=\varphi(y),\quad \forall~y\in\R^d.
\end{aligned}
\right.
\end{equation}

The following remark justifies why it is necessary to introduce~\eqref{eq:SDE2v2} instead of considering only~\eqref{eq:SDE2}.
\begin{rem}
Let us emphasize that the mapping $u^h:[0,\infty)\times\R^d\to\R$ defined by
\[
u^h(t,y)=\E[\varphi(\Y_y^h(t))]
\]
for all $t\in[0,\infty)$ and $y\in\R^d$, where $\bigl(\Y_y^h(t)\bigr)_{t\ge 0}$ denotes the solution of~\eqref{eq:SDE2} with arbitrary initial value $\Y_y^h(0)=y\in\R^d$, in general does not solve the partial differential equation
\[
\left\lbrace
\begin{aligned}
&\partial_tu^h(t,y)=-\langle \nabla u^h(t,y) , \nabla F^h(y)\rangle+\frac{h}{2}\nabla^2 u^h(t,y):a(t,y),\quad \forall~t\ge 0, y\in\R^d,\\
&u^h(0,y)=\varphi(y),\quad \forall~y\in\R^d.
\end{aligned}
\right.
\]
This would be the case only if the diffusion coefficient $\sigma$ and the auxiliary mapping $a$ would be independent of time $t$.
\end{rem}

The auxiliary functions $v_T^h$ play a role in the analysis of the second error term $\E[\varphi(X_N)]-\E[\varphi(\Y^h(Nh))]$ which is similar to the role played by the auxiliary mapping $u^0$ in Section~\ref{sec:proof1}. Assume that $T=Nh$, then owing to the definition~\eqref{eq:vh} of $v_T^h$ the second error term is written as
\begin{equation}\label{eq:decomperror21}
\E[\varphi(X_N)]-\E[\varphi(\Y^h(Nh))]=\E[v_T^h(T,X_N)]-\E[v_T^h(0,X_0)].
\end{equation}
Then a standard telescoping sum argument provides the decomposition of the error as
\[
\E[v_T^h(T,X_N)]-\E[v_T^h(0,X_0)]=\sum_{n=0}^{N-1}\bigl(\E[v_T^h(t_{n+1},X_{n+1})]-\E[v_T^h(t_n,X_{n})]\bigr).
\]
As in Section~\ref{sec:proof1}, recalling that $X_n=\tilde{X}^h(t_n)$ for all $n\in\{0,\ldots,N\}$, where the auxiliary continuous time process $\bigl(\tilde{X}^h(t)\bigr)_{t\ge 0}$ is defined by~\eqref{eq:tildeX}, one obtains
\begin{equation}\label{eq:decomperror21delta}
\E[v_T^h(T,X_N)]-\E[v_T^h(0,X_0)]=\sum_{n=0}^{N-1}\delta_n
\end{equation}
where for all $n\in\{0,\ldots,N-1\}$ the error term $\delta_n$ is defined by
\begin{equation}\label{eq:delta}
\delta_n=\E[v_T^h(t_{n+1},\tilde{X}^h(t_{n+1}))]-\E[v_T^h(t_n,\tilde{X}^h(t_n))].
\end{equation}
Then one proceeds by applying the It\^o formula, using~\eqref{eq:tildeX-SDE}. Compared with Section~\ref{sec:proof1}, one needs to bound the error terms $\delta_n$ as ${\rm O}(h^3)$ instead of ${\rm O}(h^2)$ in order to obtain second order convergence with respect to $h$. This requires different arguments, however similar techniques are used to obtain error bounds which are uniform with respect to time.

The reminder of this section is organized as follows. Moment bounds for the solutions $\bigl(\Y_y^h(s|t)\bigr)_{s\ge t}$ are first studied in Section~\ref{sec:momentboundsYh}. The large time behavior of the solution $\Y^h(T)=\Y_{x_0}^{h}(T|0)$ when $T\to\infty$ is then studied in Section~\ref{sec:largetimeY}, this is the key ingredient to treat the first error term appearing in the right-hand side of~\eqref{eq:decomperror2}. Regularity properties of the functions $v_T^h$ are stated in Section~\ref{sec:vh}, with proofs postponed to Appendix~\ref{app:lemv}. The application of those results to obtain upper bounds on the second error term appearing on the right-hand side of~\eqref{eq:decomperror2}, following the strategy outlined above, is then presented in Section~\ref{sec:error2}. Finally, Section~\ref{sec:conclusion-proof2} combines the upper bounds on the error terms and provides the proof of Theorem~\ref{theo:2}.

\subsection{Moment bounds for the solutions of~\eqref{eq:SDE2v2}}\label{sec:momentboundsYh}

Let us first provide moment bounds for the solutions $\bigl(\Y_y^h(s|t)\bigr)_{s\ge t}$ of the stochastic differential equations~\eqref{eq:SDE2v2} indexed by time $t\in[0,\infty)$ and the time-step size $h\in(0,h_0)$. It is worth mentioning that the moment bounds are uniform with respect to $s\ge t\ge 0$ and with respect to $h\in(0,h_0)$.

\begin{lemma}\label{lem:momentboundsYh}
Assume that the objective $F$ satisfies the strengthened Assumption~\ref{ass:F2} and that the diffusion coefficient $\sigma$ satisfies the basic Assumption~\ref{ass:sigma1}. There exists a non-increasing sequence $\bigl(h_p^{(0)}\bigr)_{p\in\N^\star}$, with $h_1^{(0)}\in(0,h_0)$, such that for any $p\in\N^\star$, one has 
\begin{equation}\label{eq:momentboundsYh}
\underset{h\in(0,h_p^{(0)})}\sup~\underset{s\ge t}\sup~\underset{y\in\R^d}\sup~\frac{\E[\|\Y_y^h(s|t)-x^\star\|^{2p}]}{1+\|y-x^\star\|^{2p}}<\infty.
\end{equation}
\end{lemma}

\begin{proof}
Assume that $h\in(0,h_0)$ and let $p\in\N^\star$. Let the initial time $t\in[0,\infty)$ and the initial value $y\in\R^d$ be given. Recall that the diffusion coefficient $\sigma(s,\cdot)$ has at most linear growth, uniformly with respect to $s\in[0,\infty)$, see the condition~\eqref{eq:asssigma1growth} from Assumption~\ref{ass:sigma1}. Using the auxiliary results~\eqref{gradient:norme^p} and~\eqref{borne:hessienne:norme^p} (stated in the proof of Lemma~\ref{lem:momentbounds-scheme}) and applying It\^o's formula, there exists a positive real number $C_p^{(0)}\in(0,\infty)$ such that for all $s\ge t$ one has
\begin{align*}
\frac{1}{2p}\frac{d\E[\|\Y_{y}^{h}(s|t)-x^\star\|^{2p}]}{ds}&\leq -\E[\langle \nabla F^h(\Y_{y}^{h}(s|t)),\Y_{y}^{h}(s|t)-x^\star\rangle \|\Y_{y}^{h}(s|t)-x^\star\|^{2(p-1)}]\\
&+C_p^{(0)}h\E[(1+\|\Y_{y}^{h}(s|t)-x^\star\|^2)\|\Y_{y}^{h}(s|t)-x^\star\|^{2(p-1)}].
\end{align*}
The modified objective function $F^h$ is $\mu$-convex for all $h\in(0,h_0)$, see the condition~\eqref{eq:assF2-muconvex} from Assumption~\ref{ass:F2}. Moreover, one has $\nabla F^h(x^\star)=0$ for all $h\in(0,h_0)$. As a consequence, one obtains, for all $s\ge t$,
\[
\langle \nabla F^h(\Y_{y}^{h}(s|t)),\Y_{y}^{h}(s|t)-x^\star\rangle \le -\mu \|\Y_{y}^{h}(s|t)-x^\star\|^2.
\]
Applying Young's inequality, there exists a positive real number $C_p\in(0,\infty)$ such that one has, for all $s\ge t$,
\[
\frac{1}{2p}\frac{d\E[\|\Y_{y}^{h}(s|t)-x^\star\|^{2p}]}{ds}\le -(\mu-C_p^{(0)}h)\E[\|\Y_{y}^{h}(s|t)-x^\star\|^{2p}]+C_ph.
\]
Set $h_p^{(0)}=\frac12\min\left(\frac{\mu}{C_p^{(0)}},h_0\right)$. For all $h\in(0,h_p^{(0)})$, one has $\mu-C_p^{(0)}h\ge \mu-C_p^{(0)}h_p^{(0)}>0$. 
Applying Gr\"onwall's inequality and recalling that $\Y_y^h(t|t)=y$, one obtains the following upper bound: for all $s\ge t$ and all $h\in(0,h_p^{(0)})$ one has
\[
\E[\|\Y_{y}^{h}(s|t)-x^\star\|^{2p}]\le e^{-(\mu-C_p^{(0)}h_p^{(0)})(s-t)}\|y-x^\star\|^{2p}+\frac{C_p h}{\mu-C_p^{(0)}h_p^{(0)}}.
\]
Note that $h\in(0,h_p^{(0)})$ and that $e^{-(\mu-C_p^{(0)}h_p^{(0)})(s-t)}\le 1$. Then one obtains the moment bounds~\eqref{eq:momentboundsYh}, and this concludes the proof of Lemma~\ref{lem:momentboundsYh}.
\end{proof}

\subsection{Large time behavior of the solution of~\eqref{eq:SDE2}}\label{sec:largetimeY}

\begin{propo}\label{propo:YT}
Assume that the objective function $F$ satisfies the strengthened Assumption~\ref{ass:F2} and that the diffusion coefficient $\sigma$ satisfies the basic Assumption~\ref{ass:sigma1}.

There exists a positive real number $C\in(0,\infty)$ such that for any initial value $x_0\in\R^d$, any $h\in(0,h_1^{(0)})$ and any $T\in(0,\infty)$, one has
\begin{equation}\label{eq:YT}
\E[\|\Y^h(T)-x^\star\|^{2}]\le C e^{-2\mu T}\left(\|x_0-x^\star\|^2+h\rho(T) \bigl(1+\|x_0-x^\star\|^2\bigr) \right),
\end{equation}
where we recall that $\bigl(\Y^h(t)\bigr)_{t\ge 0}$ denotes the solution of the stochastic differential equation~\eqref{eq:SDE2}, with initial value $\Y^h(0)=x_0$, and that $\rho(T)$ is defined by~\eqref{eq:rhoT}.
\end{propo}

\begin{proof}
Applying It\^o's formula to the stochastic differential equation~\eqref{eq:SDE2}, one obtains for all $t\ge 0$
\begin{align*}
\frac12\frac{d\E[\|\Y^h(t)-x^\star\|^2]}{dt}&=-\E[\langle \nabla F^h(\Y^h(t)),\Y^h(t)-x^\star\rangle]+\frac{h}{2}\E[\|\sigma(t,\Y^h(t))\|^2]\\
&\le -\mu\E[\|\Y^h(t)-x^\star\|^2]+\frac{h}{2}\varsigma(t)^2\bigl(1+\E[\|\Y^h(t)-x^\star\|^2]\bigr),
\end{align*}
using the $\mu$-convexity property~\eqref{eq:assF2-muconvex} of the modified objective function $F^h$, and the at most linear growth property~\eqref{eq:asssigma1growth} of the diffusion coefficient $\sigma(t,\cdot)$, for all $t\ge 0$.

Applying the moment bounds stated in Lemma~\ref{lem:momentboundsYh}, there exists a positive real number $C\in(0,\infty)$ such that, for all $h\in(0,h_1^{(0)})$ and all $t\ge 0$ one obtains the inequality
\[
\frac12\frac{d\E[\|\Y^h(t)-x^\star\|^2]}{dt}\le -\mu\E[\|\Y^h(t)-x^\star\|^2]+Ch\varsigma(t)^2\bigl(1+\|x_0-x^\star\|^2\bigr).
\]
Applying Gr\"onwall's inequality, as a result one has for all $h\in(0,h_1^{(0)})$ and all $t\ge 0$
\[
\E[\|\Y^h(t)-x^\star\|^2]\le e^{-2\mu t}\|x_0-x^\star\|^2+Ch\int_0^t e^{-2\mu(t-r)}\varsigma(r)^2\,dr \bigl(1+\|x_0-x^\star\|^2\bigr).
\]
This yields the inequality~\eqref{eq:YT} and concludes the proof of Proposition~\eqref{propo:YT}.
\end{proof}

\subsection{Regularity properties of the auxiliary functions $v_T^h$}\label{sec:vh}

\begin{lemma}\label{lem:v} 
Assume that the objective function $F$ satisfies the strengthened Assumption~\ref{ass:F2}, and that the diffusion coefficient $\sigma$ satisfies the strengthened Assumption~\ref{ass:sigma2}.

Let $\varphi:\R^d\to\R$ be a mapping of class $\mathcal{C}^3$, with bounded second and third order derivatives. Given any $T\in(0,\infty)$, let the mapping $v_T^h:[0,T]\times\R^d\to\R$ be defined by~\eqref{eq:vh}.

For all $t\in[0,T]$ the mapping $v_T^h(t,\cdot)$ is of class $\mathcal{C}^3$. Moreover, the first, second and third order derivatives of $v_T^h(t,\cdot)$ satisfy the following upper bounds.

\begin{itemize}
    \item There exist positive real numbers $h_1\in(0,h_0)$ and $C_1\in(0,\infty)$, satisfying the condition $C_1h_1<\mu$, such that the first order derivative satisfies: for all $h\in(0,h_1)$ and all $T\ge t\ge 0$, one has
\begin{equation}\label{eq:lemv1}
\underset{y\in\R^d}{\sup}~\underset{k\in\R^d\setminus\{0\}}{\sup}~\frac{|Dv_T^h(t,y).k|}{(1+\|y-x^\star\|)\|k\|}\le C_1(\|\nabla\varphi(x^\star)\|+\vvvert \varphi \vvvert_2)\ e^{-(\mu-C_1h)(T-t)}.
\end{equation}

\item For any $\lambda\in(0,\mu)$, there exist positive real numbers $h_2\in(0,h_1)$ and $C_{2,\lambda}\in(C_1,\infty)$, satisfying the condition $C_{2,\lambda}h_2<\lambda$, such that the second order derivative satisfies: for all $h\in(0,h_2)$ and all $T\ge t\ge 0$, one has
\begin{equation}\label{eq:lemv2}
\underset{y\in\R^d}{\sup}~\underset{k_1,k_2\in\R^d\setminus\{0\}}{\sup}~\frac{|D^2v_T^h(t,y).(k_1,k_2)|}{(1+\|y-x^\star\|^2)\|k_1\|\|k_2\|}\le C_{2,\lambda}(\|\nabla\varphi(x^\star)\|+\vvvert \varphi \vvvert_2)e^{-(\lambda-C_{2,\lambda}h)(T-t)}.
\end{equation}

\item For any $\lambda\in(0,\mu)$, there exist positive real numbers $h_3\in(0,h_2)$ and $C_{3,\lambda}\in(C_{2,\lambda},\infty)$, satisfying the condition $C_{3,\lambda}h_3<\lambda$, such that the second order derivative satisfies: for all $h\in(0,h_3)$ and all $T\ge t\ge 0$, one has
\begin{equation}\label{eq:lemv3}
\underset{y\in\R^d}{\sup}~\underset{k_1,k_2,k_3\in\R^d\setminus\{0\}}{\sup}~\frac{|D^3v_T^h(t,y).(k_1,k_2,k_3)|}{(1+\|y-x^\star\|^2)\|k_1\|\|k_2\|\|k_3\|}\le C_{3,\lambda}(\|\nabla\varphi(x^\star)\|+\vvvert \varphi \vvvert_2+\vvvert \varphi \vvvert_3)e^{-(\lambda-C_{3,\lambda}h)(T-t)}.
\end{equation}
\end{itemize}
\end{lemma}

The proof of Lemma~\ref{lem:v} is postponed to Appendix~\ref{app:lemv}.

A straightforward consequence of Lemma~\ref{lem:v} is Lemma~\ref{lem:vt} stated below, giving bounds on the first order derivative of $\partial_tv_T^h(t,\cdot)$. The proof follows from using the backward Kolmogorov equation~\eqref{eq:Kolmogorov_vh} and the upper bounds from Lemma~\ref{lem:v}.

\begin{lemma}\label{lem:vt}
Assume that the objective function $F$ satisfies the strengthened Assumption~\ref{ass:F2}, and that the diffusion coefficient $\sigma$ satisfies the strengthened Assumption~\ref{ass:sigma2}.

Let $\varphi:\R^d\to\R$ be a mapping of class $\mathcal{C}^3$, with bounded second and third order derivatives. Given any $T\in(0,\infty)$, let the mapping $v_T^h:[0,T]\times\R^d\to\R$ be defined by~\eqref{eq:vh}.

For all $t\in[0,T]$, the mapping $\partial_tv_T^h(t,\cdot)$ is of class $\mathcal{C}^1$. Moreover, for any $\lambda\in(0,\mu)$, there exists positive real numbers $h_4\in(0,h_0)$ and $C_{4,\lambda}\in(0,\infty)$, satisfying the condition $C_{4,\lambda}h_4<\lambda$, such that for all $h\in(0,h_4)$ and $T\ge t\ge 0$ one has
\begin{equation}\label{eq:lemvt}
   \underset{y\in\R^d}{\sup}~\underset{k\in\R^d\setminus\{0\}}{\sup}~ \frac{|D(\partial_tv_T^h)(t,y).k|}{(1+\|y-x^\star\|^4)\|k\|}\le C_{4,\lambda}(\|\nabla\varphi(x^\star)\|+\vvvert\varphi\vvvert_2+\vvvert\varphi\vvvert_3)e^{-(\lambda-C_{4,\lambda}h)(T-t)}.
\end{equation}
\end{lemma}
The proof of Lemma~\ref{lem:vt} is postponed to Appendix~\ref{app:lemv}.

Observe that the rate $\mu-C_1h$ appearing in the exponential factors in the right-hand sides of~\eqref{eq:lemv1} can be bounded from below uniformly with respect to $h\in(0,h_1)$: indeed one has $\mu-C_1h\ge \mu-C_1h_1>0$. Similarly, the rates $\lambda-C_{j,\lambda}h$ for $j\in\{2,3,4\}$ , \eqref{eq:lemv2}, \eqref{eq:lemv3} and~\eqref{eq:lemvt} can be bounded from below uniformly with respect to $h\in(0,h_j)$: indeed one has $\lambda-C_{j,\lambda}h\ge \lambda-C_{j,\lambda}h_j>0$. This means that the upper bounds in Lemma~\ref{lem:v} and Lemma~\ref{lem:vt} can be interpreted as being uniform with respect to $h$.

\subsection{Proof of weak error estimates}\label{sec:error2}

We now focus on the second main task of the proof of Theorem~\ref{theo:2}, which consists in proving error estimates for $\E[\varphi(X_N)]-\E[\varphi(\Y^h(Nh))]$, uniformly with respect to $N$.

\begin{propo}\label{propo:schemeSDE}
Assume that the objective function $F$ satisfies the strengthened Assumption~\ref{ass:F2}, and that the diffusion coefficient $\sigma$ satisfies the strengthened Assumption~\ref{ass:sigma2}.

There exist positive real numbers $H\in(0,h_{0})$ and $C\in(0,\infty)$ such that for any mapping $\varphi:\R^d\to\R$ of class $\mathcal{C}^3$ with bounded second and third order derivatives, for any initial value $x_0\in\R^d$, and for any $h\in(0,H)$, one has the weak error estimate
\begin{equation}\label{eq:proposchemeSDE}
\underset{N\in\N^\star}\sup~\big|\E[\varphi(X_N)]-\E[\varphi(\Y^h(Nh))]\big|\le  C_2(\varphi,x_0)h^2,
\end{equation}
where the positive real number $C_2(\varphi,x_0)$ is given by~\eqref{eq:constantC2}.
\end{propo}

In order to prove Proposition~\ref{propo:schemeSDE}, we follow the same approach as in the proof of Proposition~\ref{propo:schemeODE} in Section~\ref{sec:error1}. This approach is sketched at the beginning of Section~\ref{sec:proof2} above, see also Section~\ref{sec:error1} for some comments. Note that in~\eqref{eq:proposchemeSDE} the order of convergence with respect to $h$ is equal to $2$, instead of $1$ in~\eqref{eq:proposchemeODE}. Obtaining an higher order of convergence requires additional arguments in the analysis. In particular, in the proof one considers the solution $v_T^h$ of the backward Kolmorogorov equation~\eqref{eq:Kolmogorov_vh}, associated with the modified stochastic differential equation~\eqref{eq:SDE2}, instead of dealing with the solution $u^0$ of the partial differential equation~\eqref{eq:Kolmogorov_u}. It is worth mentioning that the proof below exploits bounds on derivatives of $v_T^h$ which are uniform with respect to $h$, given by Lemma~\ref{lem:v} and Lemma~\ref{lem:vt}.

\begin{proof}[Proof of Proposition~\ref{propo:schemeSDE}]
Let the function $\varphi:\R^d\to\R$ and the initial value $x_0$ be given. Combining the expressions~\eqref{eq:decomperror21} and~\eqref{eq:decomperror21delta} for the weak error $\E[\varphi(X_N)]-\E[\varphi(\Y^h(Nh))]$, it is sufficient to prove the following claim for the error terms $\delta_n$ defined by~\eqref{eq:delta}: for all $N\in\N^\star$ and $h\in(0,H)$, for all $n\in\{0,\ldots,N-1\}$, one has
\begin{equation}\label{eq:delta-claim}
|\delta_{n}|\le C_2(\varphi,x_0) h^2 \int_{t_n}^{t_{n+1}}e^{-\frac{\mu(t_N-t)}{4}}\,dt,
\end{equation}
with $C_2(\varphi,x_0)$ defined by~\eqref{eq:constantC2} (up to a modification of the multiplicative constant $C\in(0,\infty)$).

Let us prove the claim~\eqref{eq:delta-claim} above. For all $N\in\N^\star$ and $n\in\{0,\ldots,N-1\}$, recall that $t\in[t_n,t_{n+1}]\mapsto \tilde{X}^{h}(t)$ is solution of the stochastic differential equation~\eqref{eq:tildeX-SDE}. Therefore, applying It\^o's formula on the interval $[t_n,t_{n+1}]$ yields the following decomposition of the error term $\delta_n$: one has
\begin{align*}
    \delta_n&=\int_{t_n}^{t_{n+1}}\E[\partial_tv_T^h(t,\Tilde{X}^h(t)]\,dt\\
    &-\int_{t_n}^{t_{n+1}}\E[\langle \nabla v_T^h(t,\Tilde{X}^h(t)),\nabla F(\Tilde{X}^h(t_n))\rangle]\,dt\\
    &+\frac{h}{2}\int_{t_n}^{t_{n+1}}\E[\nabla^2v_T^h(t,\Tilde{X}^h(t)):a(t_n,\Tilde{X}^h(t_n))]\,dt.
\end{align*}
Recall that the function $v_T^h$ is the solution of the backward Kolmogorov equation~\eqref{eq:vh}. This allows us to decomposed $\delta_n$ as
\begin{equation}\label{eq:decomp_delta_n}
\delta_n= \delta_{n,1}+\delta_{n,2},
\end{equation}
where for all $n \in \{0, \dots, N-1\}$ one defines
\begin{align*}
    \delta_{n,1}&= \int_{t_n}^{t_{n+1}}\E[\langle \nabla v_T^h(t,\Tilde{X}^h(t)), \nabla F^h(\Tilde{X}^h(t))-\nabla F(\Tilde{X}^h(t_n))\rangle]\,dt,\\
    \delta_{n,2}&=\frac{h}{2}\int_{t_n}^{t_{n+1}}\E[\nabla^2v_T^h(t,\Tilde{X}^h(t)):(a(t_n,\Tilde{X}^h(t_n))-a(t,\Tilde{X}^h(t)))]\,dt.
\end{align*}
The analysis of the error term $\delta_{n,1}$ is more subtle than the analysis of the error term $\epsilon_{n,1}$ in the proof of Proposition~\ref{propo:schemeODE}, since for $\delta_{n,1}$ one needs to obtain second order convergence with respect to $h$. Let us decompose the error term $\delta_{n,1}$ as
\begin{equation}\label{eq:decomp_delta_n1}
\delta_{n,1}=\delta_{n,1,1}+\delta_{n,1,2},
\end{equation}
where the auxiliary error terms $\delta_{n,1,1}$ and $\delta_{n,1,2}$ are defined for all $n\in\{0,\ldots,N-1\}$ as
\begin{align*}
    \delta_{n,1,1}&= \int_{t_n}^{t_{n+1}}\E[\langle\nabla v_T^h(t,\Tilde{X}^h(t))-\nabla v_T^h(t_n,\Tilde{X}^h(t_n)),\nabla F^h(\Tilde{X}^h(t))-\nabla F(\Tilde{X}^h(t_n))\rangle]\,dt,\\
    \\
    \delta_{n,1,2}&=\int_{t_n}^{t_{n+1}}\E[\langle\nabla v^h_T(t_n,\Tilde{X}^h(t_n)),\nabla F^h(\Tilde{X}^h(t))-\nabla F(\Tilde{X}^h(t_n))\rangle]\,dt.
\end{align*}

The claim~\eqref{eq:delta-claim} follows from the bounds for $\delta_{n,1,1}$, $\delta_{n,1,2}$ and $\delta_{n,2}$ which are obtained below.

\emph{Upper bounds for $\delta_{n,1,1}$.}

On the one hand, one has
\begin{align*}
\nabla v_T^h(t,\Tilde{X}^h(t))-\nabla v_T^h(t_n,\Tilde{X}^h(t_n))&=\nabla v_T^h(t,\Tilde{X}^h(t))-\nabla v_T^h(t_n,\Tilde{X}^h(t))\\
&+\nabla v_T^h(t_n,\Tilde{X}^h(t))-\nabla v_T^h(t_n,\Tilde{X}^h(t_n)).
\end{align*}
On the other hand, by the definition~\eqref{eq:Fh} of the modified objective function $F^h$, one has
\[
\nabla F^h(\Tilde{X}^h(t))-\nabla F(\Tilde{X}^h(t_n))=\nabla F(\Tilde{X}^h(t))-\nabla F(\Tilde{X}^h(t_n))+\frac{h}{4}\nabla(\|\nabla F\|^2)(\tilde{X}^h(t)),
\]
and recall that $\nabla F$ is globally Lipschitz continuous (see the condition~\eqref{eq:assF1-Lip} from Assumption~\ref{ass:F1}) and that $\nabla(\|\nabla F\|^2)$ safisties the inequality~\eqref{eq:nablaFh}.

Using the inequality~\eqref{eq:lemvt} from Lemma~\ref{lem:vt} on the first order derivative of $\partial_tv_T^h(t,\cdot)$ and the inequality~\eqref{eq:lemv2} from Lemma~\ref{lem:v} on the second order derivative of $v_T^h(t,\cdot)$ with $\lambda=\frac{\mu}{2}$, for all $h\in(0,h_4)$ one has
\begin{align*}
|\delta_{n,1,1}|&\le C(\varphi)\int_{t_n}^{t_{n+1}}e^{-(\frac{\mu}{2}-Ch)(t_N-t)}(t-t_n)\E[\bigl(1+\|\tilde{X}^h(t)-x^\star\|\bigr)^4\|\tilde{X}^h(t)-\tilde{X}^h(t_n)\|]\,dt\\
&+C(\varphi)h\int_{t_n}^{t_{n+1}}e^{-(\frac{\mu}{2}-Ch)(t_N-t)}(t-t_n)\E[\bigl(1+\|\tilde{X}^h(t)-x^\star\|\bigr)^5]\,dt\\
&+C(\varphi)\int_{t_n}^{t_{n+1}}e^{-(\frac{\mu}{2}-Ch)(t_N-t)}\E[\bigl(1+\|\tilde{X}^h(t_n)-x^\star\|+\|\tilde{X}^h(t)-x^\star\|\bigr)^2\|\tilde{X}^h(t)-\tilde{X}^h(t_n)\|^2]\,dt\\
&+C(\varphi)h\int_{t_n}^{t_{n+1}}e^{-(\frac{\mu}{2}-Ch)(t_N-t)}\E[\bigl(1+\|\tilde{X}^h(t_n)-x^\star\|+\|\tilde{X}^h(t)-x^\star\|\bigr)^3\|\tilde{X}^h(t)-\tilde{X}^h(t_n)\|]\,dt,
\end{align*}
where $C(\varphi)=C(\|\nabla\varphi(x^\star)\|+\vvvert\varphi\vvvert_2+\vvvert\varphi\vvvert_3)$ for some positive real number $C\in(0,\infty)$. Moreover, choose $H\in(0,h_4)$ such that $\frac{\mu}{2}-Ch\ge \frac{\mu}{4}$ for all $h\in(0,H)$.

Applying the Cauchy--Schwarz inequality, the moment bounds~\eqref{eq:tildeX-momentbounds} and the increment bounds~\eqref{eq:tildeX-increments} from Lemma~\ref{lem:tildeX}, if $H$ is chosen such that $H\le H_4$ where $H_4$ is given by Lemma~\ref{lem:tildeX}, one obtains the upper bound
\begin{equation}\label{eq:delta_n11}
|\delta_{n,1,1}|\le C_2(\varphi,x_0) h^2 \int_{t_n}^{t_{n+1}}e^{-\frac{\mu(t_N-t)}{4}}\,dt,
\end{equation}
for all $N\in\N^\star$, all $n\in\{0,\ldots,N-1\}$ and all $h\in(0,H)$.

\emph{Upper bounds for $\delta_{n,1,2}$.}

Given the filtration $\bigl(\mathcal{F}_t\bigr)_{t\ge 0}$ introduced in Section~\ref{sec:setting_notation}, for any $n\in\{0,\ldots,N-1\}$, the random variable $\tilde{X}^h(t_n)=X_n$ is $\mathcal{F}_{t_n}$ measurable. Considering the conditional expectation $\E[\cdot|\mathcal{F}_{t_n}]$ then yields the identity
\[
\delta_{n,1,2}=\int_{t_n}^{t_{n+1}}\E[\langle\nabla v^h_T(t_n,\Tilde{X}^h(t_n)),\E[\nabla F^h(\Tilde{X}^h(t))|\mathcal{F}_{t_n}]-\nabla F(\Tilde{X}^h(t_n))\rangle]\,dt.
\]
One thus needs to study the conditional expectation $\E[\nabla F^h(\Tilde{X}^h(t))-\nabla F(\Tilde{X}^h(t_n))|\mathcal{F}_{t_n}]$ for all $t\in[t_n,t_{n+1}]$, which owing to the definition~\eqref{eq:Fh} of the modified objective function $F^h$ can be decomposed as
\begin{align*}
    \E[\nabla F^h(\Tilde{X}^h(t))|\mathcal{F}_{t_n}]-\nabla F(\Tilde{X}^h(t_n))&=
    \E[\nabla F(\Tilde{X}^h(t))|\mathcal{F}_{t_n}]-\nabla F(\Tilde{X}^h(t_n))\\
    &+\frac{h}{4}\E[\nabla(\|\nabla F\|^2)((\Tilde{X}^h(t))|\mathcal{F}_{t_n}]\\
    &=\E[\nabla F(\Tilde{X}^h(t))|\mathcal{F}_{t_n}]-\nabla F(\Tilde{X}^h(t_n))\\
    &+\frac{h}{4}\nabla(\|\nabla F\|^2)(\Tilde{X}^h(t_n))\\
    &+\frac{h}{4}\E[\nabla(\|\nabla F\|^2)(\Tilde{X}^h(t))|\mathcal{F}_{t_n}]-\frac{h}{4}\nabla(\|\nabla F\|^2)(\Tilde{X}^h(t_n)).
\end{align*}
Let us consider the first term of the right-hand side above. Recall that $t\in[t_n,t_{n+1}]\mapsto \tilde{X}^{h}(t)$ is solution of the stochastic differential equation~\eqref{eq:tildeX-SDE}. Therefore, for all $t\in[t_n,t_{n+1}]$ applying It\^o's formula on the interval $[t_n,t]$ gives
\begin{align*}
\E[\nabla F(\Tilde{X}^h(t))|\mathcal{F}_{t_n}]&-\nabla F(\Tilde{X}^h(t_n))=-\E[\int_{t_n}^{t}\nabla^2F(\Tilde{X}^h(s))\nabla F(X_n)\,ds|\mathcal{F}_{t_n}]\\
&+\frac{h}{2}\sum_{k,\ell=1}^{d}\E[\int_{t_n}^{t}D^3 F(\Tilde{X}^h(s)).\bigl(e_\ell,\sigma(t_n,X_n)e_k,\sigma(t_n,X_n)e_k\bigr)e_\ell \,ds|\mathcal{F}_{t_n}].
\end{align*}
Moreover, the first term may be written as
\begin{align*}
-\E[\int_{t_n}^{t}(\nabla^2F(\Tilde{X}^h(s))\nabla F(X_n)\,ds|\mathcal{F}_{t_n}]&=-\E[\int_{t_n}^{t}\bigl(\nabla^2F(\Tilde{X}^h(s))-\nabla^2F(\Tilde{X}^h(t_n))\bigr)\nabla F(X_n)\,ds|\mathcal{F}_{t_n}]\\
&-(t-t_n)\nabla^2F(\Tilde{X}^h(t_n))\nabla F(X_n).
\end{align*}
Observe that one has the identity
\[
\nabla^2F(\Tilde{X}^h(t_n))\nabla F(X_n)=\nabla^2F(\Tilde{X}^h(t_n))\nabla F(\Tilde{X}^h(t_n))=\frac12\nabla(\|\nabla F\|^2)(\Tilde{X}^h(t_n)).
\]
As a result, using the tower property of conditional expectation, the error term $\delta_{n,1,2}$ is decomposed as
\begin{equation}\label{eq:decomp_delta_n12}
\delta_{n,1,2}=\delta_{n,1,2,1}+\delta_{n,1,2,2}+\delta_{n,1,2,3}+\delta_{n,1,2,4},
\end{equation}
with the auxiliary error terms defined by
\begin{align*}
\delta_{n,1,2,1}&=\int_{t_n}^{t_{n+1}}\Bigl(\frac{h}{4}-\frac{(t-t_n)}{2}\Bigr)\,dt \E[\langle\nabla v^h_T(t_n,\Tilde{X}^h(t_n)),\nabla(\|\nabla F\|^2)(\Tilde{X}^h(t_n))\rangle],\\
\delta_{n,1,2,2}&=-\int_{t_n}^{t_{n+1}}\E[\langle\nabla v^h_T(t_n,\Tilde{X}^h(t_n)),\int_{t_n}^{t}\bigl(\nabla^2F(\Tilde{X}^h(s))-\nabla^2F(\Tilde{X}^h(t_n))\bigr)\nabla F(X_n)ds\rangle]\,dt,\\
\delta_{n,1,2,3}&=\frac{h}{4}\int_{t_n}^{t_{n+1}}\E[\langle\nabla v^h_T(t_n,\Tilde{X}^h(t_n)),\nabla(\|\nabla F\|^2)(\Tilde{X}^h(t))-\nabla(\|\nabla F\|^2)(\Tilde{X}^h(t_n))\rangle]\,dt,\\
\delta_{n,1,2,4}&=\frac{h}{2}\sum_{k,\ell=1}^{d}\int_{t_n}^{t_{n+1}}\int_{t_n}^{t}\E[D^3 F(\Tilde{X}^h(s)).\bigl(\nabla v^h_T(t_n,\Tilde{X}^h(t_n)),\sigma(t_n,X_n)e_k,\sigma(t_n,X_n)e_k\bigr)] \,ds \,dt.
\end{align*}
The key observation is that the first auxiliary error term $\delta_{n,1,2,1}$ vanishes, indeed it is straightforward to check that
\[
\int_{t_n}^{t_{n+1}}\frac{(t-t_n)}{2}\,dt=\frac{(t_{n+1}-t_n)^2}{4}=\frac{h^2}{4}=\int_{t_n}^{t_{n+1}}\frac{h}{4}\,dt.
\]
The other auxiliary error terms appearing in the right-hand side of the decomposition~\eqref{eq:decomp_delta_n12} of $\delta_{n,1,2}$ can then be treated as follows.

First, recall that the third order derivative of the objective function $F$ is bounded, owing to the basic Assumption~\ref{ass:F1}, and $\nabla F$ has at most linear growth, see~\eqref{eq:assF1-growth}. As a result, using the inequality~\eqref{eq:lemv1} from Lemma~\ref{lem:v} on the first order derivative of $v_T^h(t,\cdot)$, for all $h\in(0,h_1)$ one has
\[
|\delta_{n,1,2,2}|\le C(\varphi)\int_{t_n}^{t_{n+1}}\int_{t_n}^{t_{n+1}}e^{-(\mu-Ch)(t_N-t)}\E[(1+\|\tilde{X}^h(t_n)-x^\star\|)\|\tilde{X}^h(s)-\tilde{X}^h(t_n)\|\|X_n-x^\star\|]\,ds\,dt,
\]
where $C(\varphi)=C(\|\nabla\varphi(x^\star)\|+\vvvert\varphi\vvvert_2)$ for some positive real number $C\in(0,\infty)$. Moreover, choose $H\in(0,h_1)$ such that $\mu-Ch\ge \frac{\mu}{4}$ for all $h\in(0,H)$.

Applying the Cauchy--Schwarz inequality, the moment bounds~\eqref{eq:tildeX-momentbounds} and the increment bounds~\eqref{eq:tildeX-increments} from Lemma~\ref{lem:tildeX}, if $H$ is chosen such that $H\le H_2$ where $H_2$ is given by Lemma~\ref{lem:tildeX}, one obtains the upper bound
\begin{equation}\label{eq:delta_n122}
|\delta_{n,1,2,2}|\le C_2(\varphi,x_0) h^2 \int_{t_n}^{t_{n+1}}e^{-\frac{\mu(t_N-t)}{4}}\,dt,
\end{equation}
for all $N\in\N^\star$, all $n\in\{0,\ldots,N-1\}$ and all $h\in(0,H)$.

Second, recall that the mapping $\nabla(\|\nabla F\|^2)$ satisfies the local Lipschitz continuity property~\eqref{eq:assF2-LocLip}. As a result, using the inequality~\eqref{eq:lemv1} from Lemma~\ref{lem:v} on the first order derivative of $v_T^h(t,\cdot)$, for all $h\in(0,h_1)$ one has
\[
|\delta_{n,1,2,3}|\le C(\varphi)h\int_{t_n}^{t_{n+1}}e^{-(\mu-Ch)(t_N-t_n)}\E[(1+\|\tilde{X}^h(t_n)-x^\star\|+\|\tilde{X}^h(t)-x^\star\|)^2\|\tilde{X}^h(s)-\tilde{X}^h(t_n)\|]\,dt,
\]
where $C(\varphi)$ is defined as above. Moreover, choose $H\in(0,h_1)$ such that $\mu-Ch\ge \frac{\mu}{4}$ for all $h\in(0,H)$.

Applying the Cauchy--Schwarz inequality, the moment bounds~\eqref{eq:tildeX-momentbounds} and the increment bounds~\eqref{eq:tildeX-increments} from Lemma~\ref{lem:tildeX}, if $H$ is chosen such that $H\le H_2$ where $H_2$ is given by Lemma~\ref{lem:tildeX}, one obtains the upper bound
\begin{equation}\label{eq:delta_n123}
|\delta_{n,1,2,3}|\le C_2(\varphi,x_0) h^2 \int_{t_n}^{t_{n+1}}e^{-\frac{\mu(t_N-t)}{4}}\,dt,
\end{equation}
for all $N\in\N^\star$, all $n\in\{0,\ldots,N-1\}$ and all $h\in(0,H)$.

Finally, recall that the third order derivative of the objective function $F$ is bounded, owing to the basic Assumption~\ref{ass:F1}, and the diffusion coefficient $\sigma$ has at most linear growth, see~\eqref{eq:asssigma1growth} from the basic Assumption~\ref{ass:sigma1}. As a result, using the inequality~\eqref{eq:lemv1} from Lemma~\ref{lem:v} on the first order derivative of $v_T^h(t,\cdot)$, for all $h\in(0,h_1)$ one has
\[
|\delta_{n,1,2,4}|\le C(\varphi)h^2\int_{t_n}^{t_{n+1}}e^{-(\mu-Ch)(t_N-t_n)}\E[(1+\|X_n-x^\star\|)^3]\,dt,
\]
where $C(\varphi)$ is defined as above. Moreover, choose $H\in(0,h_1)$ such that $\mu-Ch\ge \frac{\mu}{4}$ for all $h\in(0,H)$.

Applying the moment bounds~\eqref{eq:momentbounds-scheme} from Lemma~\ref{lem:momentbounds-scheme}, if $H$ is chosen such that $H\le H_2$ where $H_2$ is given by Lemma~\ref{lem:tildeX}, one obtains the upper bound
\begin{equation}\label{eq:delta_n124}
|\delta_{n,1,2,4}|\le C_2(\varphi,x_0) h^2 \int_{t_n}^{t_{n+1}}e^{-\frac{\mu(t_N-t)}{4}}\,dt,
\end{equation}
for all $N\in\N^\star$, all $n\in\{0,\ldots,N-1\}$ and all $h\in(0,H)$.

Combining the upper bounds~\eqref{eq:delta_n122}, \eqref{eq:delta_n123} and~\eqref{eq:delta_n124} on the error terms $\delta_{n,1,2,2}$, $\delta_{n,1,2,3}$ and $\delta_{n,1,2,4}$, and recalling that the error term $\delta_{n,1,2,1}=0$ vanishes, owing to the decomposition~\eqref{eq:decomp_delta_n12} one obtains the following upper bound for the error term $\delta_{n,1,2}$: one has 
\begin{equation}\label{eq:delta_n12}
|\delta_{n,1,2}|\le C_2(\varphi,x_0) h^2 \int_{t_n}^{t_{n+1}}e^{-\frac{\mu(t_N-t)}{4}}\,dt,
\end{equation}
for all $N\in\N^\star$, all $n\in\{0,\ldots,N-1\}$ and all $h\in(0,H)$.

\emph{Upper bounds for $\delta_{n,1}$.}

Combining the upper bounds~\eqref{eq:delta_n11} and~\eqref{eq:delta_n12} on the error terms $\delta_{n,1,1}$ and $\delta_{n,1,2}$, owing to the decomposition~\eqref{eq:decomp_delta_n1} of the error term $\delta_{n,1}$ one obtains the following upper bound for the error term $\delta_{n,1}$: one has 
\begin{equation}\label{eq:delta_n1}
|\delta_{n,1}|\le C_2(\varphi,x_0) h^2 \int_{t_n}^{t_{n+1}}e^{-\frac{\mu(t_N-t)}{4}}\,dt,
\end{equation}
for all $N\in\N^\star$, all $n\in\{0,\ldots,N-1\}$ and all $h\in(0,H)$.

\emph{Upper bounds for $\delta_{n,2}$.}

Let $e_1,\ldots,e_d$ denote an arbitrary orthonormal system of $\R^d$. The error term $\delta_{n,2}$ is written as
\begin{align*}
\delta_{n,2}&=\frac{h}{2}\int_{t_n}^{t_{n+1}}\E[\sum_{j=1}^{d}D^2v_T^h(t,\Tilde{X}^h(t)).\bigl(\sigma(t_n,\Tilde{X}^h(t_n))e_j,\sigma(t_n,\Tilde{X}^h(t_n))e_j\bigr)]\,dt\\
&-\frac{h}{2}\int_{t_n}^{t_{n+1}}\E[\sum_{j=1}^{d}D^2v_T^h(t,\Tilde{X}^h(t)).\bigl(\sigma(t,\Tilde{X}^h(t))e_j,\sigma(t,\Tilde{X}^h(t))e_j\bigr)]\,dt\\
&=\frac{h}{2}\int_{t_n}^{t_{n+1}}\E[\sum_{j=1}^{d}D^2v_T^h(t,\Tilde{X}^h(t)).\bigl((\sigma(t_n,\Tilde{X}^h(t_n))-\sigma(t,\Tilde{X}^h(t)))e_j,(\sigma(t_n,\Tilde{X}^h(t_n))+\sigma(t,\Tilde{X}^h(t)))e_j\bigr)]\,dt.
\end{align*}
Observe that one has
\[
\sigma(t_n,\Tilde{X}^h(t_n))-\sigma(t,\Tilde{X}^h(t))=\sigma(t_n,\Tilde{X}^h(t_n))-\sigma(t,\Tilde{X}^h(t_n))+\sigma(t,\Tilde{X}^h(t_n))-\sigma(t,\Tilde{X}^h(t)).
\]
Recall, from Assumption~\ref{ass:sigma1} on the diffusion coefficient, that for all $x\in\R^d$ the mapping $\sigma(\cdot,x)$ is globally Lipschitz continuous, and that for all $t\in[0,\infty)$ the mapping $\sigma(t,\cdot)$ is globally Lipschitz continuous. Moreover, recall the linear growth property~\eqref{eq:asssigma1growth} for $\sigma$.

Using the inequality~\eqref{eq:lemv2} from Lemma~\ref{lem:v} on the second order derivative of $v_T^h(t,\cdot)$ with $\lambda=\frac{\mu}{2}$, for all $h\in(0,h_2)$ one has
\begin{align*}
|\delta_{n,2}|&\le C(\varphi)h\int_{t_n}^{t_{n+1}}e^{-(\frac{\mu}{2}-Ch)(t_N-t)}(t-t_n)\E[G_1(\tilde{X}^h(t_n),\tilde{X}^h(t))]\,dt\\
&+C(\varphi)h\int_{t_n}^{t_{n+1}}e^{-(\frac{\mu}{2}-Ch)(t_N-t)}\E[G_2(\tilde{X}^h(t_n),\tilde{X}^h(t))]\,dt,
\end{align*}
with auxiliary functions $G_1$ and $G_2$ defined by
\begin{align*}
G_1(x_1,x_2)&=(1+\|x_2-x^\star\|^2)(1+\|x_1-x^\star\|+\|x_2-x^\star\|),\\
G_2(x_1,x_2)&=(1+\|x_2-x^\star\|^2)\|x_2-x_1\|(1+\|x_1-x^\star\|+\|x_2-x^\star\|),
\end{align*}
where $C(\varphi)=C(\|\nabla\varphi(x^\star)\|+\vvvert\varphi\vvvert_2)$ for some positive real number $C\in(0,\infty)$. Moreover, choose $H\in(0,h_2)$ such that $\frac{\mu}{2}-Ch\ge \frac{\mu}{4}$ for all $h\in(0,H)$.

Applying the Cauchy--Schwarz inequality, the moment bounds~\eqref{eq:tildeX-momentbounds} and the increment bounds~\eqref{eq:tildeX-increments} from Lemma~\ref{lem:tildeX}, if $H$ is chosen such that $H\le H_2$ where $H_2$ is given by Lemma~\ref{lem:tildeX}, one obtains the following upper bound for the error term $\delta_{n,2}$: one has
\begin{equation}\label{eq:delta_n2}
|\delta_{n,2}|\le C_2(\varphi,x_0) h^2 \int_{t_n}^{t_{n+1}}e^{-\frac{\mu(t_N-t)}{4}}\,dt,
\end{equation}
for all $N\in\N^\star$, all $n\in\{0,\ldots,N-1\}$ and all $h\in(0,H)$.

\emph{Conclusion.}
Gathering the upper bounds~\eqref{eq:delta_n1} and~\eqref{eq:delta_n2} on the error terms $|\delta_{n,1}|$ and $|\delta_{n,2}|$ and recalling the decomposition~\eqref{eq:decomp_delta_n} of $\delta_n$ then yields the claim~\eqref{eq:delta-claim}.

The proof of the weak error estimate~\eqref{eq:proposchemeSDE} is then straightforward: for all $N\in\N^\star$ one has
\begin{align*}
\big|\E[\varphi(X_N)]-\E[\varphi(\Y^h(Nh))]\big|&\le \sum_{n=0}^{N-1}|\delta_n|\\
&\le C_2(\varphi,x_0) h^2 \sum_{n=0}^{N-1} \int_{t_n}^{t_{n+1}}e^{-\frac{\mu(t_N-t)}{4}}\,dt\\
&\le C_2(\varphi,x_0) h^2 \int_{0}^{t_{N}}e^{-\frac{\mu(t_N-t)}{4}}\,dt\le\frac{4C_2(\varphi,x_0)h^2}{\mu}.
\end{align*}
This concludes the proof of Proposition~\ref{propo:schemeSDE}.
\end{proof}

\subsection{Proof of Theorem~\ref{theo:2}}\label{sec:conclusion-proof2}

We are now in a position to provide the proof of Theorem~\ref{theo:2}.

\begin{proof}[Proof of Theorem~\ref{theo:2}]
First, Let us establish the weak error estimates~\eqref{eq:theo2-weak}. Recall that owing to~\eqref{eq:decomperror2} the error term appearing in the left-hand side of~\eqref{eq:theo2-weak} can be decomposed as
\[
\E[\varphi(X_N)]-\varphi(x^\star)=\E[\varphi(\Y^h(Nh))]-\varphi(x^\star)+\E[\varphi(X_N)]-\E[\varphi(\Y^h(Nh))].
\]
On the one hand, using the inequality~\eqref{eq:ineqDvarphi} and applying the upper bound~\eqref{eq:YT} from Proposition~\ref{propo:YT}, one has
\begin{align*}
\big|\E[\varphi(\Y^h(Nh))]-\varphi(x^\star)\big|&\le C\E\bigl[\bigl(\|\nabla\varphi(x^\star)\|+\vvvert\varphi\vvvert_2\|\Y^h(Nh)-x^\star\|\bigr) \|\Y^h(Nh)-x^\star\|\bigr]\\
&\le C e^{-\mu Nh}\left(1+h\rho(Nh)\right)^{\frac12}\bigl(1+\|x_0-x^\star\|\bigr).
\end{align*}
On the other hand, applying the uniform in time weak error estimates~\eqref{eq:proposchemeSDE} from Proposition~\ref{propo:schemeSDE}, one has for all $N\in\N^\star$ and all $h\in(0,H)$
\[
\big|\E[\varphi(X_N)]-\E[\varphi(\Y^h(Nh))]\big|\le C_2(\varphi,x_0)h^2,
\]
where $C_2(\varphi,x_0)$ is given by~\eqref{eq:constantC2}. Gathering the two upper bounds yields the weak error estimates~\eqref{eq:theo2-weak}.

Second, let us establish the error estimates on the expected residual~\eqref{eq:theo2-residual}. Using the decomposition~\eqref{eq:decomperror2} with $\varphi=F$, one has
\[
\E[F(X_N)]-F(x^\star)=\E[F(\Y^h(Nh))]-F(x^\star)+\E[F(X_N)]-\E[F(\Y^h(Nh))].
\]
On the one hand, using the inequality~\eqref{eq:residual_up} and applying the upper bound~\eqref{eq:YT} from Proposition~\ref{propo:YT}, one has
\[
\E[F(\Y^h(Nh))]-F(x^\star)\le C\E[\|\Y^h(Nh)-x^\star\|^2]\le C\bigl(1+\|x_0-x^\star\|^2\bigr)e^{-2\mu Nh}\left(1+h\rho(Nh)\right)
\]
On the other hand, applying the uniform in time weak error estimates~\eqref{eq:proposchemeSDE} from Proposition~\ref{propo:schemeSDE} with $\varphi=F$, one has for all $N\in\N$ and all $h\in(0,H)$
\[
\big|\E[F(X_N)]-\E[F(\Y^h(Nh))]\big|\le C\bigl(1+\|x_0-x^\star\|^5\bigr)h^2.
\]
Gathering the two upper bounds yields the error estimates on the expected residual~\eqref{eq:theo2-residual}.

Finally, let us establish the strong error estimates~\eqref{eq:theo2-strong}. Owing to the inequality~\eqref{eq:residual_low}, one obtains
\[
\E[\|X_N-x^\star\|^2]\le\frac{2}{\mu}\bigl(\E[F(X_N)]-F(x^\star)\bigr)
\]
and it then suffices to apply the error estimates on the expected residual~\eqref{eq:theo2-residual} to get~\eqref{eq:theo2-strong}.

The proof of Theorem~\ref{theo:2} is thus completed.

\end{proof}

\begin{appendix}\label{sec:app}
\section{Proof of Lemma~\ref{lem:u}}\label{app:lemu}

The objective of this section is to prove the estimates~\eqref{eq:lemu1} and~\eqref{eq:lemu2} on the first and second order derivatives of the mapping $u^0(t,\cdot)$ defined by~\eqref{eq:u}. Let us start with some preliminary results.

Repeating the proof of the inequality~\eqref{eq:cvX0} with an arbitrary initial value $x\in\R^d$, it is straightforward to obtain the following bound on the solution $\bigl(\X_x^0(t)\bigr)_{t\ge 0}$ of the ordinary differential equation~\eqref{eq:ODE1} with arbitrary initial value $\X_x^0(0)=x\in\R^d$: one has
\begin{equation}\label{eq:boundX0}
\underset{t\ge 0}\sup~\underset{x\in\R^d}\sup~\frac{\|\X_x^0(t)\|}{1+\|x-x^\star\|}<\infty.
\end{equation}

In this section, we let $\varphi:\R^d\to\R$ be a mapping of class $\mathcal{C}^2$ with a bounded second order derivative. Recall that its first order derivative is not bounded but it satisfies the inequality~\eqref{eq:ineqDvarphi}.

To simplify the notation in the proof below, let $b(x)=-\nabla F(x)$ for all $x\in\R^d$. As a result one has $Db(x).k=-\nabla^2F(x)k$ for all $x\in\R^d$ and $k\in\R^d$. Applying the equivalent formulation~\eqref{eq:assF1-muconvex-bis} of the $\mu$-convexity condition~\eqref{eq:assF1-muconvex} from Assumption~\ref{ass:F1}, then one has the inequality
\begin{equation}\label{eq:muconvex_bis-b}
\langle Db(x).k,k\rangle \le -\mu\|k\|^2,\quad \forall~x\in\R^d,k\in\R^d.
\end{equation}

The proof of Lemma~\ref{lem:u} proceeds by proving the inequalities~\eqref{eq:lemu1} and~\eqref{eq:lemu2}.

\begin{proof}[Proof of the inequality~\eqref{eq:lemu1}]
Owing to the definition~\eqref{eq:u} of the mapping $u^0$, for all $t\ge 0$, $x\in\R^d$ and $k\in\R^d$, one has
\[
Du^0(t,x).k=D\varphi(\X_x^0(t)).\eta_k^{0}(t),
\]
where $t\ge 0\mapsto \eta_{k}^{0}(t)\in\R^d$ is the solution of the linear differential equation
\[
\left\lbrace
\begin{aligned}
&\frac{d\eta_{k}^{0}(t)}{dt}=Db(\X_x^0(t)).\eta_{k}^{0}(t),\quad \forall~t\ge 0,\\
&\eta_{k}^{0}(0)=k.
\end{aligned}
\right.
\]
Therefore owing to the inequality~\eqref{eq:ineqDvarphi}, one obtains for all $t\ge 0$
\begin{equation}\label{eq:ineqD1u0}
|Du^0(t,x).k|\le \Bigl(\|\nabla\varphi(x^\star)\|+\vvvert\varphi\vvvert_2\|\X_x^0(t)-x^\star\|\Bigr)\|\eta_k^{0}(t)\|,
\end{equation}
and one thus needs to prove an upper bound for $\|\eta_k^{0}(t)\|$. Employing the inequality~\eqref{eq:muconvex_bis-b} yields, for all $t\ge 0$,
\[
\frac12\frac{d\|\eta_{k}^{0}(t)\|^2}{dt}=\langle \frac{d\eta_{k}^{0}(t)}{dt},\eta_{k}^{0}(t)\rangle=\langle Db(\X_x^0(t)).\eta_{k}^{0}(t),\eta_{k}^{0}(t)\rangle \le -\mu \|\eta_{k}^{0}(t)\|^2,
\]
and applying the Gr\"onwall lemma, one then obtains the upper bound
\begin{equation}\label{eq:boundeta0}
\|\eta_{k}^{0}(t)\|\le e^{-\mu t}\|k\|,\quad \forall~t\ge 0.
\end{equation}
Combining the inequalities~\eqref{eq:boundX0} and~\eqref{eq:boundeta0} with the upper bound~\eqref{eq:ineqD1u0} for $|Du^0(t,x).k|$ above, there exists a positive real number $C\in(0,\infty)$ such that for all $x,k\in\R^d$ and all $t\ge 0$ one has
\[
|Du^0(t,x).k|\le C\bigl(\|\nabla \varphi(x^\star)\|+\vvvert\varphi\vvvert_2\bigr)(1+\|x-x^\star\|)e^{-\mu t}\|k\|.
\]
The proof of the inequality~\eqref{eq:lemu1} is thus completed.
\end{proof}

\begin{proof}[Proof of the inequality~\eqref{eq:lemu2}]
Owing to the definition~\eqref{eq:u} of the mapping $u^0$, for all $t\ge 0$, $x\in\R^d$ and $k_1,k_2\in\R^d$, one has
\[
D^2 u^0(t,x).(k_1,k_2)=D\varphi(\X_x^0(t)).\zeta_{k_1,k_2}^{0}(t)+D^2\varphi(\X_x^0(t)).(\eta_{k_1}^{0}(t),\eta_{k_2}^{0}(t)),
\]
where $t\ge 0\mapsto \zeta_{k_1,k_2}^{0}(t)\in\R^d$ is the solution of the linear differential equation
\[
\left\lbrace
\begin{aligned}
&\frac{d\zeta_{k_1,k_2}^{0}(t)}{dt}=Db(\X_x^0(t)).\zeta_{k_1,k_2}^{0}(t)+D^2b(\X_x^0(t)).(\eta_{k_1}^{0}(t),\eta_{k_2}^{0}(t)),\quad \forall~t\ge 0,\\
&\zeta_{k_1,k_2}^{0}(0)=0.
\end{aligned}
\right.
\]
Since $\varphi$ is of class $\mathcal{C}^2$ with bounded second order derivative, using the inequality~\eqref{eq:ineqDvarphi}, one obtains
\begin{equation}\label{eq:ineqD2u0}
|D^2u^0(t,x).(k_1,k_2)|\le \Bigl(\|\nabla \varphi(x^\star)\|+\vvvert\varphi\vvvert_2\|\X_x^0(t)-x^\star\|\Bigr)\|\zeta_{k_1,k_2}^{0}(t)\|+\vvvert\varphi\vvvert_2\|\eta_{k_1}^{0}(t)\|\|\eta_{k_2}^{0}(t)\|.
\end{equation}
Upper bounds for $\|\X_x^0(t)-x^\star\|$ and $\|\eta_{k_j}^{0}(t)\|$ ($j=1,2$) are given by~\eqref{eq:boundX0} and~\eqref{eq:boundeta0} respectively. It thus remains to prove an upper bound for $\|\zeta_{k_1,k_2}^{0}(t)\|$.

For all $t\ge 0$, one has
\begin{align*}
\frac12\frac{d\|\zeta_{k_1,k_2}^{0}(t)\|^2}{dt}&=\langle \frac{d\zeta_{k_1,k_2}^{0}(t)}{dt},\zeta_{k_1,k_2}^{0}(t)\rangle\\
&=\langle Db(\X_x^0(t)).\zeta_{k_1,k_2}^{0}(t),\zeta_{k_1,k_2}^{0}(t)\rangle+\langle D^2b(\X_x^{0}(t)).(\eta_{k_1}^{0}(t),\eta_{k_2}^{0}(t)),\zeta_{k_1,k_2}^{0}(t)\rangle.
\end{align*}
Using the inequality~\eqref{eq:muconvex_bis-b}, the boundedness of the second order derivative of $b$ (the third order derivative of $F$ is bounded, owing to Assumption~\ref{ass:F1}), and the Cauchy--Schwarz and Young inequalities, there exists a positive real number $C\in(0,\infty)$ such that for all $t\ge 0$ one has
\begin{align*}
\frac12\frac{d\|\zeta_{k_1,k_2}^{0}(t)\|^2}{dt}&\le -\mu\|\zeta_{k_1,k_2}^{0}(t)\|^2+C\|\eta_{k_1}^{0}(t)\|\|\eta_{k_2}^{0}(t)\|\|\zeta_{k_1,k_2}^{0}(t)\|\\
&\le -\frac{\mu}{2}\|\zeta_{k_1,k_2}^{0}(t)\|^2+C\|\eta_{k_1}^{0}(t)\|^2\|\eta_{k_2}^{0}(t)\|^2.
\end{align*}
As a result, applying the Gr\"onwall inequality and using the inequality~\eqref{eq:boundeta0}, for all $t\ge 0$ one obtains
\[
\|\zeta_{k_1,k_2}^{0}(t)\|^2\le C\int_0^t e^{-\mu(t-s)}e^{-4\mu s}\,ds \|k_1\|^2 \|k_2\|^2,
\]
which gives the following inequality: for all $t\ge 0$
\[
\|\zeta_{k_1,k_2}^0(t)\|\le Ce^{-\frac{\mu}{2}t}\|k_1\| \|k_2\|.
\]
Using this inequality and the inequality~\eqref{eq:boundeta0}, one then obtains another upper bound: for all $t\ge 0$ one has
\begin{align*}
\frac12\frac{d\|\zeta_{k_1,k_2}^{0}(t)\|^2}{dt}&\le -\mu\|\zeta_{k_1,k_2}^{0}(t)\|^2+C\|\eta_{k_1}^{0}(t)\|\|\eta_{k_2}^{0}(t)\|\|\zeta_{k_1,k_2}^{0}(t)\|\\
&\le -\mu\|\zeta_{k_1,k_2}^{0}(t)\|^2+C e^{-(2\mu+\frac{\mu}{2})t}\|k_1\|^2\|k_2\|^2.
\end{align*}
Applying again the Gr\"onwall inequality, for all $t\ge 0$ one obtains
\[
\|\zeta_{k_1,k_2}^{0}(t)\|^2\le C\int_0^t e^{-2\mu (t-s)}e^{-2(\mu+\frac{\mu}{2})s}\,ds \|k_1\|^2 \|k_2\|^2=C e^{-2\mu t}\int_0^t e^{-\mu s}\,ds \|k_1\|^2 \|k_2\|^2.
\]
As a result one obtains the following inequality: for all $t\ge 0$
\begin{equation}\label{eq:boundzeta0}
\|\zeta_{k_1,k_2}^0(t)\|\le C e^{-\mu t}\|k_1\| \|k_2\|.
\end{equation}

Combining the inequalities~\eqref{eq:boundX0},~\eqref{eq:boundeta0} and~\eqref{eq:boundzeta0} with the upper bound~\eqref{eq:ineqD2u0} for $|D^2u^0(t,x).(k_1,k_2)|$ above, as a result there exists a positive real number $C\in(0,\infty)$ such that for all $x,k_1,k_2\in\R^d$ one has
\[
|D^2u^0(t,x).(k_1,k_2)|\le C\bigl(\|\nabla \varphi(x^\star)\|+\vvvert\varphi\vvvert_2\bigr)(1+\|x-x^\star\|)e^{-\mu t}\|k_1\|\|k_2\|.
\]
The proof of the inequality~\eqref{eq:lemu2} is thus completed.
\end{proof}

\section{Proof of Lemmas~\ref{lem:v} and~\ref{lem:vt}}\label{app:lemv}

\subsection{Proof of Lemma~\ref{lem:v}}

Like in Section~\ref{app:lemu}, let us start with some preliminary results.

In this proof, let $b^h=-\nabla F^h$. As a result one has $Db^h(y).k=-\nabla^2F^h(y)k$ for all $y\in\R^d$ and $k\in\R^d$. Since the modified objective function $F^h$ is assumed to be $\mu$-convex for all $h\in(0,h_0)$ owing to Assumption~\ref{ass:F2}, one has the inequality
\begin{equation}\label{eq:muconvex_bis-bh}
\langle Db^h(y).k,k\rangle \le -\mu\|k\|^2,\quad \forall~y\in\R^d,k\in\R^d,~h\in(0,h_0).
\end{equation}
Moreover, the second and third order derivatives of $b^h$ have at most polynomial growth, uniformly with respect to $h\in(0,h_0)$: there exists a positive real number $C\in(0,\infty)$ such that for all $h\in(0,h_0)$ and $y\in\R^d$ one has
\begin{align}
&\|D^2b^h(y).(k_1,k_2)\|\le C(1+\|y-x^\star\|)\|k_1\|\|k_2\|,\quad \forall~k_1,k_2\in\R^d,\label{eq:D2bh}\\
&\|D^3b^h(y).(k_1,k_2,k_3)\|\le C(1+\|y-x^\star\|)\|k_1\|\|k_2\|\|k_3\|,\quad \forall~k_1,k_2,k_3\in\R^d.\label{eq:D3bh}
\end{align}

To simplify the notation below, given $k_1,k_2,k_3\in\R^d$, define the set
\begin{equation}\label{eq:ensembleK}
\mathcal{K}(k_1,k_2,k_3)=\{(k_1,k_2,k_3),(k_2,k_3,k_1),(k_3,k_1,k_2)\}.
\end{equation}

The proof below proceeds as the proof of Lemma~\ref{lem:u} in Section~\ref{app:lemu}, with two main differences: one needs to deal with stochastic processes, thus It\^o's formula is employed, and one needs to prove bounds up to the third order derivative of $v^h(t,\cdot)$. In addition, the main effort is devoted to prove bounds which are independent of the parameter $h$ in a range $(0,h_1)$.

\begin{proof}[Proof of the inequality~\eqref{eq:lemv1}]
    Owing to the definition~\eqref{eq:vh} of the mapping $v_T^h$, for all $T\ge t\ge 0$, $y\in \mathbb{R}^d$ and $k\in\R^d$, one has
    \[
    D v_T^h(t,y).k=\E[D\varphi(\Y_y^h(T|t).\eta_{k}^{h}(T|t)],
    \]
    where the process $s\ge t\mapsto\eta_{k}^{h}(s|t)$ is the solution of the linear stochastic differential equation
\begin{equation}\label{eq:etah}
\left\lbrace
\begin{aligned}
&d\eta_{k}^{h }(s|t)=Db^h(\Y_y^h(s|t))\eta_{k}^{h}(s|t)~ds+\sqrt{h}D\sigma(s,\Y_y^h(s|t))\eta_{k}^{h}(s|t)~dB(s),\\
&\eta_{k}^{h}(t|t)=k.
\end{aligned}
\right.
\end{equation}
Using the inequality~\eqref{eq:ineqDvarphi} and the Cauchy--Schwarz inequality, one obtains for all $T\ge t\ge 0$
\begin{equation}\label{eq:ineqD1vh}
|Dv_T^h(t,y).k|\le \Bigl(\|\nabla\varphi(x^\star)\|+\vvvert\varphi\vvvert_2\bigl(\E[\|\Y_y^h(T|t)-x^\star\|^2]\bigr)^{\frac12}\Bigr)\bigl(\E[\|\eta_k^{h}(T|t)\|^2]\bigr)^{\frac12}.
\end{equation}
Moment bounds for $\Y_y^h(T|t)-x^\star$ are given in Lemma~\ref{lem:momentboundsYh}, it thus remains to prove moment bounds for $\eta_k^{h}(T|t)$. Even if it would be sufficient to deal with second order moments to deal with~\eqref{eq:ineqD1vh} above, moment bounds at arbitrary order are proved below since they are instrumental in the analysis of second and third order derivatives of $v_T^h$.

Recall that from Assumption~\ref{ass:sigma1} the mapping $\sigma(s,\cdot)$ is globally Lipschitz continuous on $\R^d$, uniformly with respect to $s\in\R^+$. For all $p\in\N^\star$, applying the It\^o formula to the process $s\ge t\mapsto \|\eta_{k}^{h}(s|t)\|^{2p}$, there exists a positive real number $C_p^{(1)}\in(0,\infty)$ such that for all $s\ge t$ one has
\[
\frac{1}{2p} \frac{d\E[\|\eta_{k}^{h}(s|t)\|^{2p}]}{ds}\le \E[\langle Db^h(\Y_y^{h}(s|t)).\eta_{k}^{h}(s|t),\eta_{k}^{h}(s|t)\rangle \|\eta_{k}^{h}(s|t)\|^{2(p-1)}]+C_p^{(1)}h\E[\|\eta_{k}^{h}(s|t)\|^{2p}].
\]
Applying the inequality~\eqref{eq:muconvex_bis-bh} which follows from the $\mu$-convexity of $F^h$, one obtains the following upper bound: for all $s\ge t$ one has
\[
\frac{1}{2p}\frac{d\E[\|\eta_{k}^{h}(s|t)\|^{2p}]}{ds}\le -(\mu-C_p^{(1)}h)\E[\|\eta_{k}^{h}(s|t)\|^{2p}].
\]
Let $h_p^{(1)}\in(0,\frac{\mu}{C_p^{(1)}})$, then for all $h\in(0,h_p^{(1)}]$ one has $\mu-C_p^{(1)}h\ge \mu-C_p^{(1)}h_p^{(1)}>0$. Applying the Gronw\"all inequality and recalling that $\eta_{k}^{h}(t|t)=k$ one obtains the following inequality: for all $s\ge t$ and all $h\in(0,h_p^{(1)}]$ one has
\begin{equation}\label{eq:boundetah}
\E[\|\eta_{k}^{h}(s|t)\|^{2p}]\le e^{-2p(\mu-C_p^{(1)}h)(s-t)}\|k\|^{2p}\le e^{-2p(\mu-C_p^{(1)}h_p^{(1)})(s-t)}\|k\|^{2p}.
\end{equation}
It thus suffices to combine the moment bounds~\eqref{eq:momentboundsYh} and ~\eqref{eq:boundetah} with $p=1$, with the inequality~\eqref{eq:ineqD1vh}, to obtain the inequality~\eqref{eq:lemv1}: there exist positive real number $h_1\in(0,h_0)$ and $C_1\in(0,\infty)$ such that for all $y,k\in\R^d$, all $T\ge t\ge 0$, and all $h\in(0,h_1)$ one has
\[
|Dv_T^h(t,y).k|\le C_1\bigl(\|\nabla \varphi(x^\star)\|+\vvvert\varphi\vvvert_2\bigr)(1+\|y-x^\star\|)e^{-(\mu-C_1h_1)(T-t)}\|k\|.
\]
The proof of the inequality~\eqref{eq:lemv1} is thus completed.
\end{proof}

\begin{proof}[Proof of the inequality~\eqref{eq:lemv2}]
Owing to the definition~\eqref{eq:vh} of the mapping $v_T^h$, for all $T \geq t$, $y \in \mathbb{R}^d$, and $k_1,k_2\in\R^d$, one has
\[
D^2v_T^h(t,x).(k_1,k_2)=\E[D\varphi(\Y_y^h(T|t)).\zeta_{k_1,k_2}^h(T|t)]+\E[D^2\varphi(\Y_y^h(T|t)).(\eta_{k_1}^{h}(T|t),\eta_{k_2}^{h}(T|t))],
\]
where the processes $s\ge t\mapsto \eta_{k_1}^{h}(T|t)$ and $s\ge t\mapsto \eta_{k_2}^{h}(T|t)$ are given by~\eqref{eq:etah} with $k=k_1$ and $k=k_2$ respectively, and where the process $s\ge t\mapsto\zeta_{k_1,k_2}^h(s|t)$ is the solution of the stochastic differential equation
\begin{equation}\label{eq:zetah}
\left\lbrace
\begin{aligned}
&d\zeta_{k_1,k_2}^{h}(s|t)=Db^h(\Y_y^h(s|t).\zeta_{k_1,k_2}^{h}(s)~ds+\sqrt{h}D\sigma(s,\Y_y^h(s|t)).\zeta_{k_1,k_2}^{h}(s|t)~dB(s)\\
&+D^2b^h(\Y_y^h(s|t)).(\eta_{k_1}^{h}(s|t),\eta_{k_2}^{h}(s|t))~ds+\sqrt{h}D^2\sigma(s,\Y_y^h(s|t)).(\eta_{k_1}^{h}(s|t),\eta_{k_2}^{h}(s|t))~dB(s),\\
&\zeta_{k_1,k_2}^{h}(t|t)=0.
\end{aligned}
\right.
\end{equation}
Using the inequality~\eqref{eq:ineqDvarphi} and the Cauchy--Schwarz inequality, one obtains for all $T\ge t\ge 0$
\begin{equation}\label{eq:ineqD2vh}
\begin{aligned}
|D^2v_T^h(t,y).(k_1,k_2)|&\le \Bigl(\|\nabla\varphi(x^\star)\|+\vvvert\varphi\vvvert_2\bigl(\E[\|\Y_y^h(T|t)-x^\star\|^2]\bigr)^{\frac12}\Bigr)\bigl(\E[\|\zeta_{k_1,k_2}^{h}(T|t)\|^2]\bigr)^{\frac12}\\
&+\vvvert\varphi\vvvert_2\bigl(\E[\|\eta_{k_1}^{h}(T|t)\|^2]\bigr)^{\frac12}\bigl(\E[\|\eta_{k_2}^{h}(T|t)\|^2]\bigr)^{\frac12}.
\end{aligned}
\end{equation}
Moment bounds for $\Y_y^h(T|t)-x^\star$ are given by Lemma~\ref{lem:momentboundsYh}, whereas moment bounds for $\eta_{k_1}^{h}(T|t)$ and $\eta_{k_2}^{h}(T|t)$ are given by~\eqref{eq:boundetah} above. It thus remains to prove moment bounds for $\zeta_{k_1,k_2}^{h}(T|t)$.

Recall that the first and second order derivatives of the mapping $\sigma(s,\cdot)$ are bounded owing to Assumption~\ref{ass:sigma2}, uniformly with respect to $s\in[0,\infty)$. Moreover, the second order derivative of the auxiliary mapping $b^h$ has at most linear growth, uniformly with respect to $h\in(0,h_0)$, see~\eqref{eq:D2bh}. For all $p\in\N^\star$, applying the It\^o formula to the process $s\ge t\mapsto \|\zeta_{k_1,k_2}^{h}(s|t)\|^{2p}$, there exists a positive real number $C_p\in(0,\infty)$ such that for all $s\ge t$ one has
\begin{align*}
\frac{1}{2p}\frac{d~\E[\|\zeta_{k_1,k_2}^h(s|t)\|^{2p}]}{ds}&\le \E[\langle Db^h(\Y_y^{h}(s|t)).\zeta_{k_1,k_2}^{h}(s|t),\zeta_{k_1,k_2}^{h}(s|t)\rangle \|\zeta_{k_1,k_2}^{h}(s|t)\|^{2(p-1)}]\\
&+C_p h \E[\|\zeta_{k_1,k_2}^{h}(s|t)\|^{2p}]\\
&+C_p\E[(1+\|\Y_y^h(s|t)-x^\star\|)\|\eta_{k_1}^{h}(s|t)\| \|\eta_{k_2}^{h}(s|t)\| \|\zeta_{k_1,k_2}^{h}(s|t)\|^{2p-1}]\\
&+C_p h \E[\|\eta_{k_1}^{h}(s|t)\|^2 \|\eta_{k_2}^{h}(s|t)\|^2 \|\zeta_{k_1,k_2}^{h}(s|t)\|^{2(p-1)}].
\end{align*}
Assume that $\lambda\in(0,\mu)$. Using the $\mu$-convexity property~\eqref{eq:muconvex_bis-bh} of the modified objective function $F^h$, one has
\[
\E[\langle Db^h(\Y_y^{h}(s|t)).\zeta_{k_1,k_2}^{h}(s|t),\zeta_{k_1,k_2}^{h}(s|t)\rangle \|\zeta_{k_1,k_2}^{h}(s|t)\|^{2(p-1)}]\le -\mu\E[\|\zeta_{k_1,k_2}^{h}(s|t)\|^{2p}].
\]
In addition, using the H\"older and Young inequalities, one obtains the upper bounds
\begin{align*}
\E[(1&+\|\Y_y^h(s|t)-x^\star\|)\|\eta_{k_1}^{h}(s|t)\| \|\eta_{k_2}^{h}(s|t)\| \|\zeta_{k_1,k_2}^{h}(s|t)\|^{2p-1}]\\
&\le \bigl(\E[(1+\|\Y_y^h(s|t)-x^\star\|)^{2p}\|\eta_{k_1}^{h}(s|t)\|^{2p}\|\eta_{k_2}^{h}(s|t)\|^{2p}]\bigr)^{\frac{1}{2p}}\bigl(\E[\|\zeta_{k_1,k_2}^{h}(s|t)\|^{2p}]\bigr)^{\frac{2p-1}{2p}}\\
&\le C_p\E[(1+\|\Y_y^h(s|t)-x^\star\|)^{2p}\|\eta_{k_1}^{h}(s|t)\|^{2p}\|\eta_{k_2}^{h}(s|t)\|^{2p}]+(\mu-\lambda)\E[\|\zeta_{k_1,k_2}^{h}(s|t)\|^{2p}],
\end{align*}
and
\begin{align*}
\E[\|\eta_{k_1}^{h}(s|t)\|^2\|\eta_{k_2}^{h}(s|t)\|^2 & \|\zeta_{k_1,k_2}^{h}(s|t)\|^{2(p-1)}]\\
&\le \bigl(\E[\|\eta_{k_1}^{h}(s|t)\|^{2p}\|\eta_{k_2}^{h}(s|t)\|^{2p}]\bigr)^{\frac{1}{p}} \bigl(\E[\|\zeta_{k_1,k_2}^{h}(s|t)\|^{2p}]\bigr)^{\frac{p-1}{p}}\\
&\le \frac{1}{p}\E[\|\eta_{k_1}^{h}(s|t)\|^{2p}\|\eta_{k_2}^{h}(s|t)\|^{2p}]+\frac{p-1}{p}\E[\|\zeta_{k_1,k_2}^{h}(s|t)\|^{2p}].
\end{align*}
Combining the upper bounds above, there exists a positive real number $C_p^{(2)}\in(0,\infty)$ such that for all $s\ge t$ one has
\begin{align*}
\frac{1}{2p}\frac{d~\E[\|\zeta_{k_1,k_2}^h(s|t)\|^{2p}]}{ds}&\le -(\lambda-C_p^{(2)}h)\E[\|\zeta_{k_1,k_2}^{h}(s|t)\|^{2p}]\\
&+C_p^{(2)}\E[(1+\|\Y_y^h(s|t)-x^\star\|)^{2p}\|\eta_{k_1}^{h}(s|t)\|^{2p}\|\eta_{k_2}^{h}(s|t)\|^{2p}].
\end{align*}
In addition, applying the H\"older inequality and using the moment bounds~\eqref{eq:momentboundsYh} and~\eqref{eq:boundetah}, for all $s\ge t$ one has
\[
\E[(1+\|\Y_y^h(s|t)-x^\star\|)^{2p}\|\eta_{k_1}^{h}(s|t)\|^{2p}\|\eta_{k_2}^{h}(s|t)\|^{2p}]\le C_p^{(2)}(1+\|y-x^\star\|)^{2p}e^{-4p(\mu-C_{2p}^{(1)}h)(s-t)}\|k_1\|^{2p}\|k_2\|^{2p}.
\]
Without loss of generality, it may be assumed that $C_p^{(2)}\ge C_{2p}^{(1)}$. Let $h_p^{(2)}\in(0,\frac{\lambda}{C_p^{(2)}})$, then for all $h\in(0,h_p^{(2)}]$ one has $\lambda-C_p^{(2)}h\ge \lambda-C_p^{(2)}h_p^{(2)}>0$. Applying the Gr\"onwall inequality and recalling that $\zeta_{k_1,k_2}^h(t|t)=0$, one obtains for all $s\ge t$ and all $h\in(0,h_p^{(2)}]$
\begin{equation}\label{eq:boundzetah}
\E[\|\zeta_{k_1,k_2}^{h}(s|t)\|^{2p}]\le C_{p}^{(2)}e^{-2p(\lambda-C_p^{(2)}h)(s-t)}(1+\|y-x^\star\|)^{2p}\|k_1\|^{2p}\|k_2\|^{2p}.
\end{equation}
It thus suffices to combine the moment bounds~\eqref{eq:momentboundsYh},~\eqref{eq:boundetah} and~\eqref{eq:boundzetah} with $p=1$, with the inequality~\eqref{eq:ineqD2vh}, to obtain the inequality~\eqref{eq:lemv2}: for all $\lambda\in(0,\mu)$, there exist positive real numbers $h_2\in(0,h_1)$ and $C_2\in(0,\infty)$ such that for all $y,k_1,k_2\in\R^d$, all $T\ge t\ge 0$, and all $h\in(0,h_2)$ one has
\[
|D^2v_T^h(t,y).(k_1,k_2)|\le C_2\bigl(\|\nabla \varphi(x^\star)\|+\vvvert\varphi\vvvert_2\bigr)(1+\|y-x^\star\|^2)e^{-(\mu-C_2h_2)(T-t)}\|k_1\|\|k_2\|.
\]
The proof of the inequality~\eqref{eq:lemv2} is thus completed.
\end{proof}

\begin{proof}[Proof of the inequality~\eqref{eq:lemv3}]Owing to the definition~\eqref{eq:vh} of the mapping $v_T^h$, for all $T\ge t\ge 0$, $y\in \mathbb{R}^d$, and $k_1,k_2,k_3\in\R^d$, one has
\begin{align*}
D^3v_T^h(t,x).(k_1,k_2,k_3)&=\E[D\varphi(\Y_y^h(T|t)\xi_{k_1,k_2,k_3}^h(T|t))]\\
&+\sum_{(\ell_1,\ell_2,\ell_3)\in \mathcal{K}}\E[D^2\varphi(\Y_y^h(T|t))(\eta_{\ell_1}^{h}(T|t),\zeta_{\ell_2,\ell_3}^{h}(T|t))]\\
&+\E[D^3\varphi(\Y_y^h(T|t))(\eta_{k_1}^h(T|t),\eta_{k_2}^h(T|t),\eta_{k_3}^h(T|t))],
\end{align*}
where the set $\mathcal{K}=\mathcal{K}(k_1,k_2,k_3)$ is defined by~\eqref{eq:ensembleK}, where the processes $s\ge t\mapsto \eta_{\ell}^h(T|t)$ with $\ell\in\{k_1,k_2,k_3\}$ are given by~\eqref{eq:etah}, where the processes $s\ge t\mapsto \zeta_{\ell_1,\ell_2}^{h}(T|t)$ with indices $(\ell_1,\ell_2)\in\{(k_2,k_3),(k_3,k_1),(k_1,k_2)\}$ are given by~\eqref{eq:zetah}, and where the process $s\ge t\mapsto \xi_{k_1,k_2,k_3}^{h}(s|t)$ is the solution of the stochastic differential equation
\begin{equation}\label{eq:xih}
\left\lbrace
\begin{aligned}
d\xi_{k_1,k_2,k_3}^h(s|t)&=Db^h(\Y_y^h(s|t)).\xi_{k_1,k_2,k_3}^h(s|t)~ds+\sqrt{h}D\sigma(s,\Y_y^h(s|t)).\xi_{k_1,k_2,k_3}^h(s|t)~dB(s)\\
&+\sum_{(\ell_1,\ell_2,\ell_3)\in \mathcal{K}}D^2b^h(\Y_y^h(s|t)).(\eta_{\ell_1}^h(s|t),\zeta_{\ell_2,\ell_3}^h(s|t))~ds\\
&+\sum_{(\ell_1,\ell_2,\ell_3)\in \mathcal{K}}\sqrt{h}D^2\sigma(s,\Y_y^h(s|t)).(\eta_{\ell_1}^h(s|t),\zeta_{\ell_2,\ell_3}^h(s|t))~dB(s)\\
&+D^3b^h(\Y_y^h(s|t)).(\eta_{k_1}^h(s|t),\eta_{k_2}^h(s|t),\eta_{k_3}^h(s|t))~ds\\
&+\sqrt{h}D^3\sigma(s,\Y_y^h(s|t)).(\eta_{k_1}^h(s|t),\eta_{k_2}^h(s|t),\eta_{k_3}^h(s|t))~dB(s),\\
\xi_{k_1,k_2,k_3}^h(t|t)&=0.
\end{aligned}
\right.
\end{equation}
Using the inequality~\eqref{eq:ineqDvarphi} and the Cauchy--Schwarz and H\"older inequality, one obtains for all $s\ge t$
\begin{equation}\label{eq:ineqD3vh}
\begin{aligned}
|D^3v_T^h(t,y).(k_1,k_2,k_3)|&\le \Bigl(\|\nabla\varphi(x^\star)\|+\vvvert\varphi\vvvert_2\bigl(\E[\|\Y_y^h(T|t)-x^\star\|^2]\bigr)^{\frac12}\Bigr)\bigl(\E[\|\xi_{k_1,k_2,k_3}^{h}(T|t)\|^2]\bigr)^{\frac12}\\
&+\vvvert\varphi\vvvert_2\sum_{(\ell_1,\ell_2,\ell_3)\in \mathcal{K}}\bigl(\E[\|\eta_{\ell_1}^{h}(T|t)\|^2]\bigr)^{\frac12}\bigl(\E[\|\zeta_{\ell_2,\ell_3}^{h}(T|t)\|^2]\bigr)^{\frac12}\\
&+\vvvert\varphi\vvvert_3\bigl(\E[\|\eta_{k_1}^{h}(T|t)\|^3]\bigr)^{\frac13}\bigl(\E[\|\eta_{k_2}^{h}(T|t)\|^3]\bigr)^{\frac13}\bigl(\E[\|\eta_{k_3}^{h}(T|t)\|^3]\bigr)^{\frac13}.
\end{aligned}
\end{equation}
Moment bounds for $\Y_y^h(T|t)-x^\star$ are given by Lemma~\ref{lem:momentboundsYh}, moment bounds for $\eta_{k_1}^{h}(T|t)$, $\eta_{k_2}^{h}(T|t)$ and $\eta_{k_3}^{h}(T|t)$ are given by~\eqref{eq:boundetah} above, and moment bounds for $\zeta_{k_1,k_2}^{h}(T|t)$, $\zeta_{k_2,k_3}^{h}(T|t)$ and $\zeta_{k_3,k_1}^{h}(T|t)$ are given by~\eqref{eq:boundzetah} above. It thus remains to prove moment bounds for $\xi_{k_1,k_2,k_3}^{h}(T|t)$. It suffices to consider second order moments, i.e. $p=1$.

Recall that the first, second and third order derivatives of the mapping $\sigma(s,\cdot)$ are bounded owing to Assumption~\ref{ass:sigma2}, uniformly with respect to $s\in(0,\infty)$. Moreover, the second and third order derivatives of the auxiliary mapping $b^h$ have at most linear growth, see~\eqref{eq:D2bh} and~\eqref{eq:D3bh}, uniformly with respect to $h\in(0,h_0)$. Applying the It\^o formula to the process $s\ge t\mapsto \|\xi_{k_1,k_2,k_3}^{h}(s|t)\|^{2}$, there exists a positive real number $C\in(0,\infty)$ such that for all $s\ge t$ one has
\begin{align*}
\frac{1}{2}\frac{d~\E[\|\xi_{k_1,k_2,k_3}^h(s|t)\|^2]}{ds}&\le \E[\langle Db^h(\Y_y^{h}(s|t)).\xi_{k_1,k_2,k_3}^{h}(s|t),\xi_{k_1,k_2,k_3}^{h}(s|t)\rangle]\\
&+C h \E[\|\xi_{k_1,k_2,k_3}^{h}(s|t)\|^{2}]\\
&+C\sum_{(\ell_1,\ell_2,\ell_3)\in \mathcal{K}}\E[(1+\|\Y_y^h(s|t)-x^\star\|)\|\eta_{\ell_1}^{h}(s|t)\| \|\zeta_{\ell_2,\ell_3}^{h}(s|t)\| \|\xi_{k_1,k_2,k_3}^{h}(s|t)\|]\\
&+Ch \sum_{(\ell_1,\ell_2,\ell_3)\in \mathcal{K}}\E[\|\eta_{\ell_1}^{h}(s|t)\|^2 \|\zeta_{\ell_2,\ell_3}^{h}(s|t)\|^2]\\
&+C \E[(1+\|\Y_y^h(s|t)-x^\star\|)\|\eta_{k_1}^{h}(s|t)\| \|\eta_{k_2}^{h}(s|t)\| \|\eta_{k_3}^{h}(s|t)\| \|\xi_{k_1,k_2,k_3}^{h}(s|t)\|]\\
&+C h \E[\|\eta_{k_1}^{h}(s|t)\|^2 \|\eta_{k_2}^{h}(s|t)\|^2 \|\eta_{k_3}^{h}(s|t)\|^2].
\end{align*}
Assume that $\lambda\in(0,\mu)$. Using the $\mu$-convexity property~\eqref{eq:muconvex_bis-bh} of the modified objective function $F^h$, one has
\[
\E[\langle Db^h(\Y_y^{h}(s|t)).\xi_{k_1,k_2,k_3}^{h}(s|t),\xi_{k_1,k_2,k_3}^{h}(s|t)\rangle ]\le -\mu\E[\|\xi_{k_1,k_2,k_3}^{h}(s|t)\|^{2}].
\]
Using the H\"older and Young inequalities, there exists a positive real number $C_1^{(3)}\in(0,\infty)$ such that for all $s\ge t$ one has
\begin{align*}
\frac{1}{2}\frac{d~\E[\|\xi_{k_1,k_2,k_3}^h(s|t)\|^2]}{ds}&\le -(\lambda-C_1^{(3)}h)\E[\|\xi_{k_1,k_2,k_3}^h(s)\|^2]\\
&+C_1^{(3)}\sum_{(\ell_1,\ell_2,\ell_3)\in \mathcal{K}}\E[(1+\|\Y_y^h(s|t)-x^\star\|)^2\|\eta_{\ell_1}^{h}(s|t)\|^2 \|\zeta_{\ell_2,\ell_3}^{h}(s|t)\|^2]\\
&+C_1^{(3)}\E[(1+\|\Y_y^h(s|t)-x^\star\|)^2\|\eta_{k_1}^{h}(s|t)\|^2 \|\eta_{k_2}^{h}(s|t)\|^2 \|\eta_{k_3}^{h}(s|t)\|^2]
\end{align*}
Applying the H\"older inequality and using the moment bounds~\eqref{eq:momentboundsYh},~\eqref{eq:boundetah} and~\eqref{eq:boundzetah}, one thus obtains the upper bound
\begin{align*}
\frac{1}{2}\frac{d~\E[\|\xi_{k_1,k_2,k_3}^h(s|t)\|^2]}{ds}&\le -(\lambda-C_1^{(3)}h)\E[\|\xi_{k_1,k_2,k_3}^h(s)\|^2]\\
&+C_1^{(3)}e^{-2(\mu+\lambda-(C_3^{(1)}+C_3^{(2)})h)(s-t)}(1+\|y-x^\star\|)^2\|k_1\|^2\|k_2\|^2\|k_3\|^2\\
&+C_1^{(3)}e^{-3(\mu-C_4^{(1)}h)(s-t)}(1+\|y-x^\star\|)^2\|k_1\|^2\|k_2\|^2\|k_3\|^2.
\end{align*}
Without loss of generality it may be assumed that $C_1^{(3)}\ge \max(C_3^{(1)},C_3^{(2)},C_4^{(1)})$. Let $h_1^{(3)}\in(0,\frac{\lambda}{C_1^{(3)}})$, then for all $h\in(0,h_1^{(3)}]$ one has $\lambda-C_1^{(3)}h\ge \lambda-C_1^{(3)}h_1^{(3)}>0$. Applying the Gr\"onwall inequality and recalling that $\xi_{k_1,k_2,k_3}^{h}(t|t)=0$, one obtains for all $h\in(0,h_1^{(3)}]$
\begin{equation}\label{eq:boundxih}
\E[\|\xi_{k_1,k_2,k_3}^{h}(s|t)\|^2]\le C_1^{(3)}e^{-2(\lambda-C_1^{(3)}h)(s-t)}(1+\|y-x^\star\|)^2\|k_1\|^2\|k_2\|^2\|k_3\|^2.
\end{equation}
It suffices to combine the moment bounds~\eqref{eq:momentboundsYh},~\eqref{eq:boundetah},~\eqref{eq:boundzetah} and~\eqref{eq:boundxih} with the inequality~\eqref{eq:ineqD3vh} to obtain the inequality~\eqref{eq:lemv3}: for all $\lambda\in(0,\mu)$, there exist positive real numbers $h_3\in(0,h_2)$ and $C_3\in(0,\infty)$ such that for all $y,k_1,k_2\in\R^d$, all $T\ge t\ge 0$, and all $h\in(0,h_3)$ one has
\[
|D^3v_T^h(t,y).(k_1,k_2,k_3)|\le C_2\bigl(\|\nabla \varphi(x^\star)\|+\vvvert\varphi\vvvert_2+\vvvert\varphi\vvvert_3\bigr)(1+\|y-x^\star\|)^2e^{-(\lambda-C_3h_3)(T-t)}\|k_1\|\|k_2\|\|k_3\|.
\]
The proof of the inequality~\eqref{eq:lemv3} is thus completed.
\end{proof}

\subsection{Proof of Lemma~\ref{lem:vt}}

Lemma~\ref{lem:vt} provides upper bounds for the first order derivative $D(\partial_tv_T^h)(t,\cdot)$. To prove them, it suffices to express the temporal derivative $\partial_tv_T^h(t,y)$ using the backward Kolmogorov equation~\eqref{eq:Kolmogorov_vh}, to differentiate with respect to $y$, and to apply the upper bounds on the derivatives $D^jv_T^h(t,\cdot)$ of order $j\in\{1,2,3\}$ stated in Lemma~\ref{lem:v}.

\begin{proof}[Proof of Lemma~\ref{lem:vt}]
    Owing to the backward Kolmogorov equation~\eqref{eq:Kolmogorov_vh}, for all $t \in [0, T]$, $y \in \mathbb{R}^d$, one has
    \[
    \partial_tv_T^h(t,y)=Dv_T^h(t,y).\nabla F^h(y)-\frac{h}{2}\sum_{j=1}^{d}D^2v_T^h(t,y).\bigl(\sigma(t,y)e_j,\sigma(t,y)e_j\bigr),
    \]
    where $e_1,\ldots,e_n$ denotes an arbitrary orthonormal system of $\R^d$.
    
    Differentiating with respect to $y\in\R^d$, for all $k\in\R^d$ one obtains
    \begin{align*}
        D(\partial_tv_T^h)&(t,y).k=D^2 F^h(y).\bigl(\nabla v_T^h(t,y),k\bigr)+D^2 v_T^h(t,y).\bigl(\nabla F^h(y),k\bigr)\\
        &-h \sum_{j=1}^{d}D^2v_T^h(t,y).\bigl((D\sigma(t,y).k)e_j,\sigma(t,y)e_j\bigr) -\frac{h}{2}\sum_{j=1}^{d}D^3v_T^h(t,y).\bigl(k,\sigma(t,y)e_j,\sigma(t,y)e_j\bigr). 
    \end{align*}
    It remains to prove upper bounds for the four terms appearing in the right-hand side above. Since the results of Lemma~\ref{lem:v} are employed, it is assumed that $h\in(0,h_3)$ below.
    
    For the first term, using the local Lipschitz continuity property~\eqref{eq:localLipFh} of $\nabla F^h$ and to the inequality~\eqref{eq:lemv1} on the first order derivative of $v_T^h(t,\cdot)$, one obtains the upper bounds
    \begin{align*}
    |D^2 F^h(y).\bigl(\nabla v_T^h(t,y),k\bigr)|&\le C\bigl(1+\|y-x^\star\|\bigr)\|\nabla v_T^h(t,y)\| \|k\|\\
    &\le C\bigl(\|\nabla \varphi(x^\star)\|+\vvvert\varphi\vvvert_2\bigr)\bigl(1+\|y-x^\star\|^2\bigr)e^{-(\mu-Ch)(T-t)}\|k\|.
    \end{align*}
    For the second term, observe that combining~\eqref{eq:assF1-growth} and the first inequality from~\eqref{eq:nablaFh}, one has $\|\nabla F^h(y)\|\le C\|y-x^\star\|$. As a result, using the inequality~\eqref{eq:lemv2} on the second order derivative of $v_T^h(t,\cdot)$, one obtains the upper bound
    \[
    |D^2 v_T^h(t,y).\bigl(\nabla F^h(y),k\bigr)|\le C\bigl(\|\nabla \varphi(x^\star)\|+\vvvert\varphi\vvvert_2\bigr)e^{-(\lambda-Ch)(T-t)}\bigl(1+\|y-x^\star\|^3\bigr)\|k\|.
    \]
    For the third term, using the linear growth and Lipschitz continuity properties of the diffusion coefficient $\sigma(t,\cdot)$, see~\eqref{eq:asssigma1growth} and~\eqref{eq:asssigma1Lip}, and the inequality~\eqref{eq:lemv2} on the second order derivative of $v_T^h(t,\cdot)$, one obtains the upper bounds
    \begin{align*}
    |\sum_{j=1}^{d}D^2v_T^h(t,y)&.\bigl((D\sigma(t,y).k)e_j,\sigma(t,y)e_j\bigr)|\\
    &\le C\bigl(\|\nabla \varphi(x^\star)\|+\vvvert\varphi\vvvert_2\bigr)e^{-(\lambda-Ch)(T-t)}\bigl(1+\|y-x^\star\|^2\bigr)\|D\sigma(t,y).k\| \|\sigma(t,y)\|\\
    &\le C\|\varsigma\|_\infty^2 \bigl(\|\nabla \varphi(x^\star)\|+\vvvert\varphi\vvvert_2\bigr)e^{-(\lambda-Ch)(T-t)}\bigl(1+\|y-x^\star\|^3\bigr)\|k\|.
    \end{align*}
    For the fourth term, using the linear growth property~\eqref{eq:asssigma1growth} of the diffusion coefficient $\sigma(t,\cdot)$, and the inequality~\eqref{eq:lemv3} on the third order derivative of $v_T^h(t,\cdot)$, one obtains the upper bounds
    \begin{align*}
    |\sum_{j=1}^{d}D^3v_T^h(t,y)&.\bigl(k,\sigma(t,y)e_j,\sigma(t,y)e_j\bigr)|\\
    &\le C\bigl(\|\nabla \varphi(x^\star)\|+\vvvert\varphi\vvvert_2+\vvvert\varphi\vvvert_3\bigr)e^{-(\lambda-Ch)(T-t)}\bigl(1+\|y-x^\star\|^2\bigr)\|\sigma(t,y)\|^2\\
    &\le C\bigl(\|\nabla \varphi(x^\star)\|+\vvvert\varphi\vvvert_2+\vvvert\varphi\vvvert_3\bigr)e^{-(\lambda-Ch)(T-t)}\bigl(1+\|y-x^\star\|^4\bigr)\|k\|.
    \end{align*}
    Combining the four upper bounds then yields the inequality~\eqref{eq:lemvt}, for a sufficiently large positive real number $C_4\in(0,\infty)$ and a sufficiently small positive real number $(0,h_0)$. The proof of Lemma~\ref{lem:vt} is thus completed.
\end{proof}

\end{appendix}

%
%
%
\end{document}